\documentclass[final]{siamltex}
\usepackage{todonotes}

\usepackage{verbatim}
\usepackage[linesnumbered,lined,boxed,commentsnumbered]{algorithm2e}
\usepackage{latexsym, graphicx, epsfig}
\usepackage{pgfplots}
\usepackage[notcite,notref]{showkeys}
\usepackage{showkeys}
\usepackage{url}
\usepackage{color}
\usepackage{array}
\usepackage{amsfonts,amsmath,amssymb}
\usepackage{verbatim}
\usepackage{setspace}
\usepackage{multirow}
\usepackage{slashbox}

\newtheorem{remark}{Remark}[section]

\newcommand{\jrp}[1]{{\color{orange}{#1}}}
\newcommand{\beq}{\begin{equation}}
\newcommand{\eeq}{\end{equation}}
\newcommand{\beqq}{\begin{equation*}}
\newcommand{\eeqq}{\end{equation*}}
\newcommand{\beqas}{\begin{eqnarray*}}
\newcommand{\eeqas}{\end{eqnarray*}}
\newcommand{\bsp}{\begin{split}}
\newcommand{\eesp}{\end{split}}

\renewcommand{\div}{\mathop{\rm div}\nolimits}

\newcommand{\xoe}{{x\over\epsilon}}
\newcommand{\ep}{\epsilon}

\newcommand{\dy}{ \mathrm{d}y}

\usepackage{amscd}
\usepackage{colordvi}
\usepackage{epstopdf}
\usepackage{hyperref}
\usepackage{subcaption}
\usepackage{cleveref}
\def\norm#1{\|#1\|}

\title{ physics-informed neural networks for learning the homogenized coefficients of multiscale elliptic equations 
}
\author{ Jun Sur R Park\thanks{Department of Mathematics, University of Iowa, Iowa City, IA 52246. USA. Email: 
junsur-park@uiowa.edu.}
           \and  Xueyu Zhu\thanks{Department of Mathematics, University of Iowa, Iowa City, IA 52246. USA. Email: xueyu-zhu@uiowa.edu.}}

\begin{document}

\maketitle 

\begin{abstract}
Multiscale elliptic equations with scale separation are often approximated by the corresponding homogenized equations with slowly varying homogenized coefficients (the G-limit).
The traditional homogenization techniques typically rely on the periodicity of the multiscale coefficients, thus finding the G-limits often requires sophisticated techniques in more general settings even when multiscale coefficient is known, if possible. Alternatively, we propose a simple approach to estimate the G-limits from (noisy-free or noisy) multiscale solution data, either from the existing forward multiscale solvers or sensor measurements. 
By casting this problem into an inverse problem, 
our approach adopts physics-informed neural networks (PINNs) algorithm to estimate the G-limits from the multiscale solution data by leveraging a priori knowledge of the underlying homogenized equations. Unlike  the existing approaches,  our approach does not  rely on the periodicity assumption or the known multiscale coefficient during the learning stage, allowing us to estimate homogenized coefficients in more general settings beyond the periodic setting.   We demonstrate that the proposed approach can deliver reasonable and accurate approximations to the G-limits as well as homogenized solutions through several benchmark problems.  

\end{abstract}


\pagestyle{myheadings}
\thispagestyle{plain}

\section{Introduction}
\label{sec:Intro}

A wide range of scientific and engineering problems involve multiple scales due to the heterogeneity of the media properties. Direct numerical simulation for multiscale problems, such as multiscale elliptic equations, is typically computationally demanding due to the finescale fluctuation of the media properties. 
A major effort has been made in past decades to approximate a multiscale equation  by the corresponding homogenized equation, whose coefficient, known as the homoegenized coefficient or G-limit \cite{spagnolo1967sul,spagnolo1976convergence}, does not depend on the fine scale. The resulting solution is referred to as the homogenized solution. 
However, deriving the homogenized equations requires the computation of the G-limits, which is a difficult task for general problems. For standard periodic or locally periodic problems, there are several homogenization methods to find the G-limits, such as two-scale and multiscale convergence \cite{allaire1992homogenization, allaire1996multiscale}, but they can be computationally demanding as they often involve a large number of local problem computations. Additionally, if the periodicity assumption does not hold, the standard homogenization methods are not directly applicable, and non-trivial extensions are usually needed if possible. As a result, deriving the homogenized models from the first principles 
remains challenging for general homogenization problems.

Alternatively, there has been a surge of interest in data-driven learning the effective macroscale model from available measurements or simulated data.  In \cite{you2021data}, coarse-grained nonlocal models are learned from synthetic high-fidelity (multiscale) data by recovering the sign-changing kernels. In  \cite{chen2020physics}, physics-informed neural networks (PINNs) were employed to
retrieve the effective permittivity parameters from scattering data in inverse scattering electromagnetic problems.  In  \cite{arbabi2020linking}, 
a neural network algorithm coupled with an equation-free method has been developed to approximate homogenized solution of a time-dependent multiscale problem using simulated multiscale solution data. Regarding the homogenization on  multiscale  elliptic equations, there have been several inversion approaches related to the homogenization problems in the past several years. The authors in \cite{frederick2014numerical, abdulle2020bayesian, abdulle2020ensemble} recovered multiscale coefficients from (noisy) multiscale solution data using corresponding homogenized models based on numerical homogenization techniques  - the finite element heterogeneous method (FE-HMM) to reduce the computational cost of their forward problems.
A Bayesian estimation has been developed to reconstruct the slowly varying parts of the multiscale coefficients from the noisy measurement of multiscale solution data in \cite{nolen2009fine}.  Nonetheless,
the majority of the existing methods assume that the multiscale coefficients are periodic.
More general multiscale coefficients such as non-periodic coefficients are considered in \cite{gulliksson2016separating}.
The authors separated the oscillations of the multiscale coefficients from the weak $L^2$ limits of them and recovered the part of G-limits from the contributions of the oscillations. 
However, they required the multiscale coefficients to be known  during the inversion stage. In addition, the existing inversion methods often require specialist knowledge, such as numerical multiscale methods or homogenization theory, which can be difficult for  application practitioners. 
These limitations motivate the development of simple and flexible   algorithms for the homogenization of multiscale  elliptic equations with scale separations in more general settings. 

The goal in this paper is to develop a simple and flexible framework to learn the G-limits and corresponding homogenized solutions simultaneously for multiscale elliptic equations, given multiscale solution data. 
Unlike other approaches, 
 our approach does not require    the periodicity of the multiscale coefficient or a known multiscale coefficients during the learning stage.
Instead, we assume that the (simulated or measured) solution data of the multiscale equations are available and the the structure of corresponding homogenized equations are known. 
 We mainly consider the following two possible scenarios:
\begin{center}
\begin{itemize}
\item 
{\it Noise-free data}: In this scenario, we assume that the traditional homogenization methods may not be  applicable, e.g., in  non-periodic cases, but the multiscale solution data can be generated by the exisiting forward solver of the multiscale problem with a known multiscale coefficient. Our goal is to estimate the corresponding G-limit and the homogenized solutions.
\item {\it Noisy data:} In this case, we consider that 
noisy multiscale solution data  
(from a specific medium with a fixed  finescale size $\ep$) can be collected by sensors. 
We aim to learn the G-limit of the unknown multiscale coefficient and corresponding homogenized solution as they can serve as  good approximations to the effective behavior of the multiscale coefficient when $\ep$ is sufficiently small.
\end{itemize}
\end{center}

Specifically, we adopt one emerging scientific machine learning framework  - the physics-informed neural network (PINNs) for our problem. They have been successfully used for approximating solutions to both forward and inverse problems regarding PDEs \cite{lu2019deepxde,chen2020physics,raissi2019physics}. One key component of PINNs is to provide neural network approximations to the solutions of forward or inverse problems by incorporating prior physics knowledge into the loss functional. This feature turns out to be beneficial for our current setting. Since
the multiscale solution data often contains rapid oscillations or noise, estimating the G-limits and homogenized solutions from the multiscale or random fluctuations is a fundamental challenge. 
To address these issues, we trained the neural works to approximate the G-limit and the corresponding homogenized solution for the elliptic homogenized equations based on the multiscale solution data. By incorporating the corresponding homogenized equation into the loss function, PINNs can encourage the neural network to capture the slowly varying parts of the multiscale solution data.

It is worth noting that collecting a large number of the multiscale solution data containing sufficient finescale information is in general difficult, especially when the finescale parameter $\ep$ is very small. In addition, the measurements by sensors are often corrupted by noises that dominate the finescale fluctuations.   
Nonetheless, we found that our approach does not require dense sampling of  the multiscale solution  data in space in order to retain the detailed finescale information as we are only interested in the macroscopic (homogenized) behavior of the multiscale solution data. 
With the prior knowlege of the structure of the  homogenized equation,
PINNs can provide an effective regularization that can cope with the noise and the multiscale features in the data.
We demonstrate the applicability and performance of our approach via several benchmark examples with both noise-free and noisy data. 

The paper is organized as follows. In Section 2, we introduce the concepts of G-convergence and G-limit, and the formulation of the inverse problem. In Section 3, we briefly review the physics-informed neural networks (PINNs) and adopt them in our context.
Finally, we demonstrate the performance of the proposed methods with several numerical examples, including locally periodic,  non-periodic, non-standard, and random homogenization cases.

\pagestyle{myheadings}
\thispagestyle{plain}


\section{Background and Problem Setup}
In this section, we first introduce the definition of G-convergence and G-limit in the homogenized equations, given the multiscale elliptic equations. 
Then we discuss the convergence of homogenization in a special case where the periodic multiscale coefficients are given. 
Finally, we formulate the inverse problem to learn the G-limits and the corresponding homogenized solutions.  
\subsection{G-convergence and G-limit} 
We first briefly review the general theory of homogenization and introduce the notion of the G-convergence and G-limit (homogenized coefficient).
Let us consider a sequence of the following second order multiscale elliptic equations:
\beq
\label{eq:original_gen}
\bsp
-\div \bigg(A^\ep (x) \nabla u^\ep(x) \bigg) &= f(x) \ \ \textrm{in} \ \ \Omega, \\
u^\ep(x) &= 0 \ \ \textrm{on} \ \ \partial \Omega,
\end{split}
\eeq
where $\Omega \in \mathbb{R}^N$ is the domain and $A^\ep: \Omega \to \mathbb{R}^{N\times N}$ is a symmetric multiscale coefficient with finescale size $\ep$.
We consider the sequence of coefficients $A^\ep(x)$ and the corresponding solution $u^\ep(x)$ of (\ref{eq:original_gen}).
The G-convergence of the sequence $A^\ep(x)$ is defined as follows \cite{spagnolo1967sul,spagnolo1976convergence}:
\begin{definition}
\label{def:glimit}
A sequence of coefficient $A^\ep(x)$ in (\ref{eq:original_gen}) is said to G-converge to a limit $A^*(x)$ as $\ep$ tends to $0$, if the sequence of solution $u^\ep(x)$ converges weakly in $H^1_0(\Omega)$ to $u_0(x)$, the unique solution of the following homogenized equation,
\beq
\label{eq:homogenized_gen}
\bsp
-\div \bigg(A^*(x) \nabla u_0(x) \bigg) &= f(x) \ \ \textrm{in} \ \ \Omega, \\
u_0(x) &= 0 \ \ \textrm{on} \ \ \partial \Omega,
\end{split}
\eeq
for any source term $f(x)$. The limit matrix $A^*(x)$ is called the G-limit of $A^\ep(x)$.
\end{definition}

We now define the following class of matrices. 
\begin{definition}
A matrix function $A(x)$ is said to belong to $E(\alpha,\beta,\Omega)$ if the followings are satisfied for some $\alpha$, $\beta>0$.
\beq
\bsp
&A(x) \in L^\infty(\Omega)^{N\times N}, \\
&A(x)k\cdot k \geq \alpha |k|^2, \ \ \textrm{for all} \ \ k \in\mathbb{R}^N,\ a.e. \ x\in\Omega\\
&|A(x)k| \leq \beta|k|, \ \ \textrm{for all} \ \ k \in\mathbb{R}^N,\ a.e. \ x\in\Omega.
\end{split}
\eeq
\end{definition}
We have the following theorem that justifies the definition of G-convergence \cite{defranceschi1993introduction}.
\begin{theorem}
\label{thm:Gconv}
Let $A^\ep(x)$ be a sequence of functions that belong to $E(\alpha,\beta,\Omega)$. Then there exist a function $A^*(x) \in E(\alpha,\beta,\Omega)$ such that $A^\ep(x)$ G-converges to $A^*(x)$ up to subsequence.
\end{theorem}

The following theorem guarantees the uniqueness of the G-limit.
\begin{theorem}
\label{thm:gconvunique}
The G-limit of a G-converging sequence is unique.
\end{theorem}
\begin{proof}
See \cite[Section 7]{defranceschi1993introduction}
\end{proof}

The following remark provides important properties of the G-limit and one motivation for the recovery of the G-limit.
\begin{remark}
The G-limit $A^*(x)$ does not depend on the source term $f(x)$ by definition. It is also known that it also does not depend on the boundary conditions \cite[Chapter 1]{allaire2012shape}.
Thus, the G-limit recovered with specific source term $f(x)$ and the boundary condition $g(x)$ in (\ref{eq:original_gen}) can be reused with the different source terms in the same medium.  
\end{remark} 

From Theorem \ref{thm:Gconv}, we know that a well-posed homogenized limit (\ref{eq:homogenized_gen}) exists, but in general, there is no systemic way to find the explicit formula for the G-limit $A^*(x)$.
In addition, the G-convergence is only guaranteed up to a subsequence in the theorem. For (locally) periodic media, the G-convergence is well studied and the G-limit can be computed by the periodic homogenization methods \cite{papanicolau1978asymptotic,jikov2012homogenization}.
We remark that even though it might not be clear that how to construct the explicit form of the the G-limit $A^*(x)$ in general,  \eqref{eq:homogenized_gen} does provide the structure of the homogenized equation served as a generic priori knowledge for PINNs.

\subsection{Homogenization for periodic media}
 In this section, we present the outline and the convergence results of the standard periodic homogenization.
We let $\Omega \in \mathbb{R}^N$ be a bounded domain and $Y$ be a unit cube in $\mathbb{R}^N$. 
We consider the homogenization of the following multiscale elliptic equation:
\beq
\label{eq:original_per}
\bsp
-\div \bigg(A^\ep (x) \nabla u^\ep(x) \bigg) &= f(x) \ \ \textrm{in} \ \ \Omega, \\
u^\ep(x) &= 0 \ \ \textrm{on} \ \ \partial \Omega.
\end{split}
\eeq
Here, $\ep$ represents the fine scale of the system. The coefficient has the scale separation and is defined by $A^\ep(x) = A(x,\xoe)$, where $A(x,y)$ is $Y$-periodic with respect to the fast variable $y$. Thus, we consider the coefficient $A^\ep(x)$ with smooth finescale oscillations. We further assume that $A^\ep(x)$ is in ${\mathcal C}^\infty(\Omega)$ and uniformly positive, i.e., $A^\ep(x) > c >0$ for some constant $c$. 

We consider the following two-scale asymptotic expansion of the solution $u^\ep(x)$.
\beq
\label{eq:asympt}
u^\ep(x) = u_0(x) + \ep u_1(x,\xoe) + \ep^2 u_2(x,\xoe) + \dots,
\eeq
where $u_i(x,\xoe)$, ($i = 1,2,\dots$) are $Y$-periodic with respect to $y = \xoe$.
We can derive the following homogenized equation with the G-limit $A^*(x)$ using the above expansion:
\beq
\label{eq:homogenized_per}
\bsp
-\div \bigg(A^*(x) \nabla u_0(x) \bigg)  &= f(x) \ \ \textrm{in} \ \ \Omega, \\
u_0(x) &= 0 \ \ \textrm{on} \ \ \partial \Omega,
\end{split}
\eeq
where the G-limit $A^*(x)$ is defined as follows:
\beq
\label{eq:kappastar}
A^*_{ij}(x) =  \int_Y  A(x,y) (\delta_{ij} + {\partial  \chi^j(x,y)\over \partial y_i}) \dy,
\eeq
where $\chi^i(x,y)$ is the solution of the following {\it cell problem}:
\beq
\label{eq:cell}
\div_y\bigg(A(x,y) \nabla_y \chi^i(x,y)\bigg) = -\div_y(A(x,y)e^i),
\eeq
on $Y$ with periodic boundary condition. Here, $e^i$ is the standard basis vector in $\mathbb{R}^n$.

This homogenized equation does not depend on the fine scale $\ep$ and the solution $u_0(x)$ represents the macroscopic behavior of the solution $u^\ep(x)$ to the multiscale equation (\ref{eq:original_per}) when $\ep$ is sufficiently small.
This can be rigorously explained by the following theorem on the convergence of the multiscales solution $u^\ep(x)$ to the homogeinzed solution $u_0$ \cite{papanicolau1978asymptotic}.
\begin{theorem}
\label{thm:weakconv}
Assume $A^\ep(x) \in L^\infty(\Omega)$, $f(x)\in L^2(\Omega)$. Let $u^\ep(x)$ and $u_0(x)$ be the solutions to (\ref{eq:original_per}) and (\ref{eq:homogenized_per}) respectively. Then as $\ep \to 0$, the sequence $u^\ep(x)$ converges weakly in $H^1(\Omega)$ to $u_0(x)$.
\end{theorem}

Above result is obtained under minimal regularity assumptions on the multiscale coefficient and the source term. 
However, in practice, it is often the case that we can achieve the strong convergence of the multiscale solution to the homogenized solution. For example, we have the following convergence estimates \cite[Chapter 6]{papanicolau1978asymptotic}. 
\begin{remark}
\label{rmk:strongconv}
Let $u^\ep(x)$ and $u_0(x)$ be the solutions to (\ref{eq:original_per}) and (\ref{eq:homogenized_per}) respectively. Assuming $A^\ep(x)$ and $f(x)$ are smooth, we have the following convergence rate.
\beq
\label{eq:linftyconv}
\bsp
\norm{u^\ep(x)- u_0(x)}_{L^\infty(\Omega)} \leq C \ep, 
\end{split}
\eeq
where $C>0$ is independent of $\ep$.
\end{remark}

In (locally) periodic media, we solve the cell problems (\ref{eq:cell}) to compute the G-limit when the explicit form of the multiscale coefficient $A^\ep(x)$ is known. 
Without the periodicity assumption or a known multiscale coefficient $A^\ep(x)$,  traditional homogenization methods are typically not applicable. Despite the fact that the convergence result (\ref{eq:linftyconv}) holds only for periodic cases, we can still expect the multiscale solution data to be close to the homogenized solution for sufficiently small $\ep$ even if the periodic assumption is violated. This  motivates us to utilize  multiscale solution data as a surrogate for the corresponding homogenized solution data for more general scenarios beyond the periodicity assumptions.

\subsection{Inverse Problem formulation}

Equipped with the background knowledge introduced above, we now consider the  multiscale elliptic equations (\ref{eq:original_gen}) and assume   a  well-posed  homogenized  limit (\ref{eq:homogenized_gen})  exists. 
We also assume the multiscale coefficient is 
smooth, 
but no geometric assumptions, such as periodicity, are required. 

In this work, we consider the following inverse problem setting: given a set of observations/data points, our goal is to learn the G-limit $A^*(x)$ and the homogenized solution $u_0(x)$ of the homogenized limit (\ref{eq:homogenized_gen}). 
As we mentioned before, the homogenized solution data are often not  available.
Instead, we utilize the multiscale solution data of the equation (\ref{eq:original_gen}) as a surrogate for the homogenized solution data. 

We remark that even though the multiscale solution data are close to the homogenized solution in most of the regions in our domain for sufficiently small $\ep$, our solution data contains multiscale or noise fluctuations that do not present in the homogenized solution.  
This introduces additional difficulties because one needs to approximate the slowly varying functions from multiscale solution.
It is preferable for a method to be less sensitive to these finescale oscillations and the noise in our multiscale solutions data. 
Motivated by recent developments of the physics-informed neural networks (PINNs) \cite{lu2019deepxde,raissi2019physics}, we propose to develop PINNs for estimating the G-limits,  which not only  simultaneously match the measurements/data  while respecting the underlying physics in the problem,  but also provide an effective regularization to mitigate the adversarial effects due to the multiscale fluctuations or noise  in the data.

\section{Method}
Next, we will briefly review physics-informed neural networks (PINNs) \cite{lu2019deepxde,chen2020physics} and adopt it to tackle the inverse problems to learn the G-limits in the homogenized equation (\ref{eq:homogenized_gen}), given the corresponding multiscale solution data. 

\subsection{Feed-forward neural network}
We shall use feed-forward neural networks to approximate the solution $u_0(x)$ and the effective coefficients $A^*(x)$ in (\ref{eq:homogenized_gen}). 
The feed-forward neural network with $L$ layers and $N_l$ neurons in the $l$th layer is a function ${\mathcal N}_\theta(x): \mathbb{R}^{N_0} \to \mathbb{R}^{N_L}$ defined by 
\beq
\bsp
{\mathcal N}_\theta(x) &= W^L {\mathcal N}^{L-1}(x) + b^L, \\ 
{\mathcal N}^l(x) &= \sigma (W^l {\mathcal N}^{l-1}(x) + b^l), \\ 
{\mathcal N}^1(x) &= W^1 x + b^1,
\end{split}
\eeq
for $1<l<L$. The matrix $W^l \in \mathbb{R}^{N_{l-1} \times N_l}$ and the vector $b^l \in \mathbb{R}^{N_l}$ represent the weight and bias in $l$-th layer and $\sigma$ is a nonlinear activation function, such as ReLU function, the hyperbolic tangent function, and the sine function \cite{goodfellow2016deep}. 
We further define the set of tunable weights and biases of the neural network, $\theta = \{ W^l, b^l\}$ for $1\leq l \leq L$.

\subsection{PINNs for inverse problems}
For ease of presentation, we consider the following partial differential equation for the solution $u(x)$ with an unknown coefficient $A(x)$:
\beq
\label{eq:residual}
{\mathcal F}\left[u(x);A(x)\right] = 0, \ \textrm{in} \ \Omega, 
\eeq
given a Dirichlet boundary condition,
\beq
\label{eq:bc}
u(x) = g(x), \ \ \textrm{on} \ \ \partial\Omega. 
\eeq
Given  the observation data on the solution $u(x)$ is available,
we are interested in recovering the unknown coefficient $A(x)$ and the entire solution field $u(x)$.

PINNs  employ the feed-forward networks ${\mathcal N}_{\theta_u}(x)$ and ${\mathcal N}_{\theta_A}(x)$ to approximate the solution and the unknown coefficients respectively, where $\theta_u$ and $\theta_A$ represent the trainable network parameters for each network. Then we train the networks to get the approximations $\hat{u}(x)$ for solutions and $\hat{A}(x)$ for unknown coefficients by minimizing the following loss functional, including  data misfit, the PDE residual loss (\ref{eq:residual}) and the boundary condition loss (\ref{eq:bc}) over the training set ${\mathcal T}$:
\beq
\label{eq:loss_gen}
{\mathcal L}(\theta_{u},\theta_{A}; {\mathcal T}) = \lambda_r{\mathcal L_r}(\theta_{u},\theta_{A}; {\mathcal T}_r)+  \lambda_d{\mathcal L_d}(\theta_{u},\theta_{A}; {\mathcal T}_d) +    \lambda_b{\mathcal L_b}(\theta_{u},\theta_{A}; {\mathcal T}_b),
\eeq
where
\beq
\bsp
&{\mathcal L}_r(\theta_{u},\theta_{A}; {\mathcal T}_r) = \frac{1}{|{\mathcal T}_r|} \displaystyle\sum_{x_i^r\in{\mathcal T}_r} \left\lvert {\mathcal F}\left(\hat{u}(x_i^r);\hat{A}(x_i^r)\right) \right\rvert^2,\\
&{\mathcal L}_d(\theta_{u},\theta_{A}; {\mathcal T}_d) = \frac{1}{|{\mathcal T}_d|} \displaystyle\sum_{x_i^d\in{\mathcal T}_d} |\hat{u}(x_i^d)-u(x_i^d) |^2, \\
&{\mathcal L}_b(\theta_{u},\theta_{A}; {\mathcal T}_b) = \frac{1}{|{\mathcal T}_b|} \displaystyle\sum_{x_i^b\in{\mathcal T}_b} |\hat{u}(x_i^b)-g(x_i^b) |^2,
\end{split}
\eeq
where  $\lambda_r$, $\lambda_d$  and $\lambda_b$ denote the weights for each loss term. The training points ${\mathcal T} = {\mathcal T}_r \cup {\mathcal T}_d \cup {\mathcal T}_b$. ${\mathcal T}_d$, ${\mathcal T}_r$, and ${\mathcal T}_b$ denote  data/measurement points, PDE residual points, and boundary data points. Both ${\mathcal T}_r \in \Omega$, and ${\mathcal T}_b \in \partial \Omega$ are predfined and can be chosen from mesh grid points or randomly.
The parameters $\theta_u$ and $\theta_A$ can be found 
by minimizing the loss function (\ref{eq:loss_gen}), and the resulting networks $\hat{u}(x)$ and $\hat{A}(x)$ are the approximations to the solution $u(x)$ and the coefficient $A(x)$ of the equation \eqref{eq:residual}.

\subsection{Learning the G-limits via PINNs}

Following that, we adopt the PINNs framework to tackle the inverse problem of estimating the G-limit $A^*(x)$  for the multiscale elliptic equation (\ref{eq:original_gen}). One issue is that the measurements of the homogenized solution are often not available. Motivated by the convergence results for periodic media in \eqref{eq:linftyconv}, we employ the multiscale solution data $u^{\epsilon}$ of the multiscale equation (\ref{eq:original_gen}) as the training data, which is expected to be a good surrogate for the homogenized solution data when $\ep$ is sufficiently small.

We construct two feed-forward neural networks $\hat{A}^*(x) = {\mathcal N}_{\theta_{A^*}}(x)$ and $\hat{u}_0(x)={\mathcal N}_{\theta_{u_0}}(x)$ 
to approximate the G-limit and the solution of the homogenized equation (\ref{eq:homogenized_gen}). Since we consider Dirichlet boundary condition \eqref{eq:bc} in this work, the boundary condition can be embeded into the neural network exactly. 
Specifically, we follow the approach suggested in \cite{lu2021physics} by modifying the solution network output ${\mathcal N}_{\theta_{u_0}}$:
\beq
\label{eq:hardconst}
\hat{u}_0(x) = g(x)+l(x){\mathcal N}_{\theta_{u_0}},
\eeq
where $u_0(x) = g(x)$ is a Dirichlet boundary condition, and $l(x)$ is a function that satisfies the following conditions.
\beq
l(x) = 0 \ \ \textrm{on} \ \ \partial \Omega, \ \ \ \   l(x) >0 \ \ \textrm{in} \ \ \Omega-\partial \Omega.
\eeq
With a simple domain, we can analytically choose $l(x)$ \cite{lagaris1998artificial}. For example, for the domain $[a,b]^2$, we can choose $l(x) = (x_1-a)(b-x_1)(x_2-a)(b-x_2)$, where $x = (x_1,x_2)$.

We then seek a set of network parameters $\theta_{u_0}$ and $ \theta_{A^*}$ that minimize the loss function defined as follows:
\beq
\label{eq:loss_hom}
{\mathcal L}(\theta_{u_0},\theta_{A^*}; {\mathcal T}) = \lambda_r{\mathcal L_r}(\theta_{u_0},\theta_{A^*}; {\mathcal T}_r)+  \lambda_d{\mathcal L_d}(\theta_{u_0},\theta_{A^*}; {\mathcal T}_d),
\eeq
where 
\beq
\bsp
&{\mathcal L}_r(\theta_{u_0},\theta_{A^*}; {\mathcal T}_r) = \frac{1}{|{\mathcal T}_r|} \displaystyle\sum_{x_i^r\in{\mathcal T}_r} \left\lvert \div\bigg(\hat{A}^*(x_i^r)\nabla \hat{u}_{0}(x_i^r)\bigg)+f(x_i^r) \right\rvert^2,\\
&{\mathcal L}_d(\theta_{u_0},\theta_{A^*}; {\mathcal T}_d) = \frac{1}{|{\mathcal T}_d|} \displaystyle\sum_{x_i^d\in{\mathcal T}_d} |\hat{u}_{0}(x_i^d)-u^{\ep}(x_i^d) |^2.
\end{split}
\eeq
Here, $u^{\ep}(x_i^d)$ denotes the (noise-free/noisy) multiscale solution data at $x_i^d$. Note that since the neural network $\hat{u}_0(x)$ satisfies the boundary condition exactly, there are only two terms in the loss function.
The first term  (\ref{eq:loss_hom}) encourages the neural network to respect  the homogenized equation (\ref{eq:homogenized_gen}). 
The second term makes sure that  the approximated homogenized solution is not far from the multiscale solution data $u^\ep(x)$. 
The choice of the regualization parameters  $\lambda_r$ and $\lambda_d$ could affect the training performance considerably. We adopted the adaptive weight techniques \cite{wang2020and} in this work. In summary,
Figure \ref{fig:pinns_scheme} presents the schematic deisgn of the PINNs for our problem.

\begin{figure}[t]
 \centering
\includegraphics[scale=0.3]{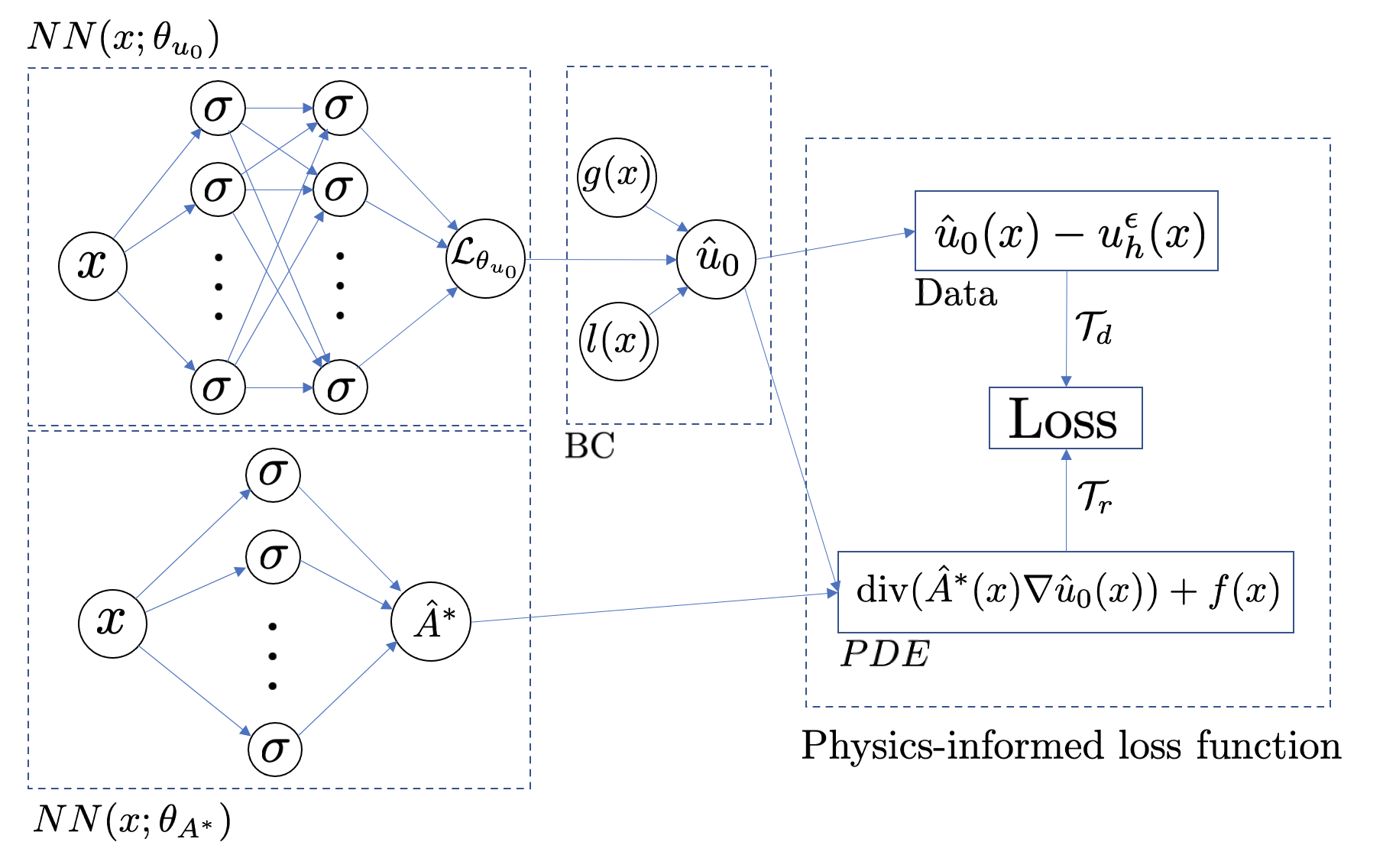}
\caption{ The schematic architecture of PINNs  for learning the G-limit $A^*(x)$ in the homogenized equation (\ref{eq:homogenized_gen}) by the neural network $\hat{A}^*(x)$. The boundary condition $u_0(x) = g(x)$ is strictly imposed using (\ref{eq:hardconst}).}
 \label{fig:pinns_scheme}
\end{figure}

\section{Numerical Examples}
\label{sec:NumExample}

In this section, we present several numerical examples to illustrate
the effectiveness and applicability of our method, 
including the elliptic equations with locally periodic, non-periodic, and ergodic random multiscale coefficients. 
The noise-free measurements are generated from the multiscale solution data 
of the multiscale elliptic equation by the underlying forward FEM simulation. For noisy scenario, we corrupt the noise-free data with independent, and identically distributed normal noise with different noise levels. 

To estimate the accuracy of the recovered G-limit $\hat{A}^*(x)$ and the homogenized solutions $\hat{u}_0(x)$, we use the following relative $L^2$-errors computed over a predefined mesh grid in spatial domain: 
\beq
\label{eq:kappaerror}
e_{\hat{A}^*} = \frac{\norm{\hat{A}^*(x)-A^*(x)}_{{L^2}(\Omega)}}{\norm{A^*(x)}_{{L^2}(\Omega)}}, \ \
e_{\hat{u}_{0}} = \frac{\norm{\hat{u}_{0}(x) -u_{0,h}(x)}_{L^2(\Omega)}}{\norm{ u_{0,h}(x)}_{L^2(\Omega)}}.
\eeq
Here, $A^*(x)$ is the reference G-limit that is either exact or pre-computed by FEM via traditional homogenization methods. 
The reference homogenized solutions, $u_{0,h}(x)$, are computed by FEM using the reference G-limits. 

During the training stage, 
we alternatively use ADAM and L-BFGS as suggested in \cite{lu2019deepxde,shin2020convergence}. 
A hyperbolic tangent function is used as the activation in all examples.  
In addition, the architectural parameters of neural network
were tuned to achieve reasonable results. The architecture parameters and other hyperparameters used for each example are listed in {\it Table \ref{tb:hyperparameters}} in the appendix. Advanced hyperparameter selection techniques can further improve the results, which, however, is not the focus of this work. 
In addition, the multiple restarts approach is adopted to
prevent the results from being affected by how the weights are (randomly) initialized. More specifically, we train the nets with a number of random initialization using Glorot normal initializer, 
and report the best possible results for each example.
All examples are carried out 
on Google’s Colab \cite{carneiro2018performance} using the library SciANN \cite{haghighat2021sciann}.

\subsection{Homogenization of a slowly varying periodic coefficient}
To test the basic capability  of our proposed method, we first consider the following multiscale elliptic equation with a slowly varying periodic coefficient: 
\beq
\label{eq:1dlocper}
\bsp
-\frac{d}{d x}\left(\frac{1+x^2}{2+\sin(2\pi\xoe)}\frac{d}{d x}u^\ep(x)\right) &= \cos(\pi x)\ \ \textrm{in} \ \ \Omega = [0,1] , \\
u^\ep(0) &= u^\ep(1) = 0.
\end{split}
\eeq
In this example, the permeability coefficient depends on both $x$ and $\xoe$, and is 
periodic with respect to $\xoe$. 
The analytical G-limit is known as $A^*(x) = \frac{x^2+1}{2}$. We compute the reference homogenized solution $u_{0,h}(x)$ using FEM via the exact G-limit. 

To generate the noise-free data, we compute the multiscale solution to the equation (\ref{eq:1dlocper}) for each finescale parameter value $\epsilon$ by FEM with mesh size $h = 1/10^5$ and obtain equally spaced data sampled from the multiscale solution  as the  training data.  For noisy data, we corrupt the measurements with different noise levels. %
The architecture parameters and other hyperparameters of PINNs are listed in {\it Table \ref{tb:hyperparameters}} in appendix.
The relative $L^2$ errors for both G-limit and homogenized solution are computed on a mesh with size $h=1/10^5$.

We first plot 
the relative $L^2$ errors with respect to the size of training data with $\ep = 2^{-7}$ in Figure \ref{figure:glimitsoltdns1}. 
For the noise-free case, the proposed method can achieve the errors at the level of ${\mathcal O}(10^{-3})$ for the G-limit and ${\mathcal O}(10^{-4})$ for the homogenized solution.
As the data set was enriched, the error level saturated. 
With noisy data, the error increases with the noise level and can be reduced as additional data are available, particularly for a high noise level. With $5\%$ noise, the relative errors for the approximated G-limit and solution are roughly $4\%$ and $1\%$ respectively, given enough data.
To further demonstrate the performance of the method, Figure \ref{figure:1dlocper_plots_ns} plots the G-limits and the homogenized solution recovered by PINNs for $\ep = 2^{-7}$, where  
 both G-limit and homogenized solution are well approximated under the different noise levels. As shown in Figure \ref{fig:1dlocper_plot_data_homsols_ums1_ns1_sc} and \ref{fig:1dlocper_plot_data_homsols_ums1_ns3_sc}, even when the multiscale data contain non-negligible random fluctuations, PINNs can still capture the macroscopic variation of the data thanks to the regularization provided by the homogenized equation.

\begin{figure}[!hbt]
	\centering
	\begin{subfigure}{0.40\textwidth}
  \includegraphics[width=\textwidth]{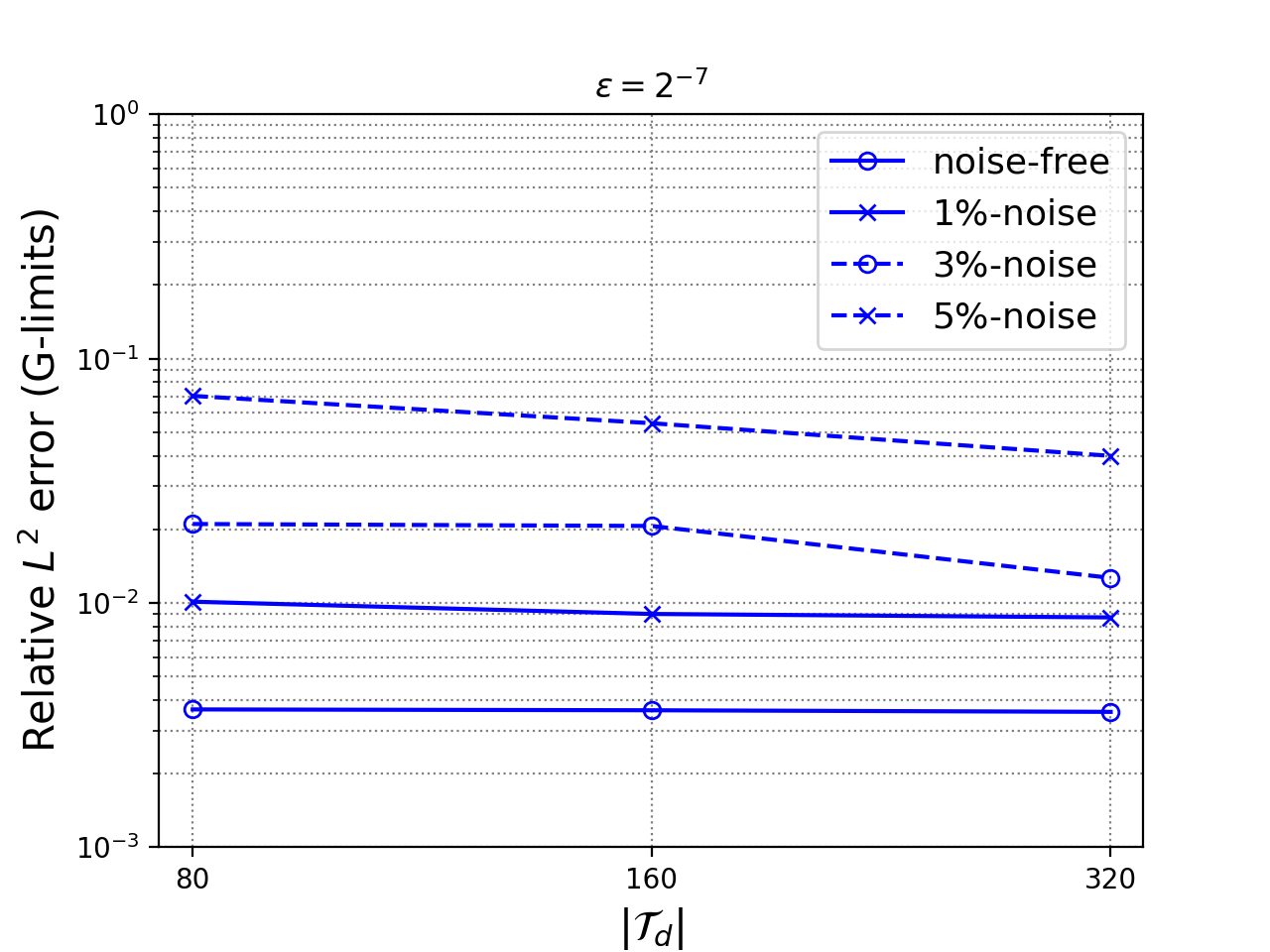}
  \caption{G-limits}
  \label{fig:1dlocper_errors_glimit_nd}
\end{subfigure}
\hspace{.3in}
  \begin{subfigure}{0.40\textwidth}
  \includegraphics[width=\textwidth]{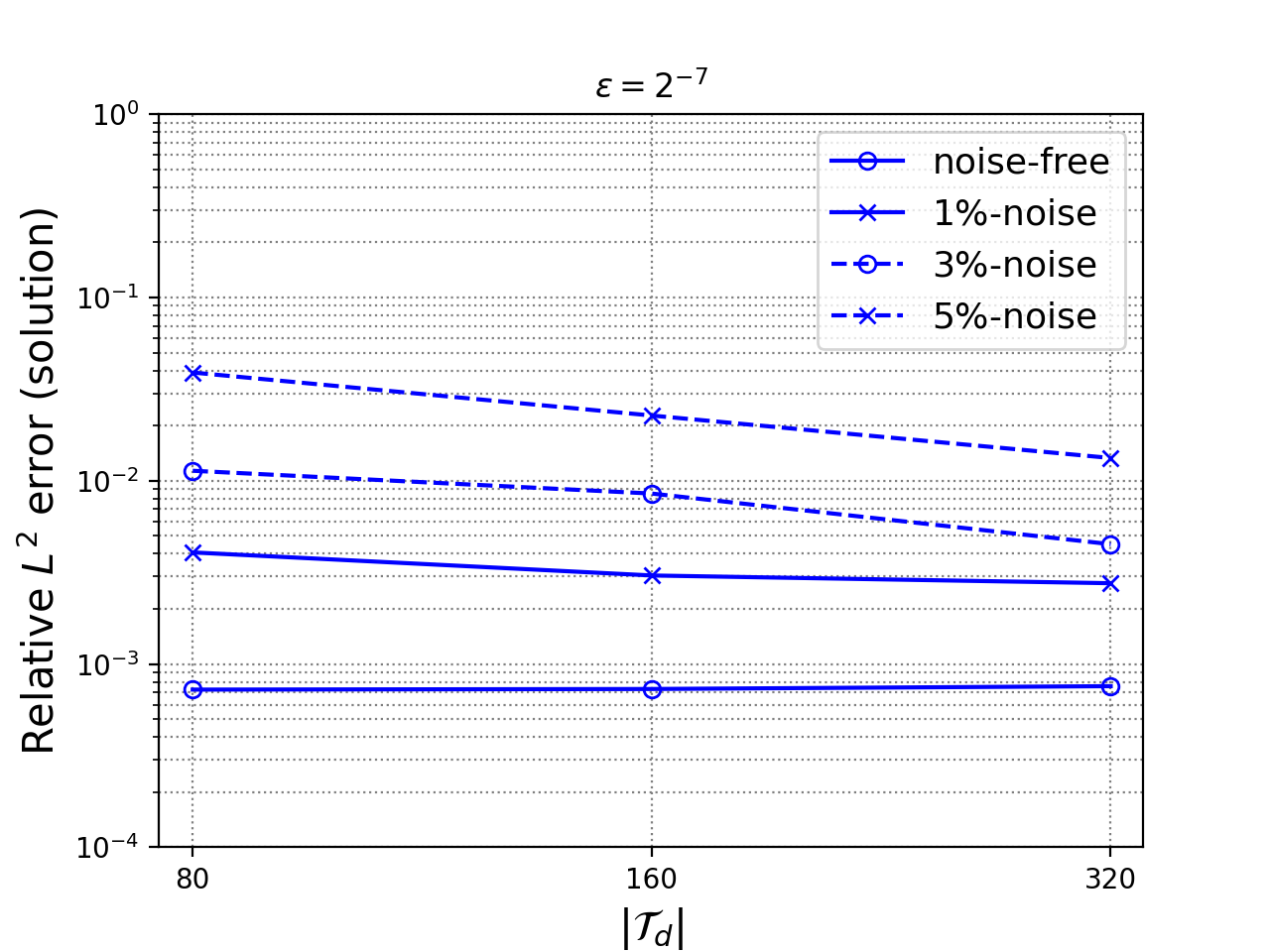}
  \caption{Homogenized solutions}
  \label{fig:1dlocper_errors_homsol_nd}
 \end{subfigure}
 \vspace*{-4mm}
 \caption{Problem (\ref{eq:1dlocper}): the relative $L^2$ errors for the G-limits and the homogenized solutions  with different number of multiscale data points corrupted by different noise levels for $\ep = 2^{-7}$ and the number of PDE residual points is $|{\mathcal T}_r| = |{\mathcal T}_d| +30$.}
\label{figure:glimitsoltdns1}
\end{figure}
\begin{figure}[!hbt]
	\centering
   	  \begin{subfigure}{0.35\textwidth}
  \includegraphics[width=\textwidth]{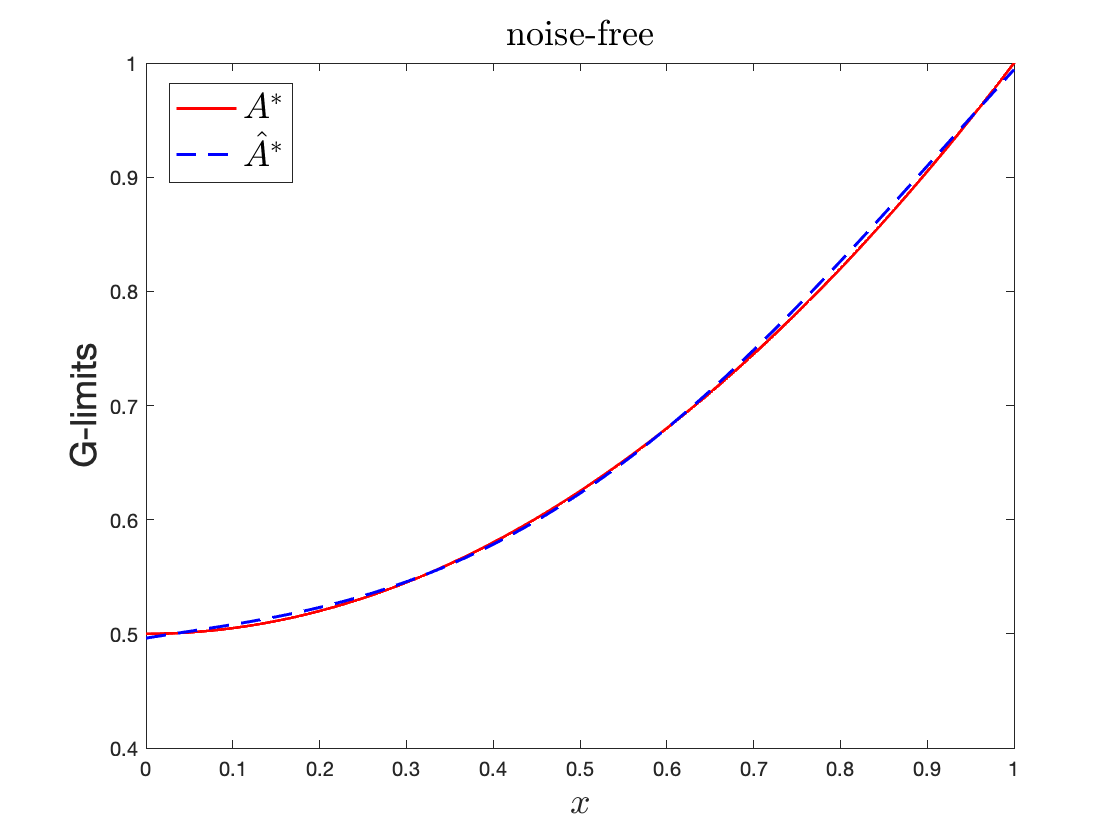}
  \caption{G-limits (noise-free)}
  \label{fig:1dlocper_plot_glimit_ums1_sc}
 \end{subfigure}
  \hspace{-.25in}
		\begin{subfigure}{0.35\textwidth}
  \includegraphics[width=\textwidth]{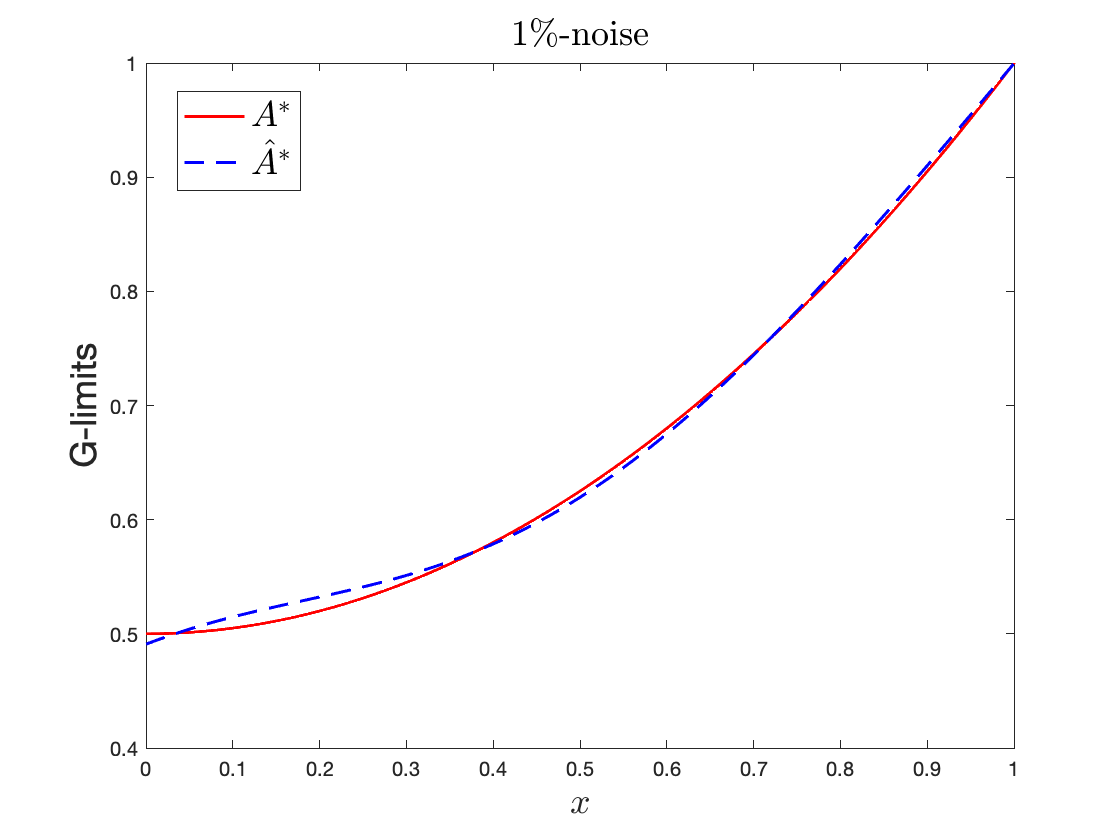}
  \caption{G-limits ($1\%$-noise)}
  \label{fig:1dlocper_plot_kappa_ums1_ns1_sc}
\end{subfigure}
  \hspace{-.25in}
	\begin{subfigure}{0.35\textwidth}
  \includegraphics[width=\textwidth]{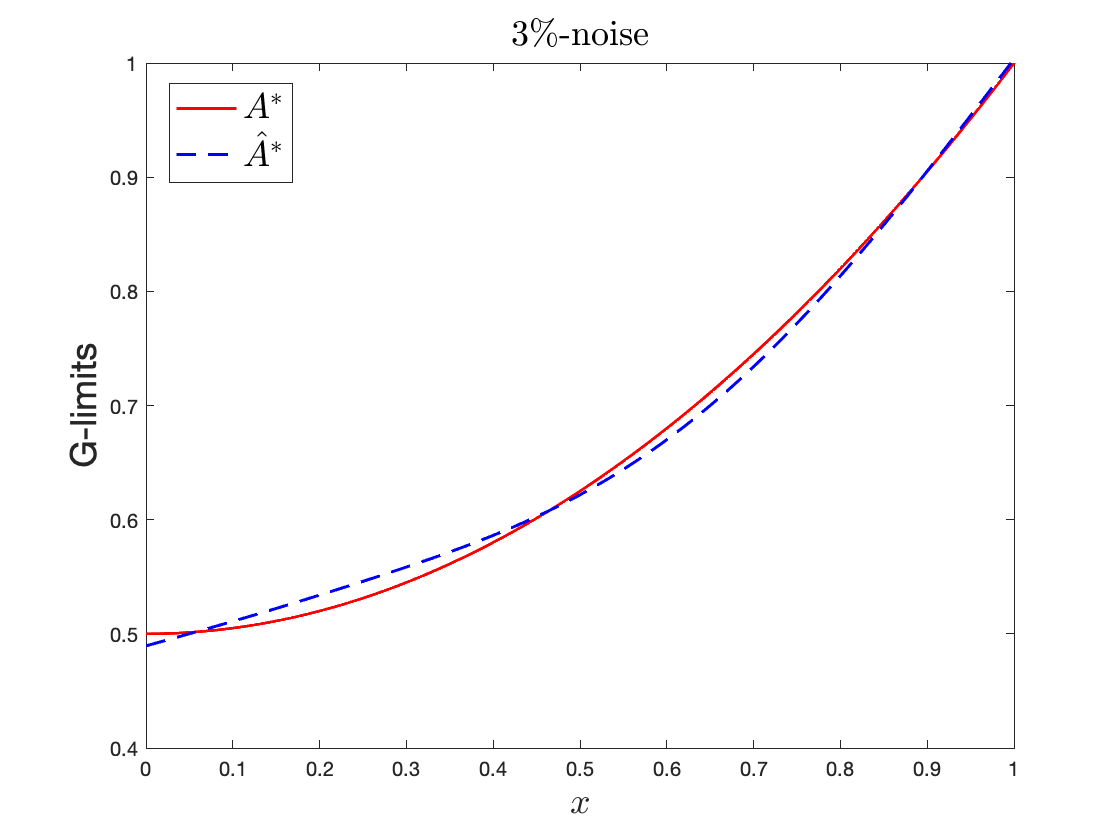}
  \caption{G-limits ($3\%$-noise)}
  \label{fig:1dlocper_plot_kappa_ums1_ns3_sc}
\end{subfigure}
  \hspace{-.25in}
  \begin{subfigure}{0.35\textwidth}
  \includegraphics[width=\textwidth]{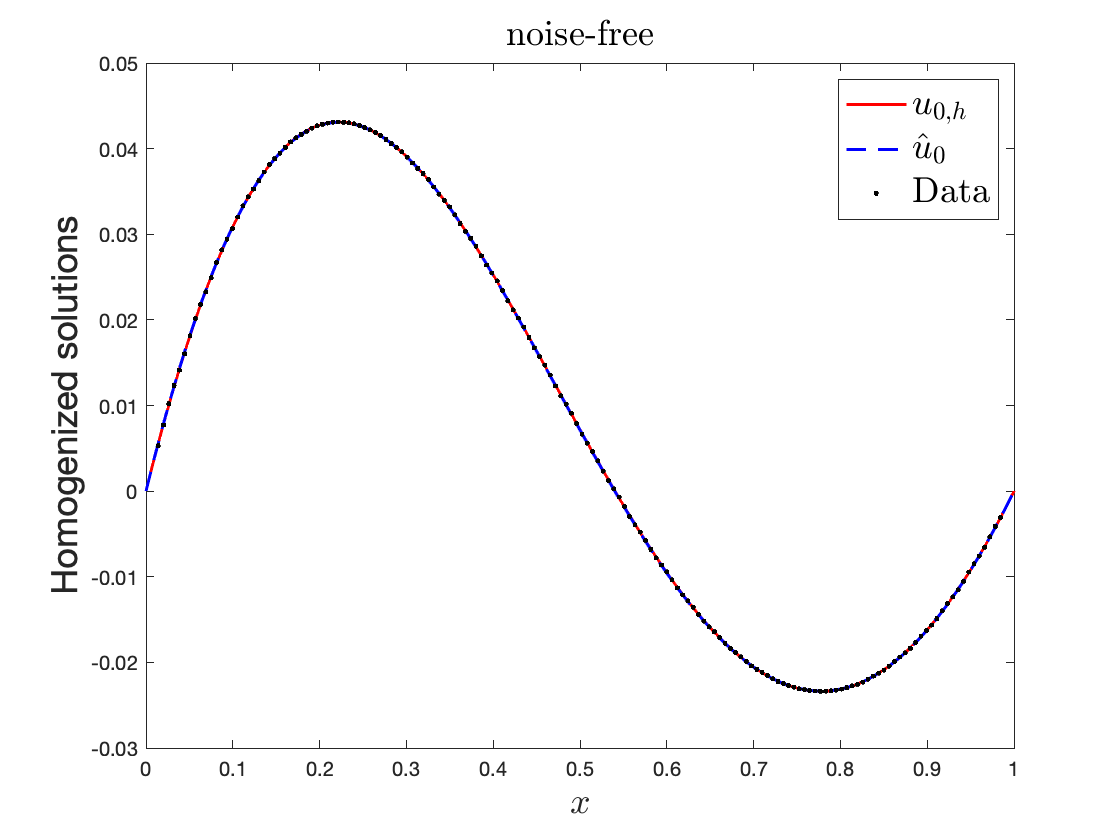}
   \caption{Solutions (noise-free)}
   \label{fig:1dlocper_plot_data_homsols_ums1_sc}
 \end{subfigure}
   \hspace{-.25in}
  \begin{subfigure}{0.35\textwidth}
  \includegraphics[width=\textwidth]{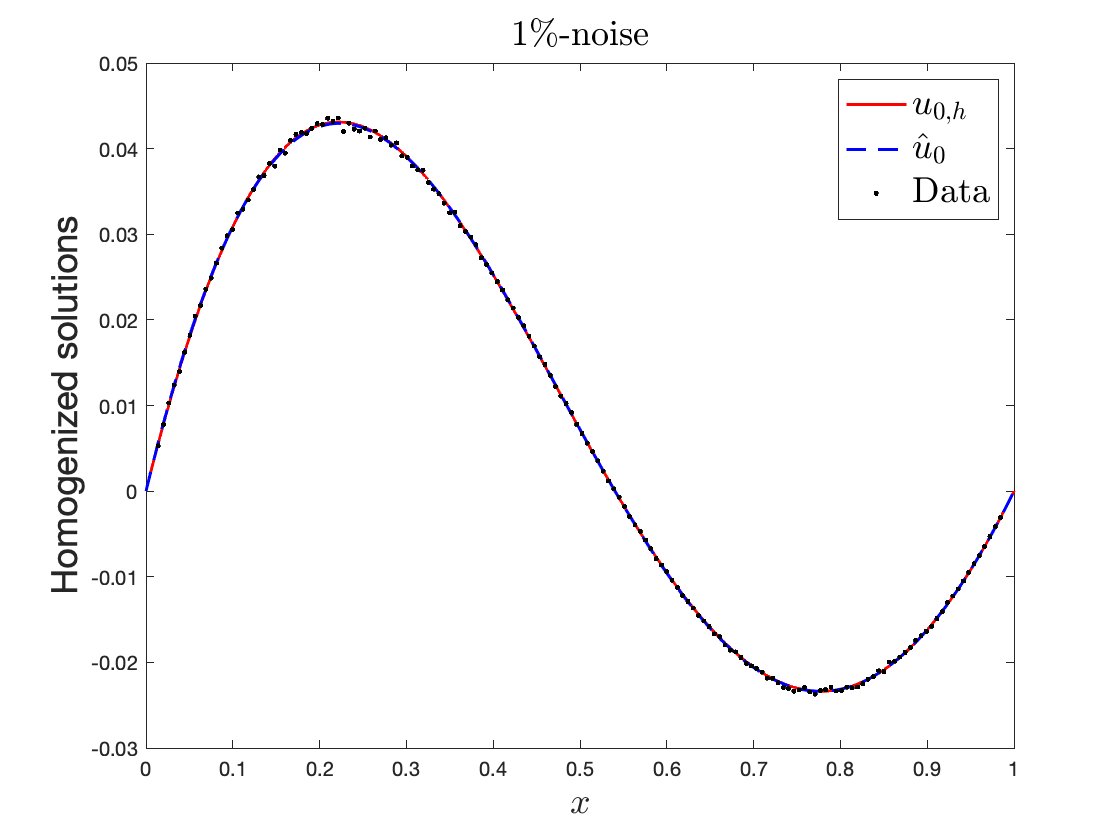}
  \caption{Solutions ($1\%$-noise)}
  \label{fig:1dlocper_plot_data_homsols_ums1_ns1_sc}
 \end{subfigure}
  \hspace{-.25in}
  \begin{subfigure}{0.35\textwidth}
  \includegraphics[width=\textwidth]{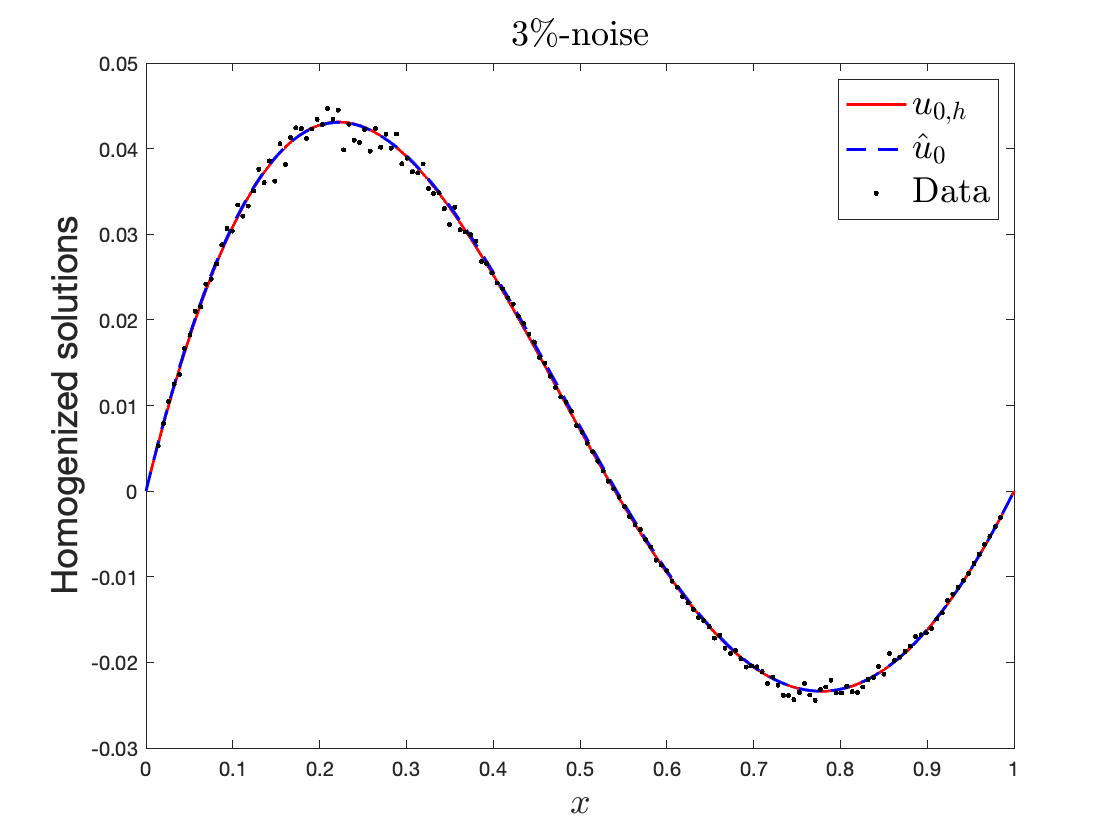}
  \caption{Solutions ($3\%$-noise)}
  \label{fig:1dlocper_plot_data_homsols_ums1_ns3_sc}
 \end{subfigure}
 \vspace*{-4mm}
 \caption{Error results for problem (\ref{eq:1dlocper}):   comparison of  the reference solutions ($A^*(x)$ and $u_{0,h}(x)$)  and the counterparts  ($\hat{A}^*(x)$ and $\hat{u}_{0}(x)$) learned by PINNs  with different noise levels in the data. (a. b. c.):  the  G-limit; (d. e. f.):  the homogenized solution, where $\ep = 2^{-7}$ and the number of multiscale data and PDE residual points are $|{\mathcal T}_d| = 160$, $|{\mathcal T}_r| = 190$.}
\label{figure:1dlocper_plots_ns}
 \vspace*{-5mm}
\end{figure}

\begin{figure}[!htb]
	\centering
	\begin{subfigure}{0.40\textwidth}
  \includegraphics[width=\textwidth]{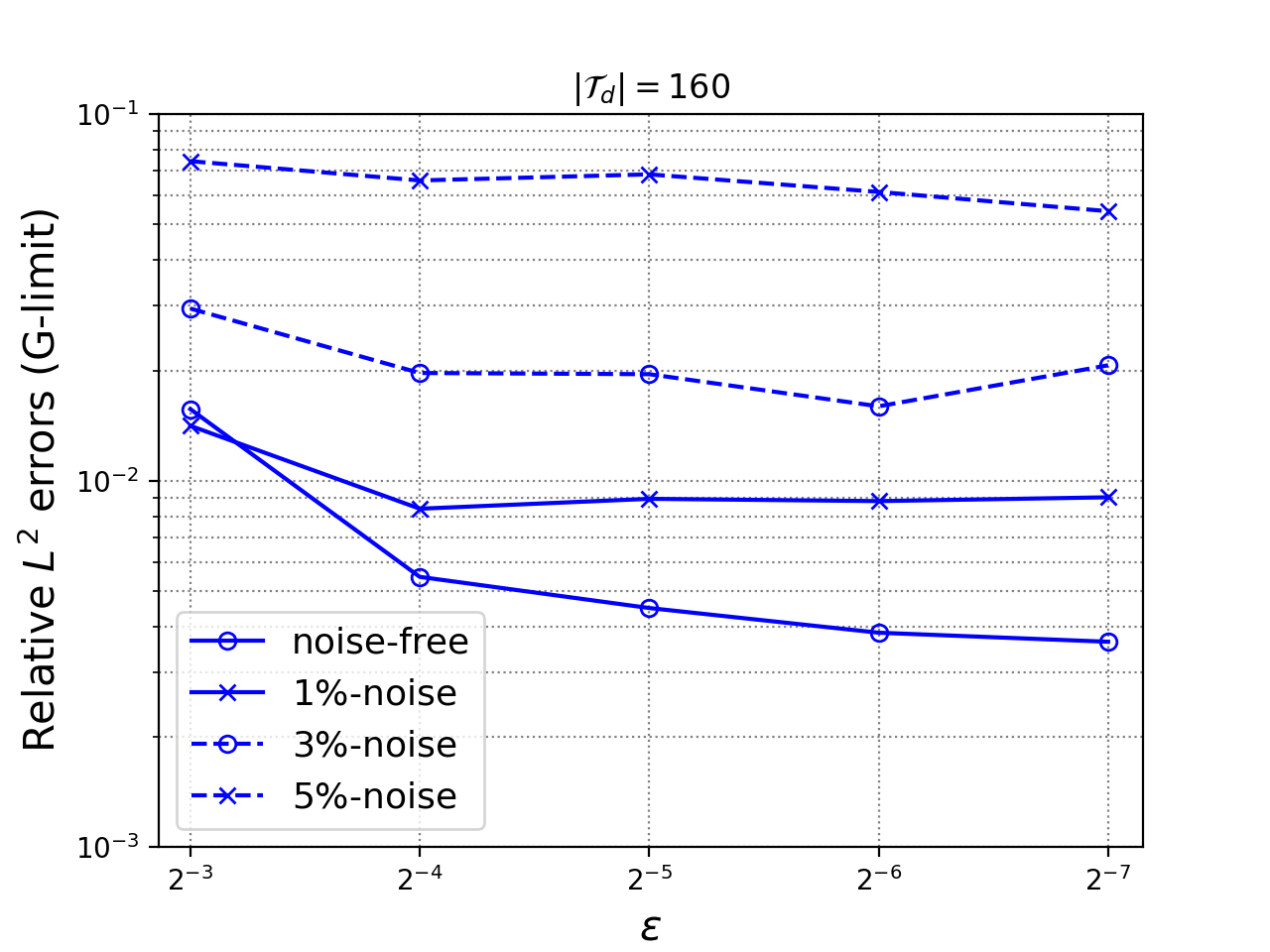}
  \caption{G-limits}
  \label{fig:1dlocper_errors_glimit_ep}
\end{subfigure}
\hspace{.3in}
  \begin{subfigure}{0.40\textwidth}
  \includegraphics[width=\textwidth]{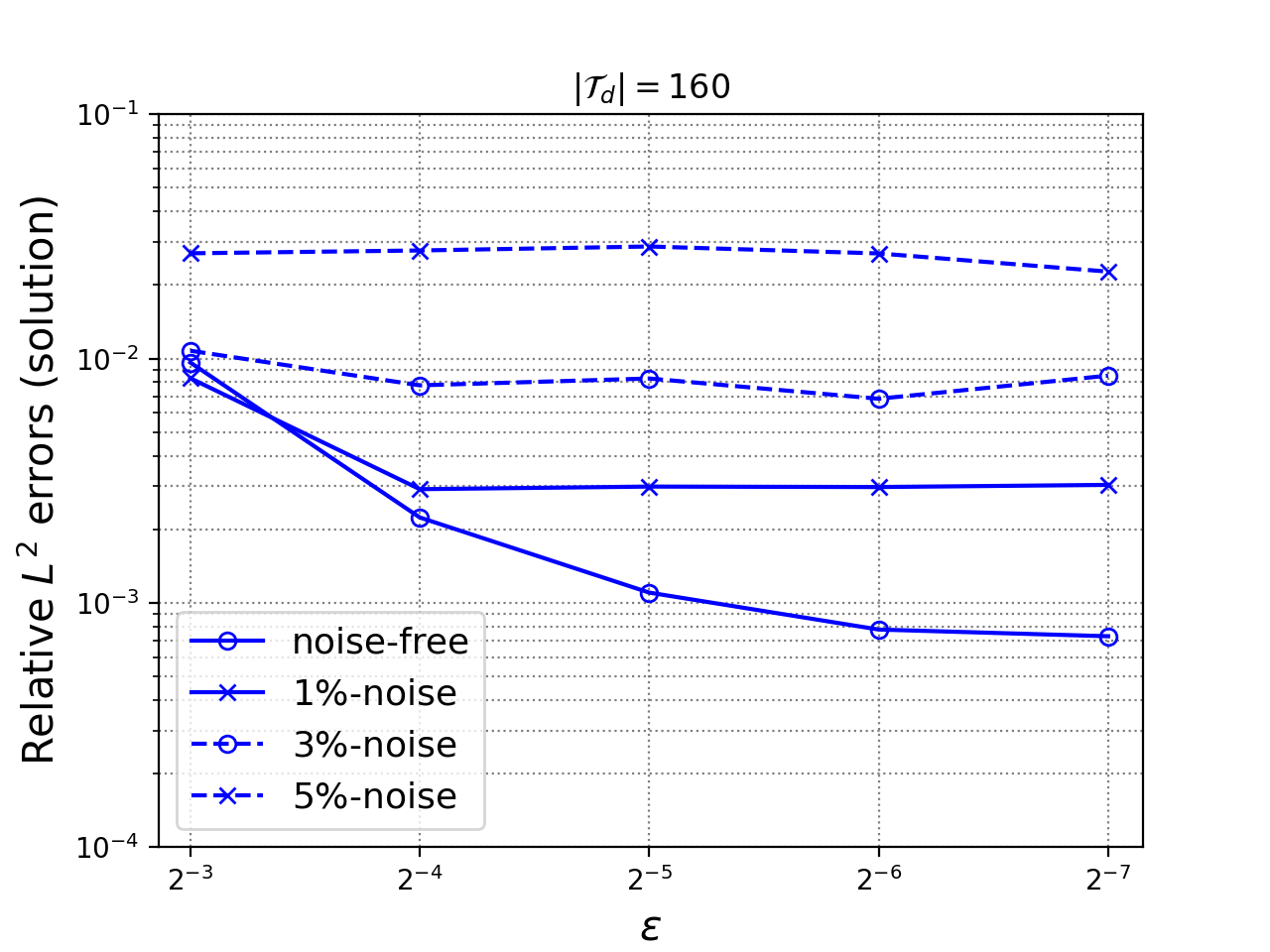}
  \caption{Homogenized solutions}
  \label{fig:1dlocper_errors_homsol_ep}
 \end{subfigure}
  \vspace*{-4mm}
 \caption{Error results for problem (\ref{eq:1dlocper}): the relative $L^2$ errors for the G-limits and the homogenized solutions  with  different finescale parameter $\ep$ and noise levels,  when the number of  multiscale  data and PDE residual points are  $|{\mathcal T}_d| = 160$ and $|{\mathcal T}_r| = 190$.}
\label{figure:glimitsolepns1}
  \vspace*{-3mm}
\end{figure}
  \begin{figure}[!htb]
 	\centering
 	  \begin{subfigure}{0.35\textwidth}
  \includegraphics[width=\textwidth]{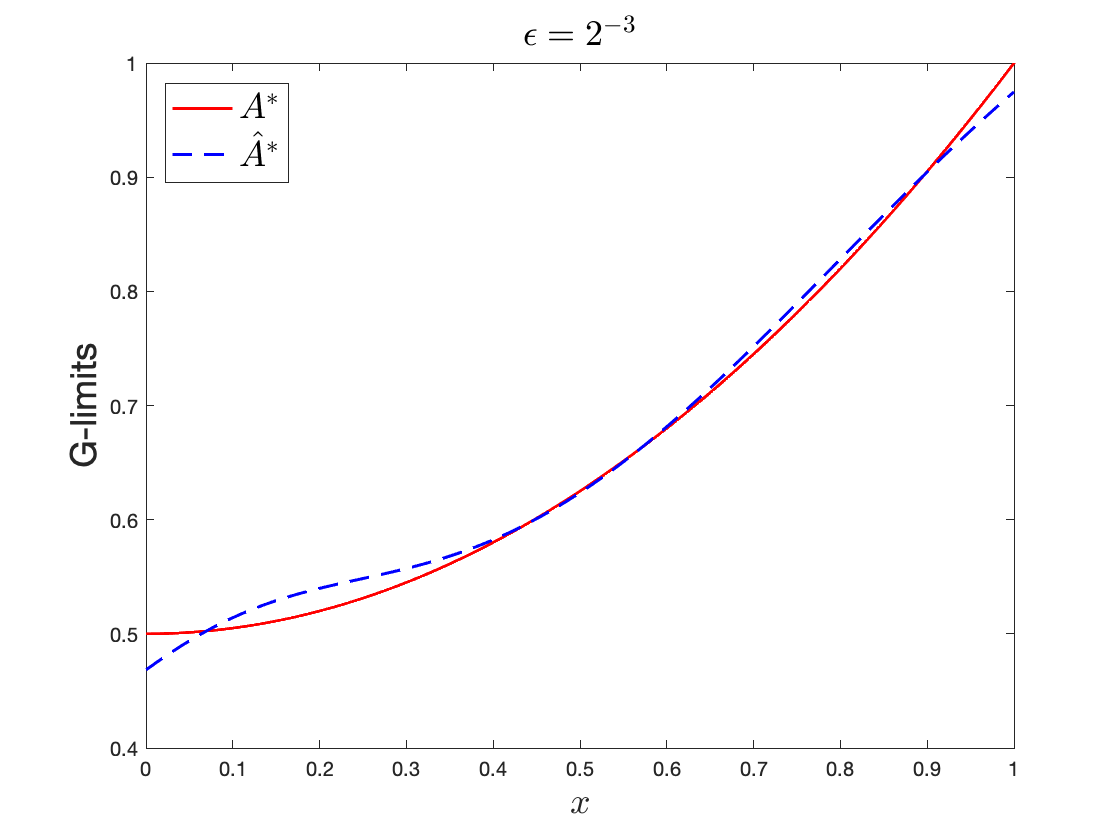}
  \caption{G-limit ($\ep = 2^{-3}$)}
  \label{fig:1dlocper_plot_glimit_ums1_ep3}
 \end{subfigure}
 \hspace{-.25in}
   	  \begin{subfigure}{0.35\textwidth}
  \includegraphics[width=\textwidth]{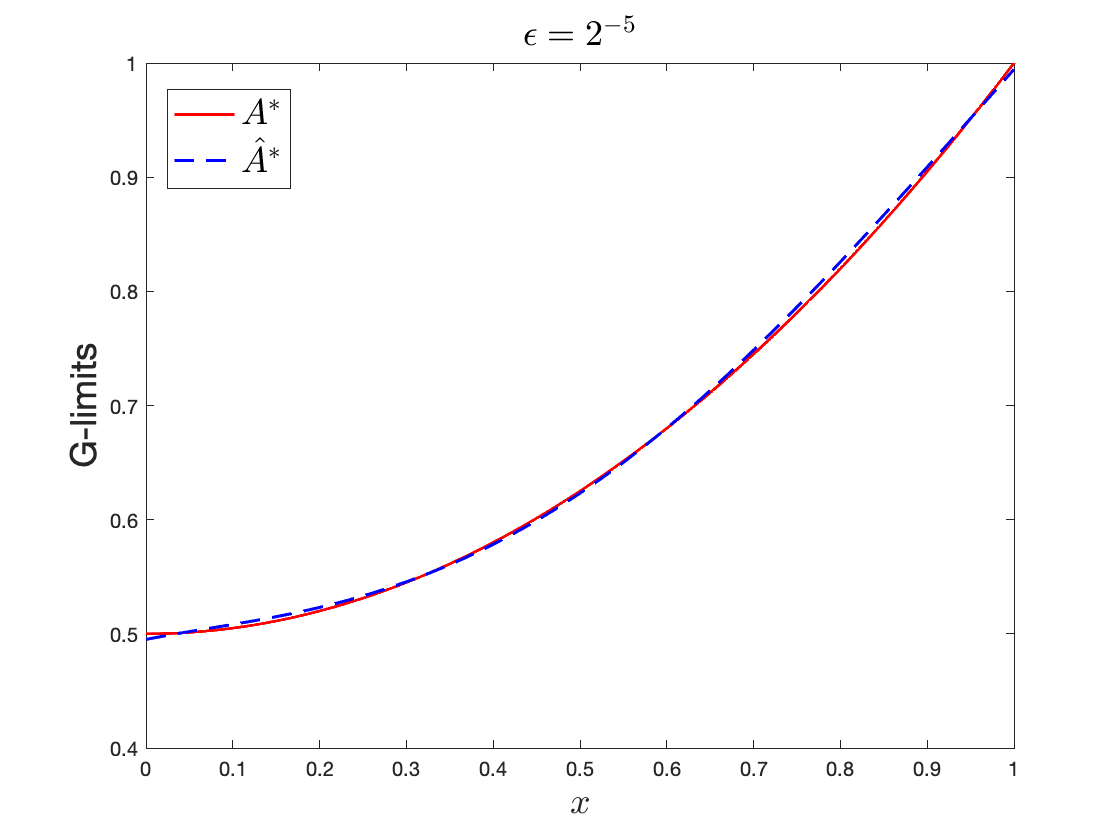}
  \caption{G-limit ($\ep = 2^{-5}$)}
  \label{fig:1dlocper_plot_glimit_ums1_ep5}
 \end{subfigure}
 \hspace{-.25in}
   	  \begin{subfigure}{0.35\textwidth}
  \includegraphics[width=\textwidth]{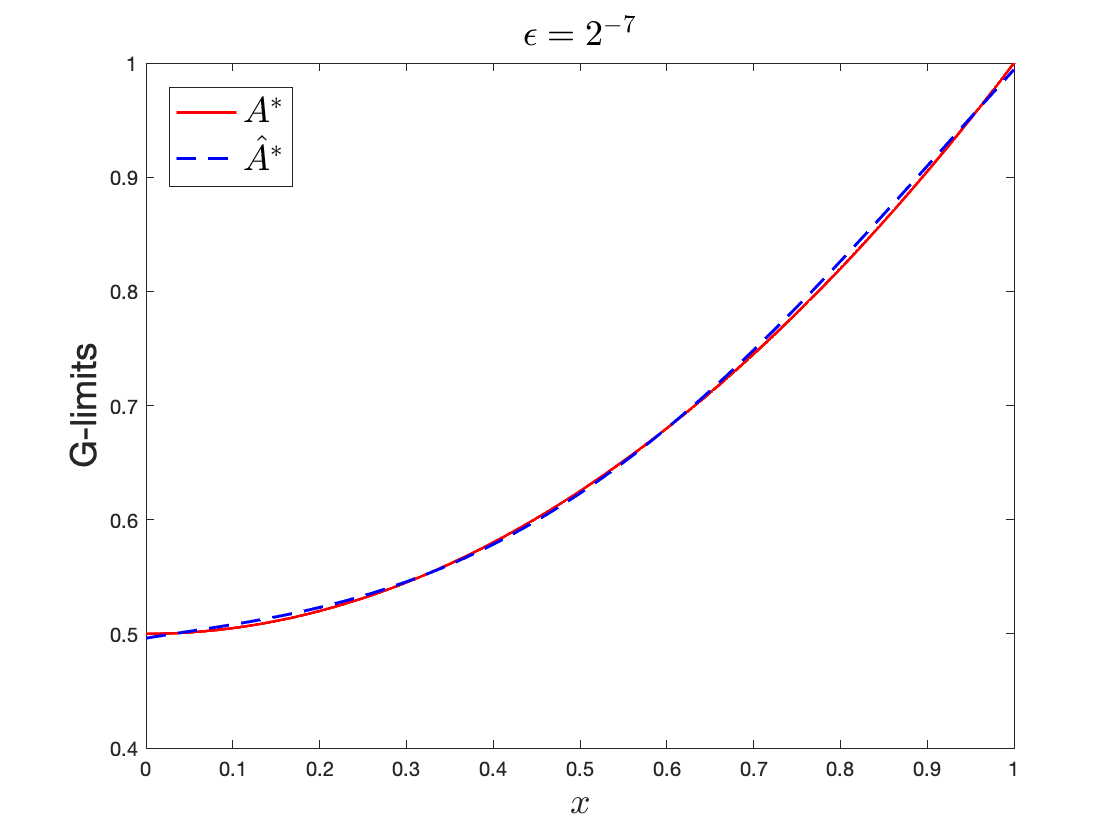}
  \caption{G-limit ($\ep = 2^{-7}$)}
  \label{fig:1dlocper_plot_glimit_ums1}
 \end{subfigure}
  \hspace{-.25in}
   	  \begin{subfigure}{0.35\textwidth}
  \includegraphics[width=\textwidth]{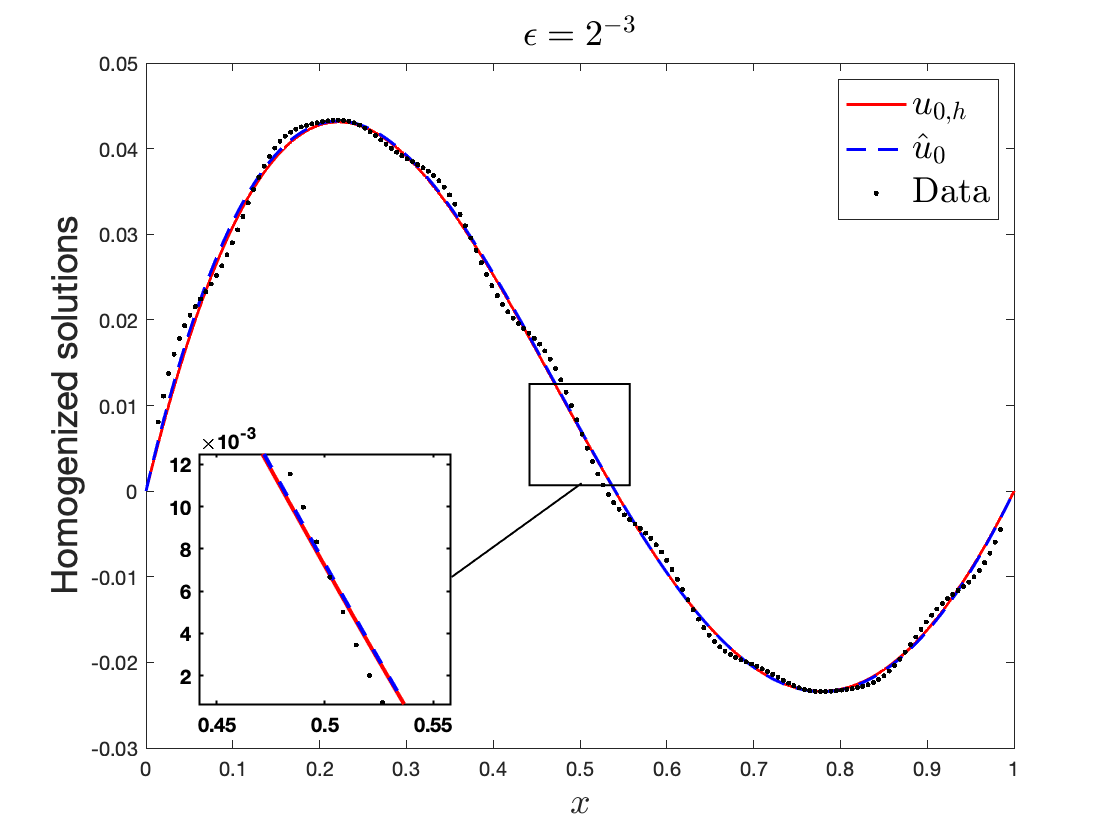}
   \caption{Solutions ($\ep = 2^{-3}$)}
   \label{fig:1dlocper_plot_data_homsols_ums1_ep3}
 \end{subfigure}
 \hspace{-.25in}
  	  \begin{subfigure}{0.35\textwidth}
  \includegraphics[width=\textwidth]{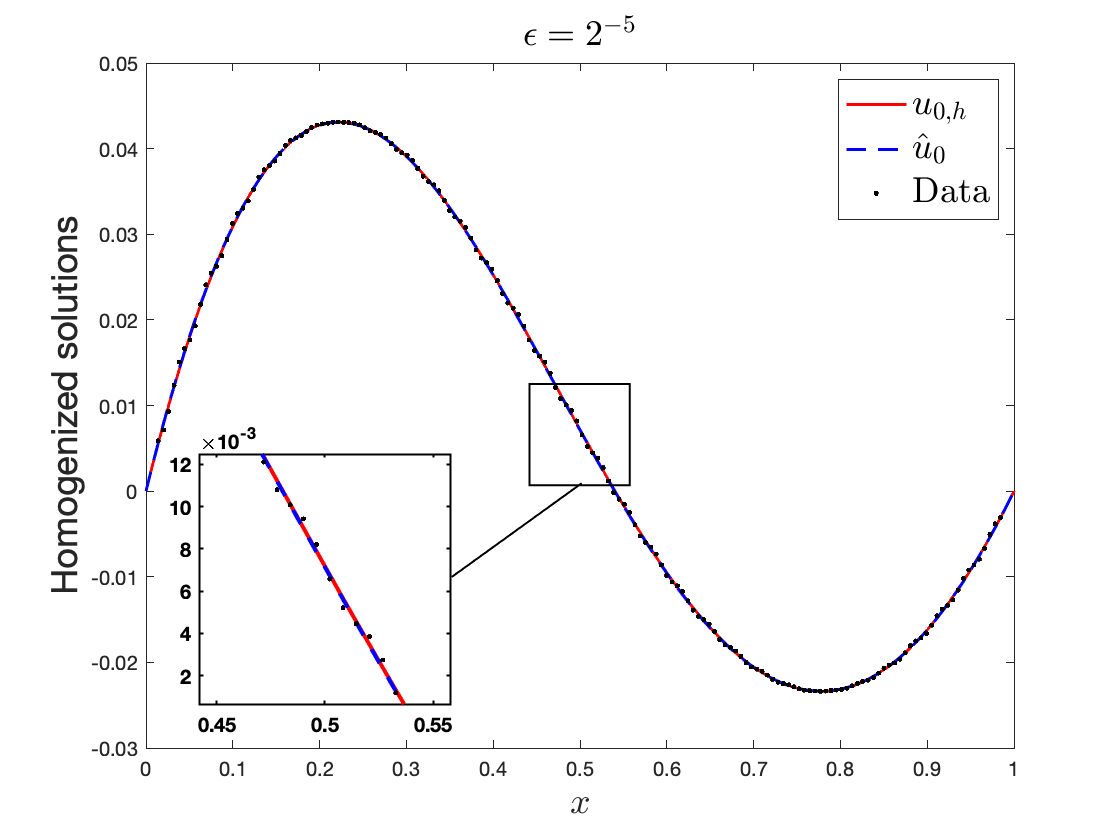}
   \caption{Solutions ($\ep = 2^{-5}$)}
   \label{fig:1dlocper_plot_data_homsols_ums1_ep5}
 \end{subfigure}
 \hspace{-.25in}
  	  \begin{subfigure}{0.35\textwidth}
  \includegraphics[width=\textwidth]{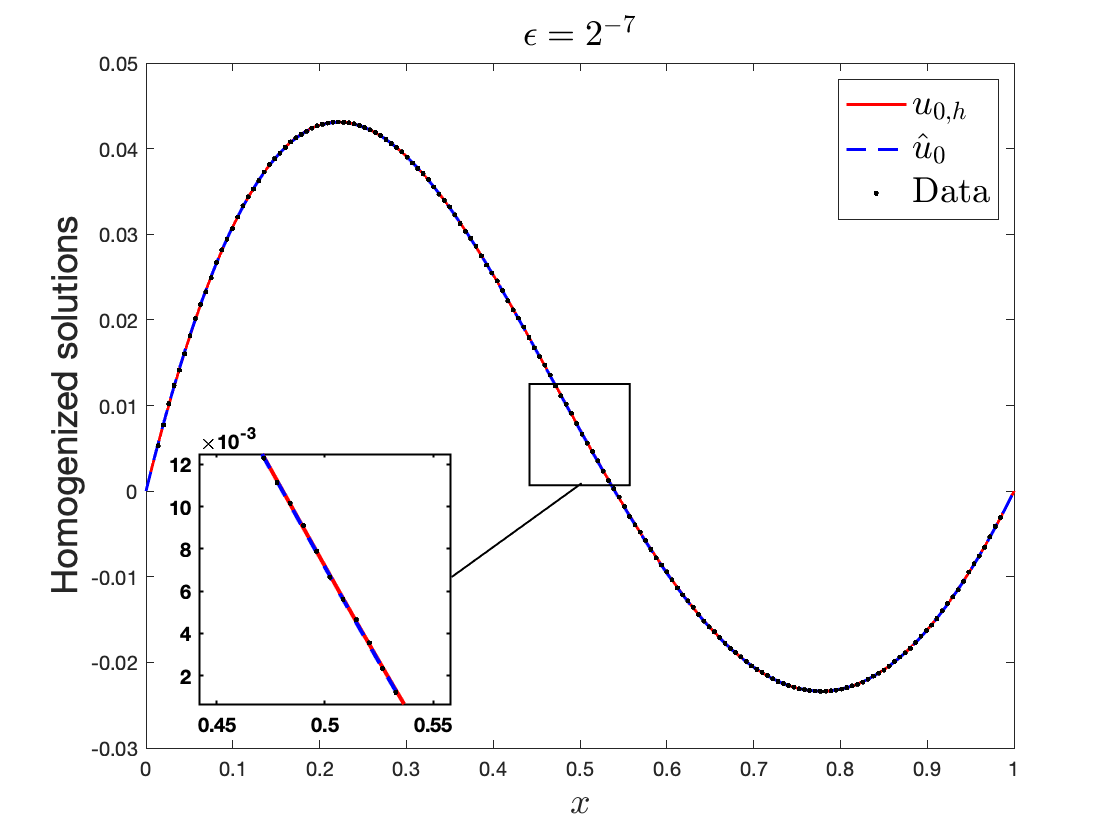}
   \caption{Solutions ($\ep = 2^{-7}$)}
   \label{fig:1dlocper_plot_data_homsols_ums1}
 \end{subfigure}
  \vspace*{-4mm}
  \caption{Problem (\ref{eq:1dlocper}) with noise-free data:  comparison of  the reference solutions ($A^*(x)$ and $u_{0,h}(x)$)  and the counterparts ($\hat{A}^*(x)$ and $\hat{u}_{0}(x)$) learned by PINNs with different values of $\ep$. (a. b. c.):  the  G-limit; (d. e. f.):  the homogenized solution, when the number of  multiscale  data and PDE residual points are $|{\mathcal T}_d| = 160$, $|{\mathcal T}_r| = 190$.}
  \label{figure:1dlocper_plot_ums1_ep}
   \vspace*{-4mm}
  \end{figure}

To investigate the impacts of the finescale parameter $\ep$ of the multiscale data on the performance of our algorithm, we plot the relative errors for the G-limit and the homogenized solution with different values of $\ep$ in Figure \ref{figure:glimitsolepns1}. 
For each $\epsilon$ value,   $|{\mathcal T}_d|=160$ multiscale solution data  collected at fixed spatial locations are used.
In noise-free cases, the errors for homogenized solutions tend to decrease when $\ep$ becomes smaller.
 This is expected as  multiscale solution data are closer to the homogenized solution for smaller $\ep$. 

 Figure \ref{figure:1dlocper_plot_ums1_ep} further shows the learned G-limit and homogenized solution with different finescale parameters $\ep$.
 We can observe that the multiscale data converge to the reference homogenized solution as $\ep$ becomes smaller. 
For example, when $\ep = 2^{-7}$, our data almost overlap with the reference homogenized solution (Figure \ref{fig:1dlocper_plot_data_homsols_ums1}) and both G-limit and the homogenized solution learned by PINNs agree very well with their references. 
Furthermore, despite the presence of noticeable multiscale oscillations  in the data, PINNs can still provide reasonably good results for larger epsilons ($\ep = 2^{-3}, \ 2^{-5}$).  This is because the proposed PINN tends to promote the smooth macroscale behavior of the data rather than their microscale fluctuations shown in Figure \ref{fig:1dlocper_plot_data_homsols_ums1_ep3} and \ref{fig:1dlocper_plot_data_homsols_ums1_ep5}.

For noisy scenarios, the approximation quality deteriorates with the noise level as seen in Figure \ref{figure:glimitsolepns1}. 
We also observe that the impact of the finescale size $\ep$ of the medium  becomes negligible once the noise level is large enough, suggesting that the magnitude of the noises is dominant over the multiscale oscillations in our data. 
Nonetheless, our approach can still provide reasonably good approximations under a mild noise level. 

\subsection{Homogenization of a heavily oscillatory coefficient}
Next, we consider the following  elliptic equation with a heavily oscillatory permeability coefficient introduced in \cite{floden2009g}:
\beq
\label{eq:1nonpermulti}
\bsp
-\frac{d}{d x} \bigg(A^\ep(x) \frac{d}{d x} u^\ep(x)\bigg)&= 3+\sin(x)\ \ \textrm{in} \ \ \Omega \in [0,1],\\
 u^\ep(0)&=0, \ u^\ep(1)  = 0,
\end{split}
\eeq
where $A^\ep(x) = \int_Y \left(1 +\frac{1}{2} \sin \left(\left( y + \frac{1}{2\ep}\sin\left(\pi \sqrt{\frac{2}{\ep}}x\right)\right)^2\right)\right) e^{y(1+\sin x)}dy$. 
The coefficient $A^\ep(x)$ is quite oscillatory. 
Figure \ref{figure:perm_nonper} illustrates the multiscale coefficients $A^\ep(x)$ and the effective coefficients $A^*(x)$ for $\ep = 2^{-3}, 2^{-5}$. Due to strong oscillations in the coefficients, direct numerical simulation of this problem is very expensive when the formula for $A^\ep(x)$ is known. 
This homogenization problem is in general challenging: (1) The explicit integral of the multiscale coefficient is not available.
(2) This problem cannot be handled by the traditional homogenization method, such as the two-scale convergence method, because the oscillations in $A^\ep(x)$ cannot be captured by any test functions admissible for the two-scale convergence \cite{allaire1992homogenization}. 
For this example, it can be shown that the analytical G-limit coincides with the weak $L^2$ limit of $A^\ep(x)$ given by $A^*(x) = \frac{e^{(1+\sin x)}-1}{1+\sin x}$ \cite{floden2009g}, but this is not the case in general \cite[Chapter1]{allaire2012shape}. 

\begin{figure}[!htb]
	\centering
	\begin{subfigure}{0.45\textwidth}
  \includegraphics[width=\textwidth]{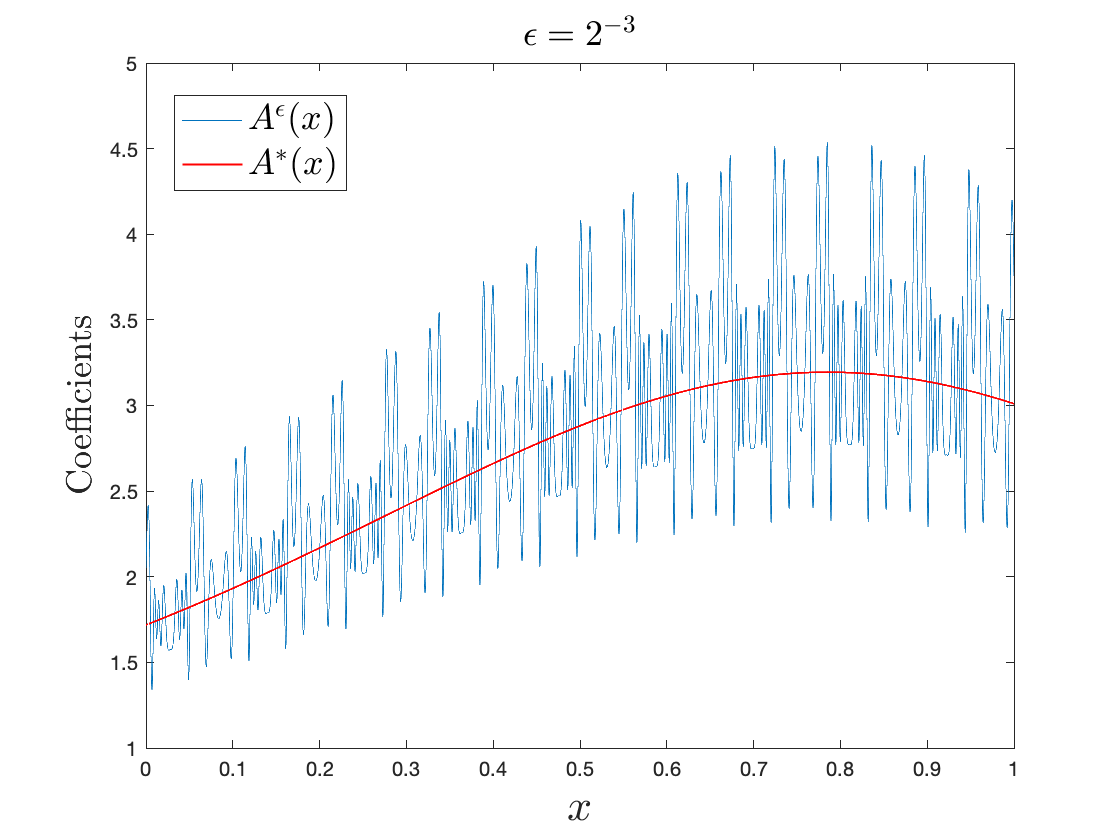}
  \caption{$A^\ep(x)$ and $A^*(x)$, $\ep = 2^{-3}$}
  \label{perm1_nonper}
\end{subfigure}
\hfill
  \begin{subfigure}{0.45\textwidth}
  \includegraphics[width=\textwidth]{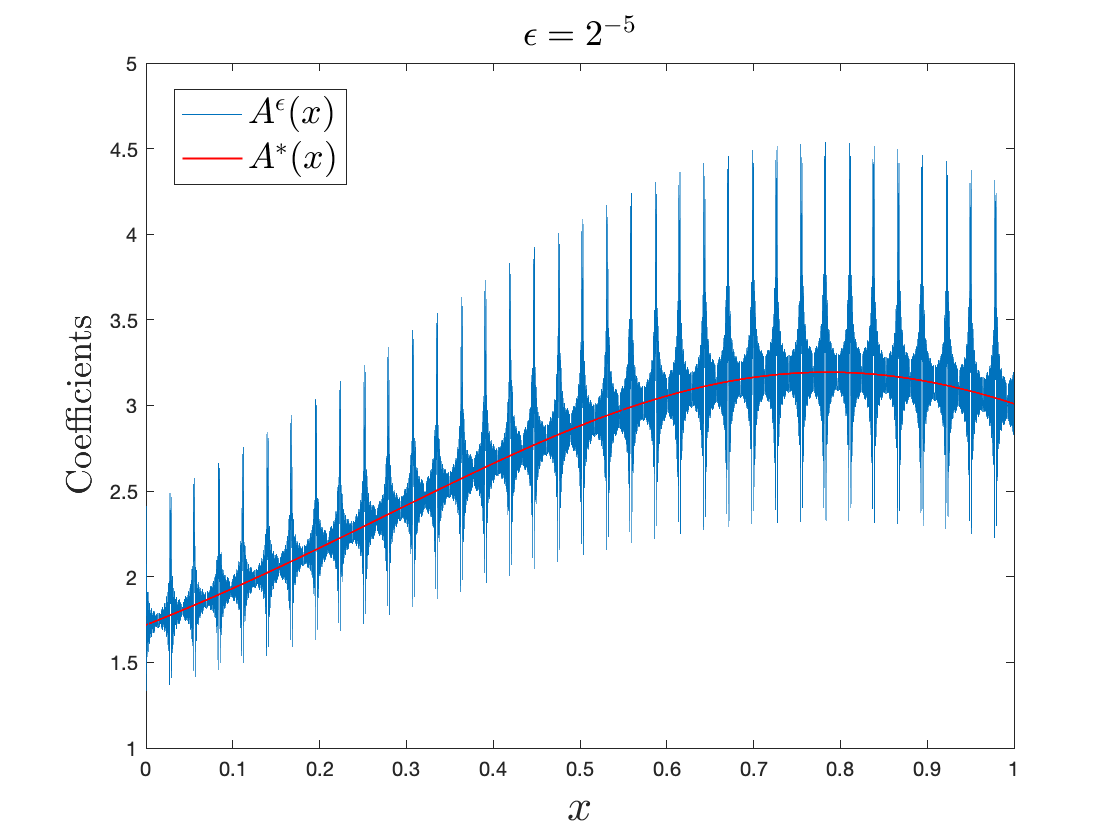}
  \caption{$A^\ep(x)$ and $A^*(x)$, $\ep = 2^{-5}$}
  \label{perm2_nonper}
 \end{subfigure}
 \caption{The G-limits $A^*(x)$ and multiscale coefficients $A^\ep(x)$ with $\ep = 2^{-3}$ and $\ep = 2^{-5}$ for problem (\ref{eq:1nonpermulti})}
  \label{figure:perm_nonper}
 \end{figure}

The synthetic training data set are equally spaced  sampled  from the multiscale solution for each finescale parameter value $\epsilon$ 
computed by FEM with a mesh size of $h = 1/10^5$. 
The reference homogenized solution is computed by FEM with the same mesh 
based on the analytic G-limit. The architecture parameters and other hyperparameters of PINNs are listed in {\it Table \ref{tb:hyperparameters}} in appendix.
The relative $L^2$ errors for both G-limit and homogenized solution are computed based on  the same mesh aforementioned.

We first consider the case with  a relatively small finescale size $\ep = 2^{-7}$.
Figure \ref{figure:glimitsoltdns2} presents the relative $L^2$ errors for the G-limit and homogenized solution with respect to the number of multiscale solution data. With noise-free data, we can achieve an error level of   ${\mathcal O}(10^{-3})$ for both homogenized coefficient and the homogenized solution. It appears that  $80$ multiscale data are enough to obtain  good approximations. 
With a high noise level, 
PINNs can still achieve satisfactory approximations when the data set is large enough. 
This can be further supported by  the corresponding G-limit and homogenized solution obtained by PINNs 
under different levels of noise corruptions in Figure \ref{figure:1dnonper_plots_ns3}. It is clear that the proposed method can still capture the G-limit and the smooth homogenized solution accurately  under mild noise corruptions. This can be further evidenced  by  Figure \ref{fig:1dnonper_plot_data_homsols_ums1_ns1} and \ref{fig:1dnonper_plot_data_homsols_ums1_ns3} where the learned homogenized solutions tend to fit the macroscopic behavior of the noisy data that is close to the reference solution.

\begin{figure}[!htb]
	\centering
	\begin{subfigure}{0.40\textwidth}
  \includegraphics[width=\textwidth]{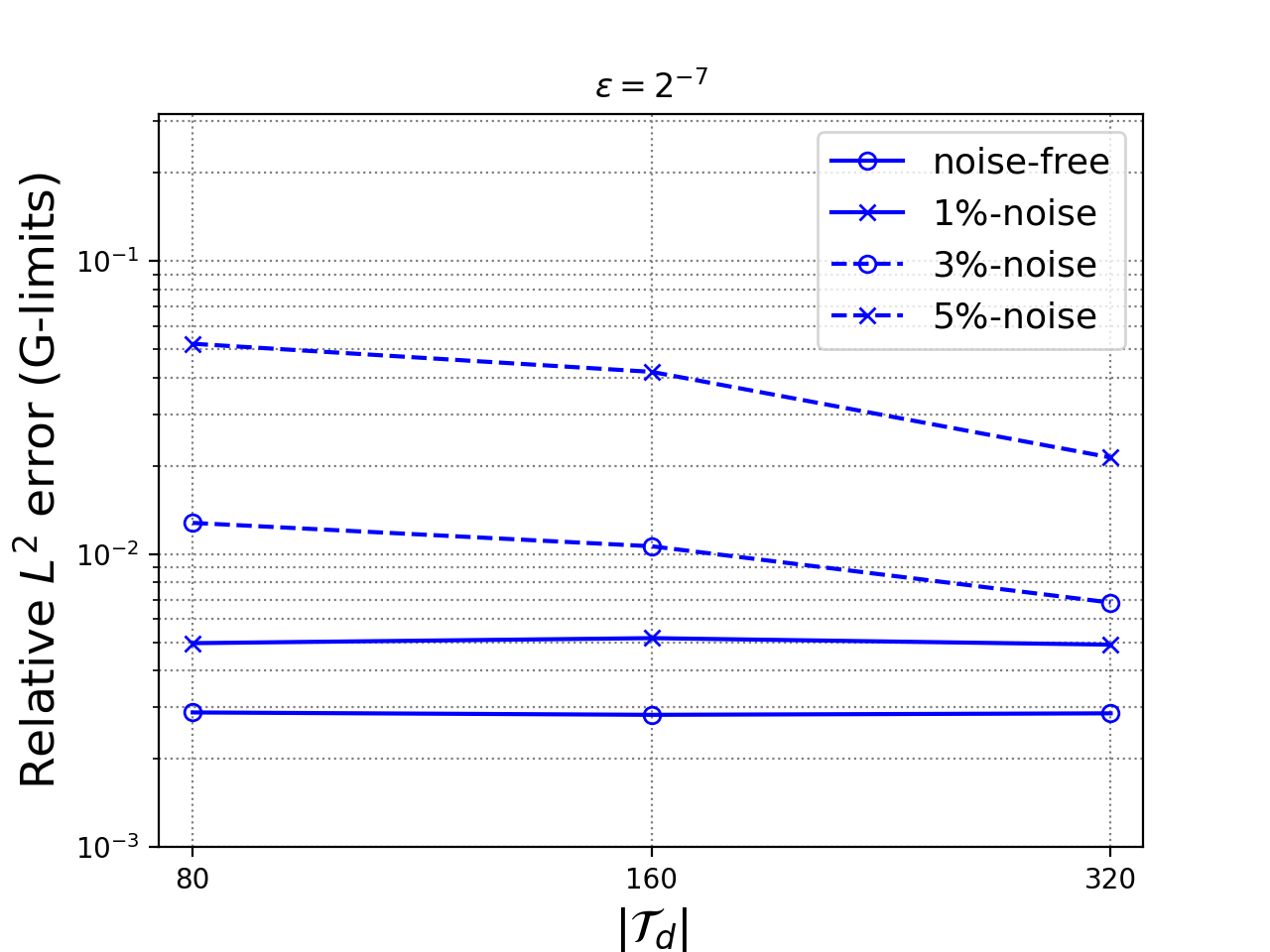}
  \caption{G-limits}
  \label{fig:1dnonper_errors_glimit_nd}
\end{subfigure}
\hspace{.3in}
  \begin{subfigure}{0.40\textwidth}
  \includegraphics[width=\textwidth]{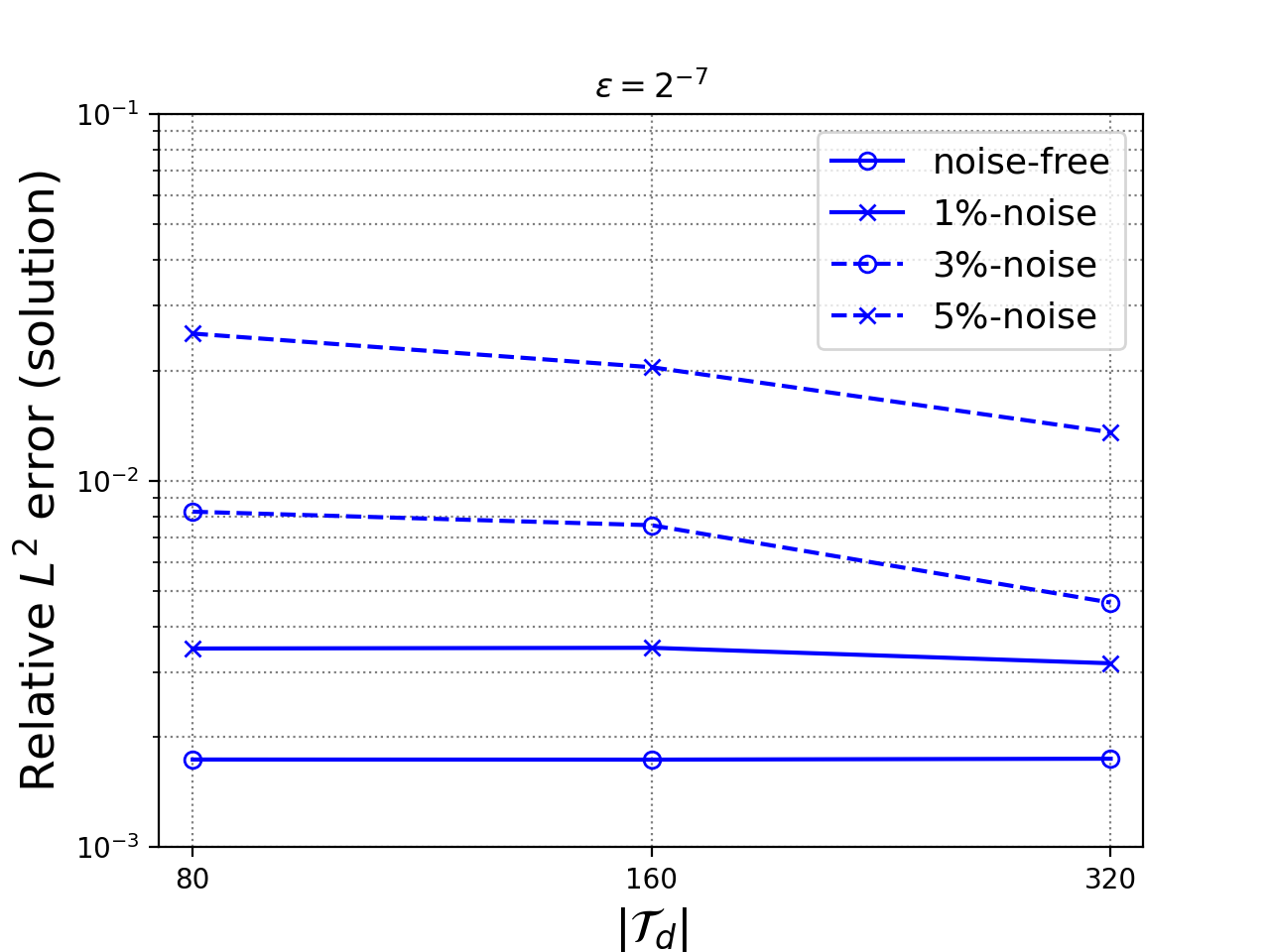}
  \caption{Homogenized solutions}
  \label{fig:1dnonper_errors_homsol_nd}
 \end{subfigure}
 \caption{Error results for problem (\ref{eq:1nonpermulti}): the relative $L^2$ errors for the G-limits  and the homogenized solutions  with different number of multiscale data  corrupted by different noise levels, when  $\ep = 2^{-7}$ and the number of PDE residual points is $|{\mathcal T}_r| = |{\mathcal T}_d| +30$.}
\label{figure:glimitsoltdns2}
\end{figure}

\begin{figure}[!htb]
	\centering
   \begin{subfigure}{0.35\textwidth}
  \includegraphics[width=\textwidth]{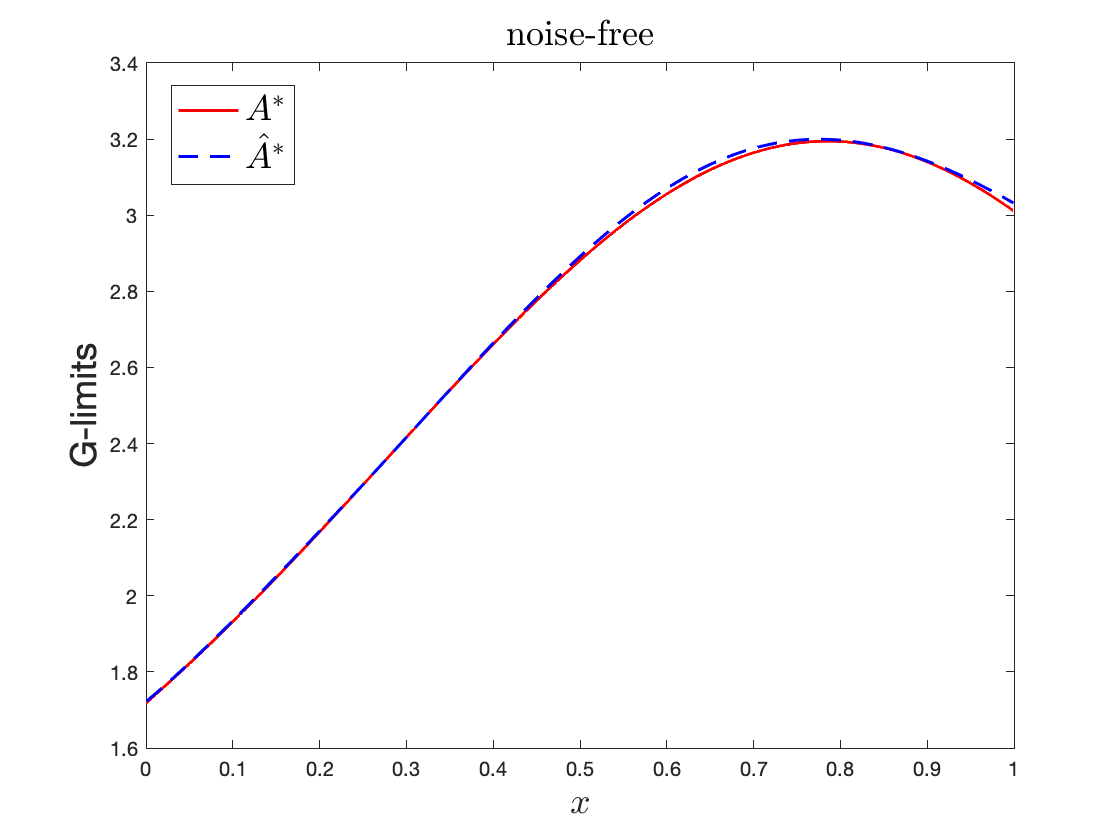}
  \caption{G-limits (noise-free)}
  \label{fig:1dnonper_plot_kappa_ums11}
\end{subfigure}
\hspace{-.25in}
		\begin{subfigure}{0.35\textwidth}
  \includegraphics[width=\textwidth]{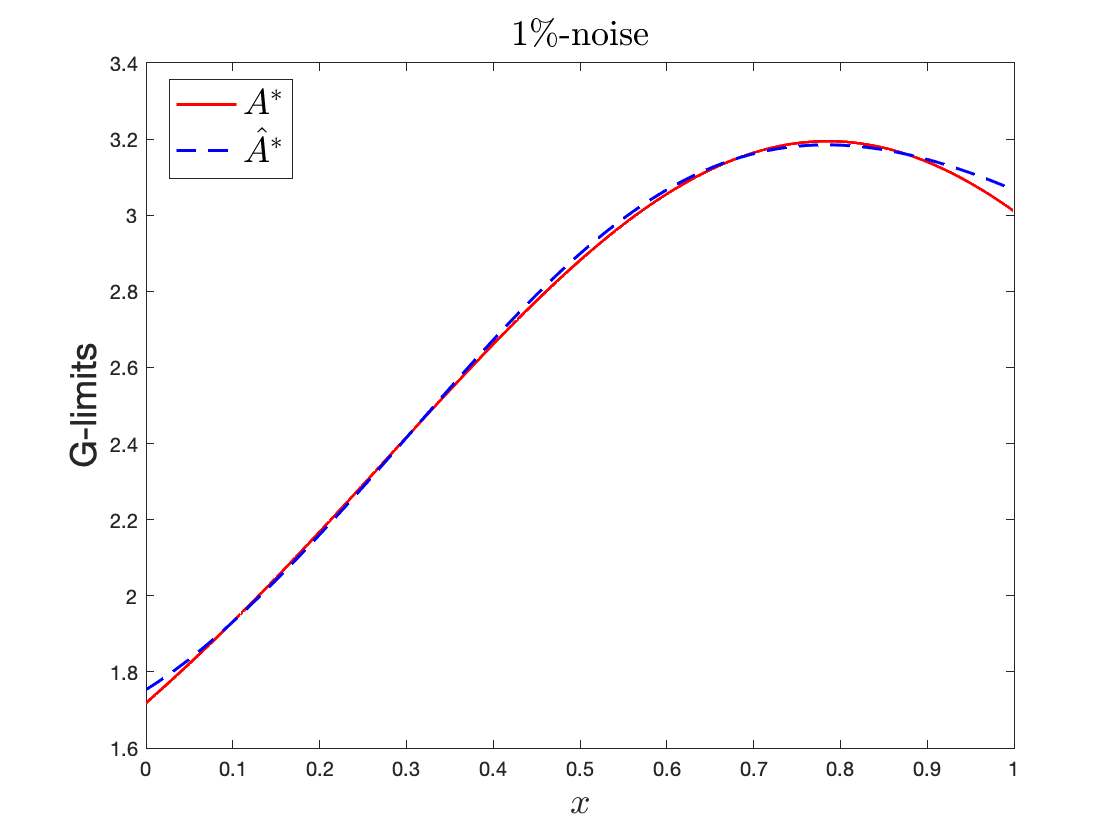}
  \caption{G-limits ($1\%$ noise) }
  \label{fig:1dnonper_plot_kappa_ums1_ns1}
\end{subfigure}
\hspace{-.25in}
	\begin{subfigure}{0.35\textwidth}
  \includegraphics[width=\textwidth]{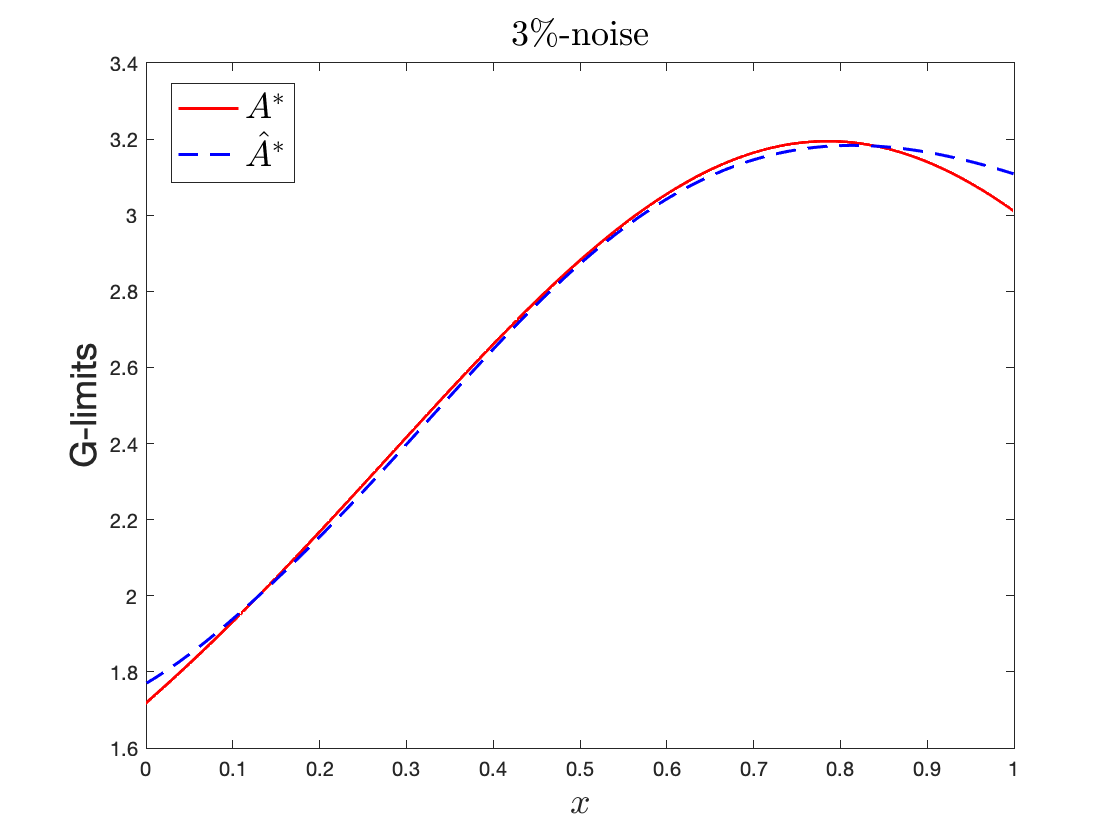}
  \caption{G-limits ($3\%$ noise)}
  \label{fig:1dnonper_plot_kappa_ums1_ns3}
\end{subfigure}
\hspace{-.25in}
  \begin{subfigure}{0.35\textwidth}
  \includegraphics[width=\textwidth]{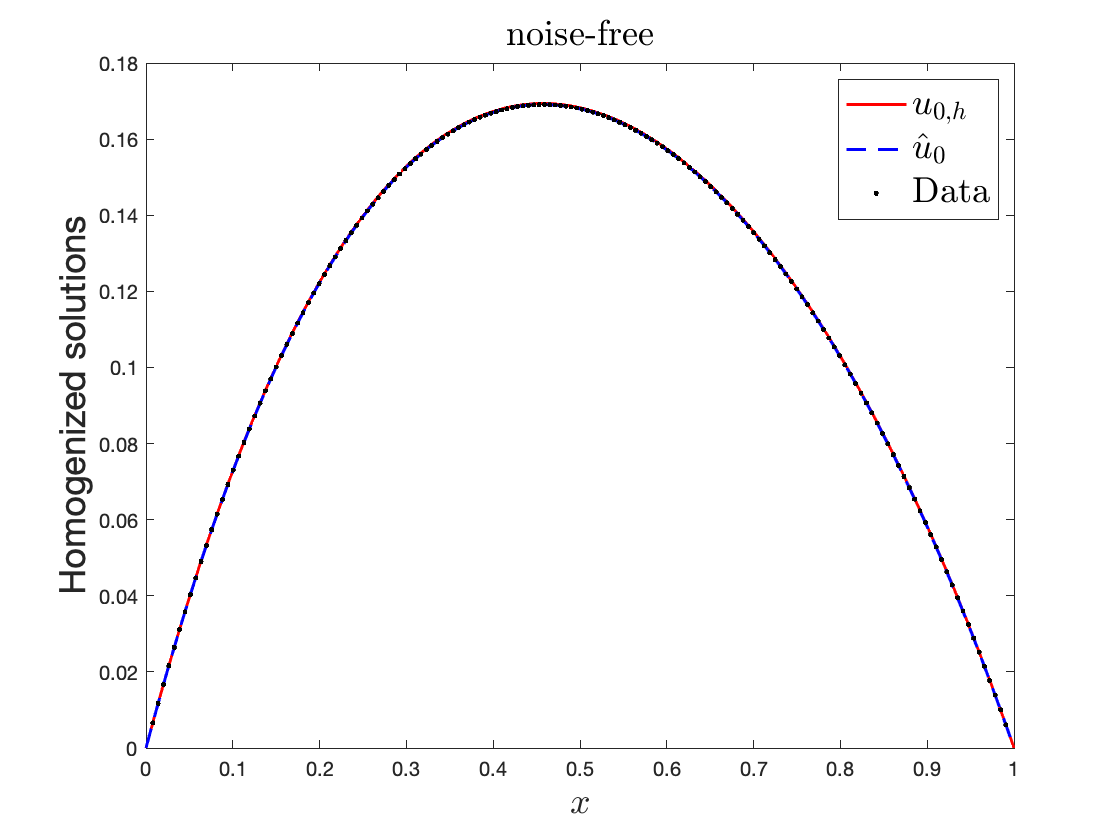}
  \caption{Solutions (noise-free)}
  \label{fig:1dnonper_plot_data_homsols_ums11}
 \end{subfigure}
 \hspace{-.25in}
  \begin{subfigure}{0.35\textwidth}
  \includegraphics[width=\textwidth]{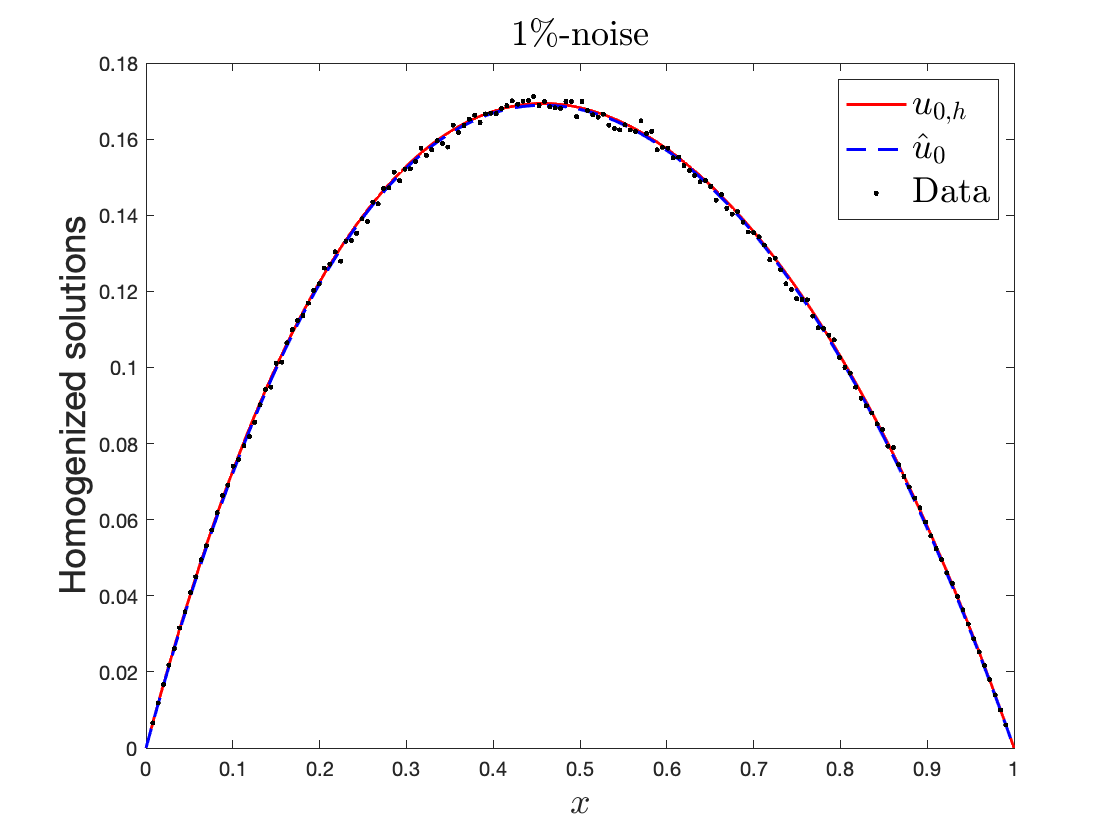}
  \caption{Solutions ($1\%$ noise)}
  \label{fig:1dnonper_plot_data_homsols_ums1_ns1}
 \end{subfigure}
\hspace{-.25in}
  \begin{subfigure}{0.35\textwidth}
  \includegraphics[width=\textwidth]{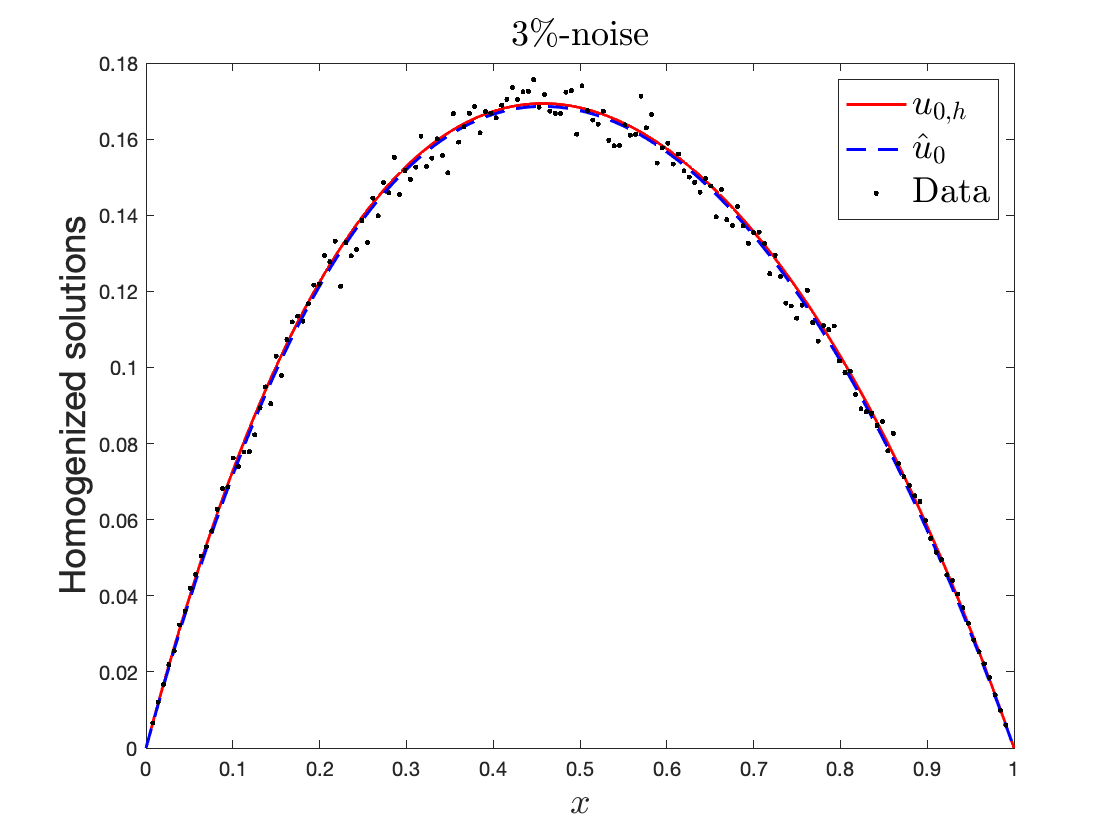}
  \caption{Solutions ($3\%$ noise)}
  \label{fig:1dnonper_plot_data_homsols_ums1_ns3}
 \end{subfigure}
 \vspace*{-4mm}
 \caption{Problem (\ref{eq:1nonpermulti}):  comparison of  the reference solutions ($A^*(x)$ and $u_{0,h}(x)$)  and the counterparts  ($\hat{A}^*(x)$ and $\hat{u}_{0}(x)$) learned by PINNs  with different noise levels in the data. (a. b. c.):  the  G-limit; (d. e. f.):  the homogenized solution, when $\ep = 2^{-7}$ and the number of  multiscale  data and PDE residual points are $|{\mathcal T}_d| = 160$, $|{\mathcal T}_r| = 190$.}
\label{figure:1dnonper_plots_ns3}
\vspace*{-4mm}
\end{figure}

\begin{figure}[!htb]
	\centering
	\begin{subfigure}{0.40\textwidth}
  \includegraphics[width=\textwidth]{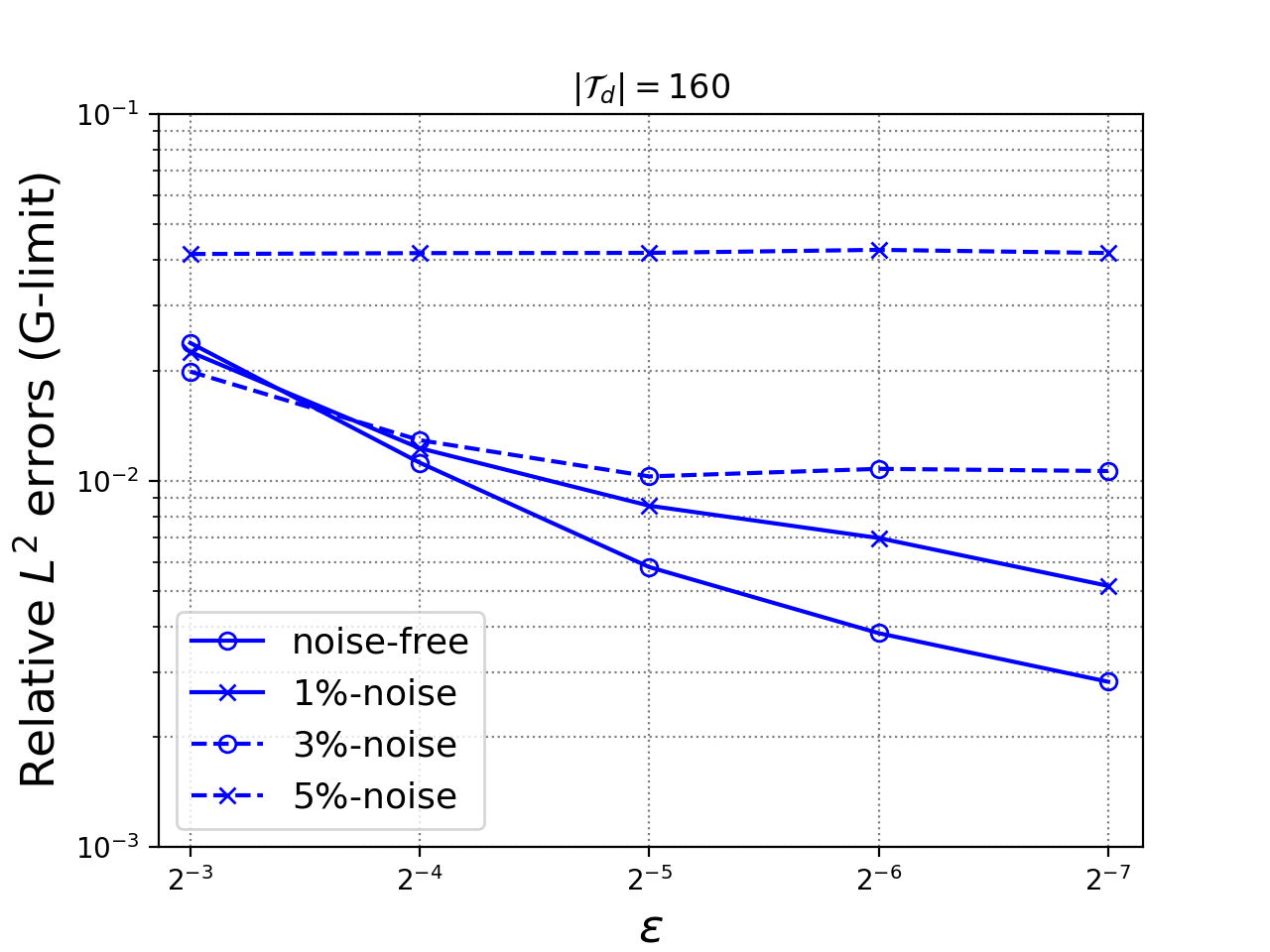}
  \caption{G-limits}
  \label{fig:1dnonper_errors_glimit_ep}
\end{subfigure}
\hspace{.3in}
  \begin{subfigure}{0.40\textwidth}
  \includegraphics[width=\textwidth]{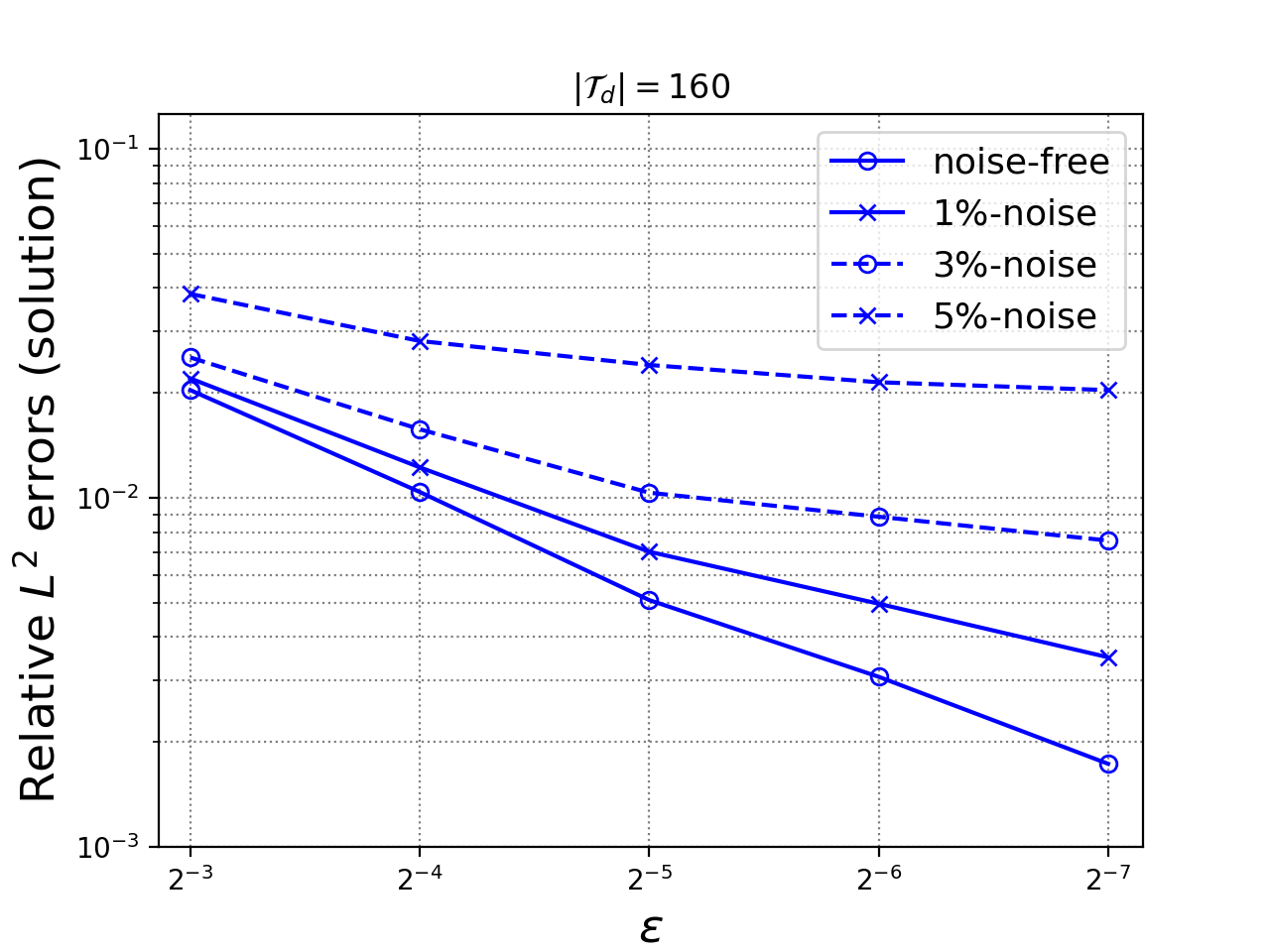}
  \caption{Homogenized solutions}
  \label{fig:1dnonper_errors_homsol_ep}
 \end{subfigure}
  \vspace*{-4mm}
 \caption{Error results for problem (\ref{eq:1nonpermulti}): the relative $L^2$ errors for the G-limits  and the homogenized solutions  with different finescale parameter $\ep$ and noise levels of data, when the number of  multiscale  data and PDE residual points are  $|{\mathcal T}_d| = 160$ and $|{\mathcal T}_r| = 190$.}
\label{figure:glimitsolepns2}
  \vspace*{-5mm}
\end{figure}

  \begin{figure}[!htb]
 	\centering
 	  \begin{subfigure}{0.35\textwidth}
  \includegraphics[width=\textwidth]{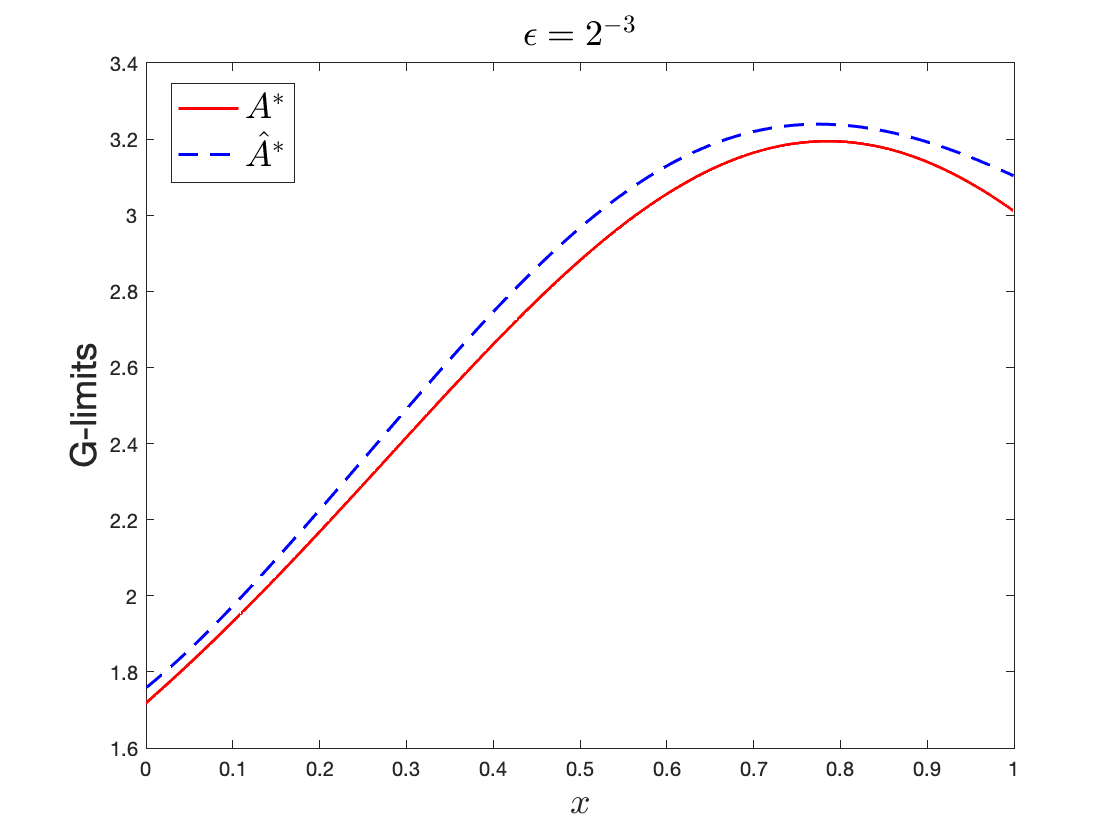}
  \caption{G-limits ($\ep = 2^{-3}$)}
  \label{fig:1dnonper_plot_glimit_ums1_ep3}
 \end{subfigure}
\hspace{-.25in}
  	  \begin{subfigure}{0.35\textwidth}
  \includegraphics[width=\textwidth]{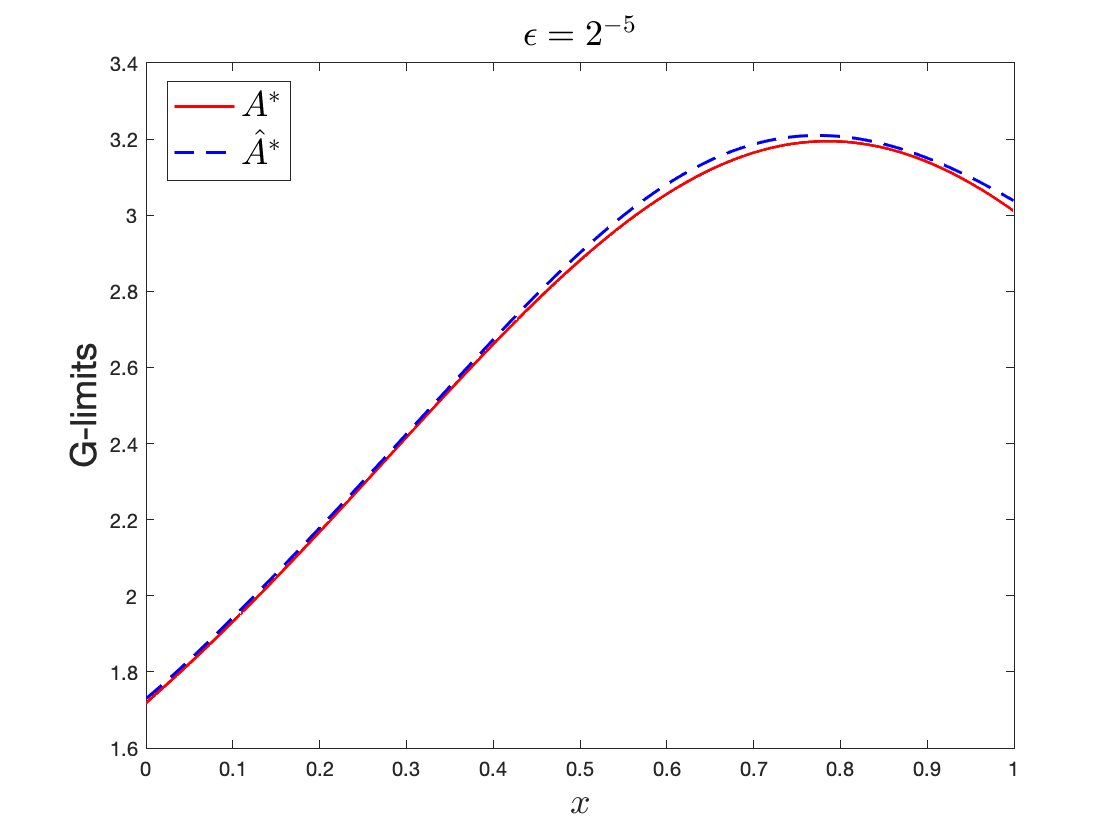}
  \caption{G-limits ($\ep = 2^{-5}$)}
  \label{fig:1dnonper_plot_glimit_ums1_ep5}
 \end{subfigure}
\hspace{-.25in}
   \begin{subfigure}{0.35\textwidth}
  \includegraphics[width=\textwidth]{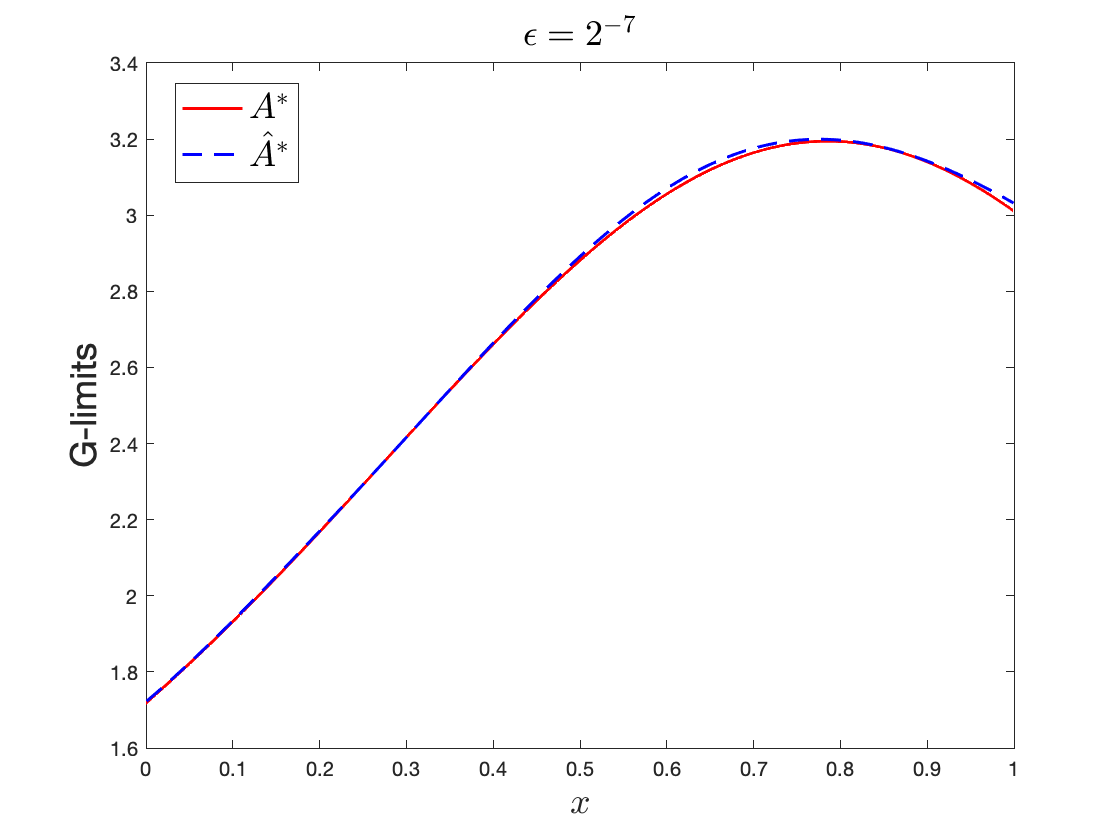}
  \caption{G-limits ($\ep = 2^{-7}$)}
  \label{fig:1dnonper_plot_kappa_ums1}
\end{subfigure}
\hspace{-.25in}
   	  \begin{subfigure}{0.35\textwidth}
  \includegraphics[width=\textwidth]{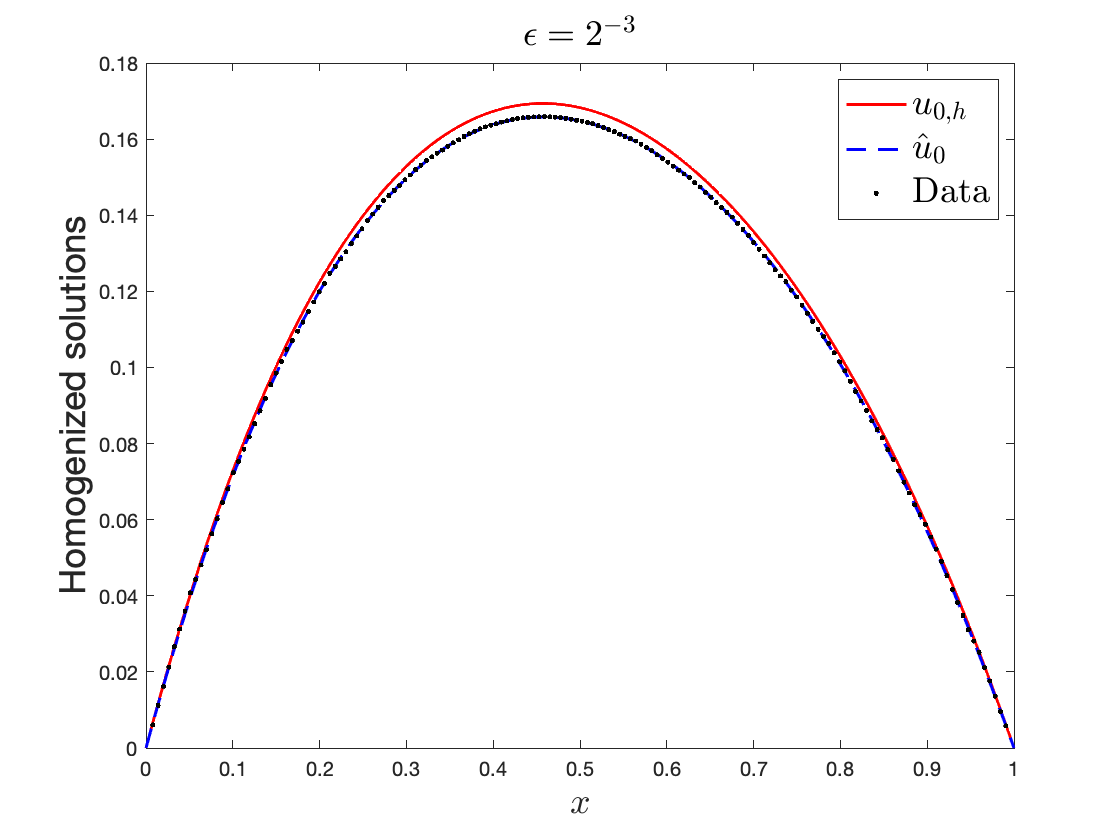}
   \caption{Solutions ($\ep = 2^{-3}$)}
   \label{fig:1dnonper_plot_data_homsols_ums1_ep3}
 \end{subfigure}
\hspace{-.25in}
  	  \begin{subfigure}{0.35\textwidth}
  \includegraphics[width=\textwidth]{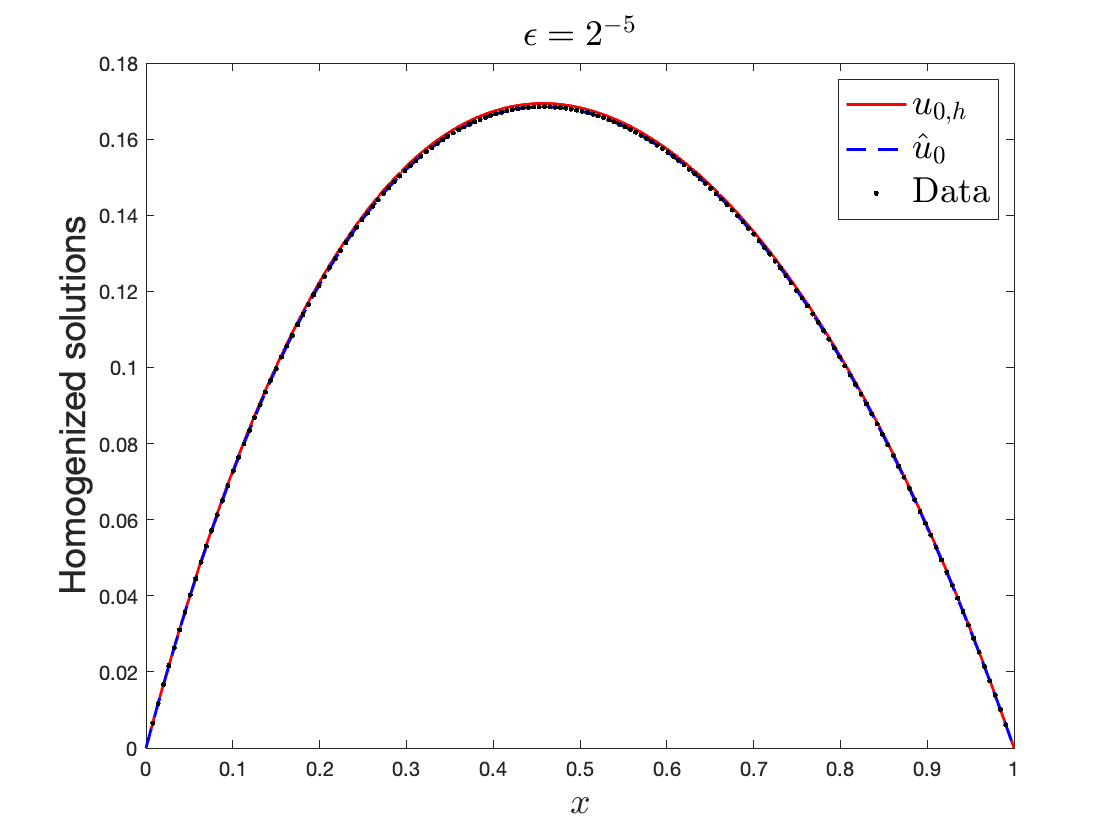}
   \caption{Solutions ($\ep = 2^{-5}$)}
   \label{fig:1dnonper_plot_data_homsols_ums1_ep5}
 \end{subfigure}
\hspace{-.25in}
  \begin{subfigure}{0.35\textwidth}
  \includegraphics[width=\textwidth]{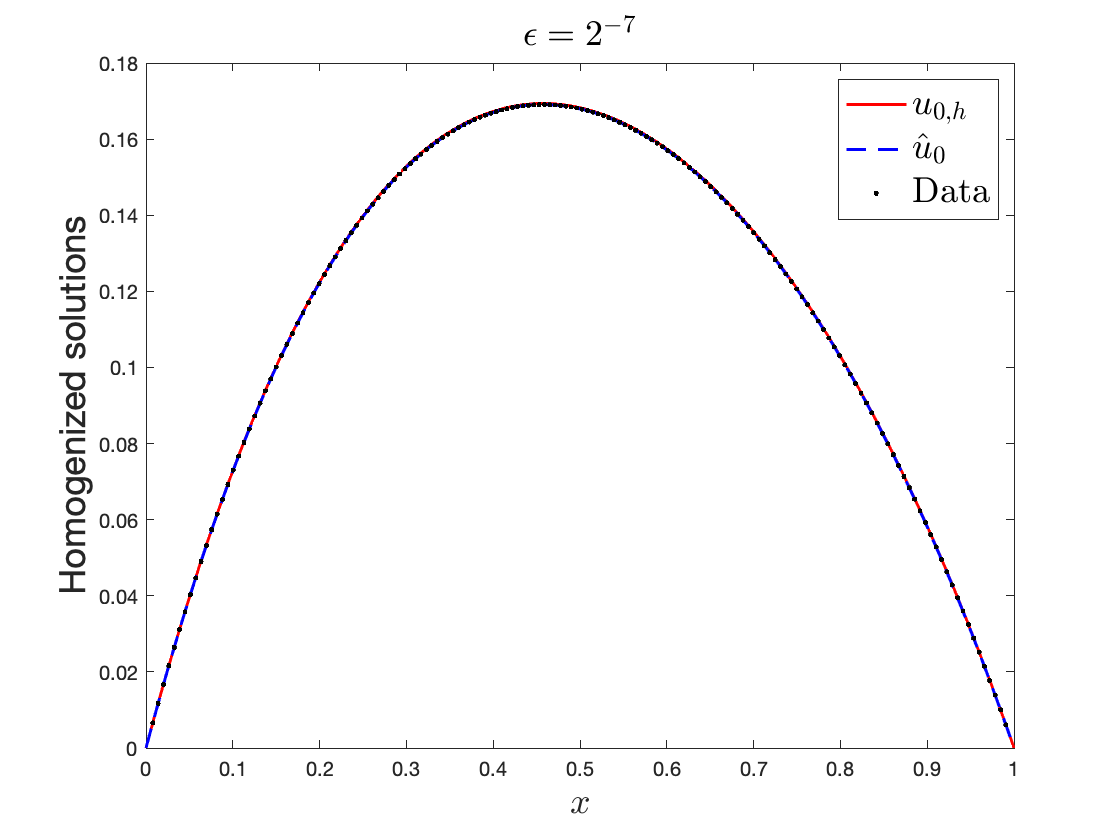}
  \caption{Solutions ($\ep = 2^{-7}$)}
  \label{fig:1dnonper_plot_data_homsols_ums1}
 \end{subfigure}
   \vspace*{-3mm}
  \caption{Problem (\ref{eq:1nonpermulti})  with noise-free data:  comparison of  the reference solutions ($A^*(x)$ and $u_{0,h}(x)$)  and the counterparts  ($\hat{A}^*(x)$ and $\hat{u}_{0}(x)$) learned by PINNs with different values of $\ep$. (a. b. c.):  the  G-limit; (d. e. f.):  the homogenized solution, when the number of  multiscale  data and PDE residual points are $|{\mathcal T}_d| = 160$, $|{\mathcal T}_r| = 190$.}
  \label{figure:1dnonper_plot_ums1_ep}
  \end{figure}
Figure \ref{figure:glimitsolepns2} presents  errors of  the estimated G-limit and the homogenized solution with respect to the finescale parameter $\ep$ and $160$ multiscale solution data collected at fixed spatial locations for all $\epsilon$, i.e., $|{\mathcal T}_d|=160$. As expected, better approximation of the G-limit and the homogenized solution can be delivered as $\ep$ becomes smaller in noise-free case.  
In addition, Figure \ref{figure:1dnonper_plot_ums1_ep} 
again shows that 
the learned homogenized solutions tend to fit the multiscale solution data. 
We note that even if the finescale oscillations in our data are not visible in the figures, a relatively large  $\ep$  (=$2^{-3}$, $2^{-5}$) could result in the non-negligible difference between the reference homogenized solution and our data. As a result,  approximations of G-limits and homogenized solutions are less accurate but satisfactory for larger $\epsilon$. In contrast to noise-free scenarios, the finescale size $\ep$ has much less impact on  both learned G-limit and the homogenized solution  as the noise level increases as shown in Figure \ref{figure:glimitsolepns2}. 
\subsection{Homogenization of a 2D non-periodic coefficient}
We next consider the following 2D multiscale elliptic equation with a non-periodic coefficient introduced in \cite{persson2012selected}:
\beq
\label{eq:original_nonper_2D}
\begin{split}
-\div \bigg( A^\ep(x) \cdot \nabla u^\ep(x) \bigg) &= 1 \ \ \textrm{in} \ \ \Omega = [1,2]^2\\
u^\ep(x) &= 0 \ \ \textrm{on} \ \ \partial\Omega,
\end{split}
\eeq
where $A^\ep(x) = \left(1+0.9\sin(2\pi\frac{x_1}{\ep})\sin(2\pi \frac{x_2^2}{\ep})\right)$. Figure \ref{figure:Perm_2Dnonper} illustrates the multiscale coefficient $A^\ep(x)$ when $\ep = 2^{-3}$. 
The G-limit $A^*= \big(\begin{smallmatrix} A_{11}(x) & A_{12}(x)\\ A_{21}(x) & A_{22}(x) \end{smallmatrix}\big)$ that is a $2\times2$ matrix function, can be found via the $\lambda$-scale convergence technique \cite{persson2012selected}. 
Since $A^\ep(x)$ is periodic with respect to $x_1$, we know that the G-limit only depends on $x_2$, i.e., $A^*(x) = A^*(x_2)$. 
We assume a priori that the non-diagonal entries of the G-limit are zeros
i.e., $A_{12}(x_2) = A_{21}(x_2) = 0$.
Therefore, we shall only approximate the diagonal entries of the G-limit.
The reference G-limit is computed by the $\lambda$-scale convergence method. 
Specifically, we solved local cell problems at $129$  equidistant points of $x_2$ and each problem is solved by FEM with a mesh size of $h = 1/1000$. With the reference G-limit, we computed the reference homogenized solution  by FEM with a mesh size $h=1/128$.

\begin{figure}[!t]
	\centering
  \includegraphics[scale =0.15]{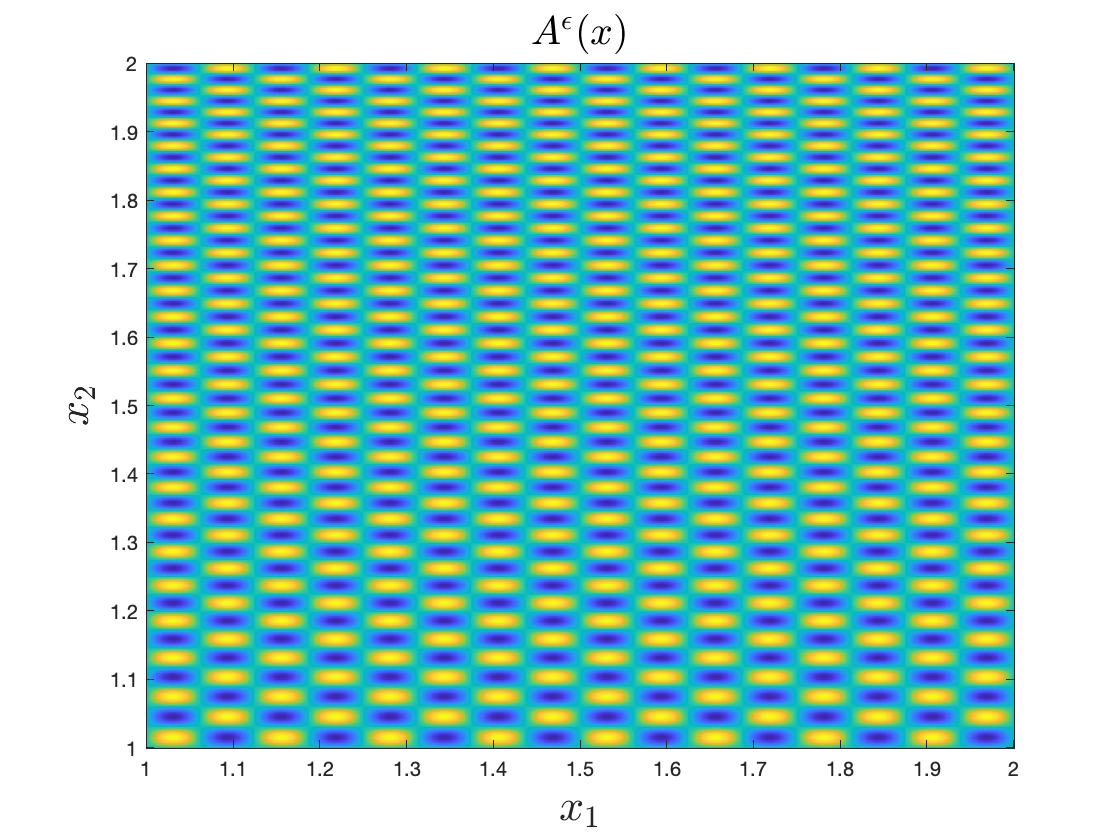}
  \caption{The multiscale permeability coefficient $A^\ep(x)$ in (\ref{eq:original_nonper_2D}) when $\ep = 2^{-3}$.}
  \label{figure:Perm_2Dnonper}
 \end{figure}

The training data are equally spaced sampled from the multiscale solution for each finescale parameter value $\epsilon$ obtained by the forward FEM simulation of the problem (\ref{eq:original_nonper_2D}) with a fine mesh size $1/8000$. The architecture parameters and other hyperparameters of PINNs are listed in {\it Table \ref{tb:hyperparameters}} in appendix. 
We compute the errors on a mesh with size $h=1/128$ in the spatial domain.

Figure \ref{figure:glimitsoltdns3_sc} presents  the error convergence of the G-limit and the homogenized solutions for $\ep = 2^{-7}$ with different numbers of training data.  
With noise-free data,  $400$ data points appear to be sufficient to obtain good approximations with errors of less than $10^{-3}$ for both G-limit and homogenized solution. Increasing the amount of training data helps improve the accuracy for noisy data cases.
Even with $5\%$-noise in the data, our proposed method can still achieve an error less than ${\mathcal O}(10^{-2})$ for both coefficient and the solution, when the number of available data is large enough. 
We also plot the G-limits and the homogenized solutions at $x_2 = 1.25$ shown in 
Figure \ref{figure:2dnonper_plots_ns} 
when $\ep = 2^{-7}$. Both diagonal entries of the G-limit and homogenized solution agree well with their references. 
Again, we observe that the learned solutions tend to fit the macroscale behaviors of the data, even when non-negligible noises present in \ref{figure:2dnonper_plot_data_homsols_ums1_ns1_sc} and \ref{figure:2dnonper_plot_data_homsols_ums1_ns3}.

\begin{figure}[!htb]
	\centering
	  \vspace*{3mm}
	\begin{subfigure}{0.40\textwidth}
  \includegraphics[width=\textwidth]{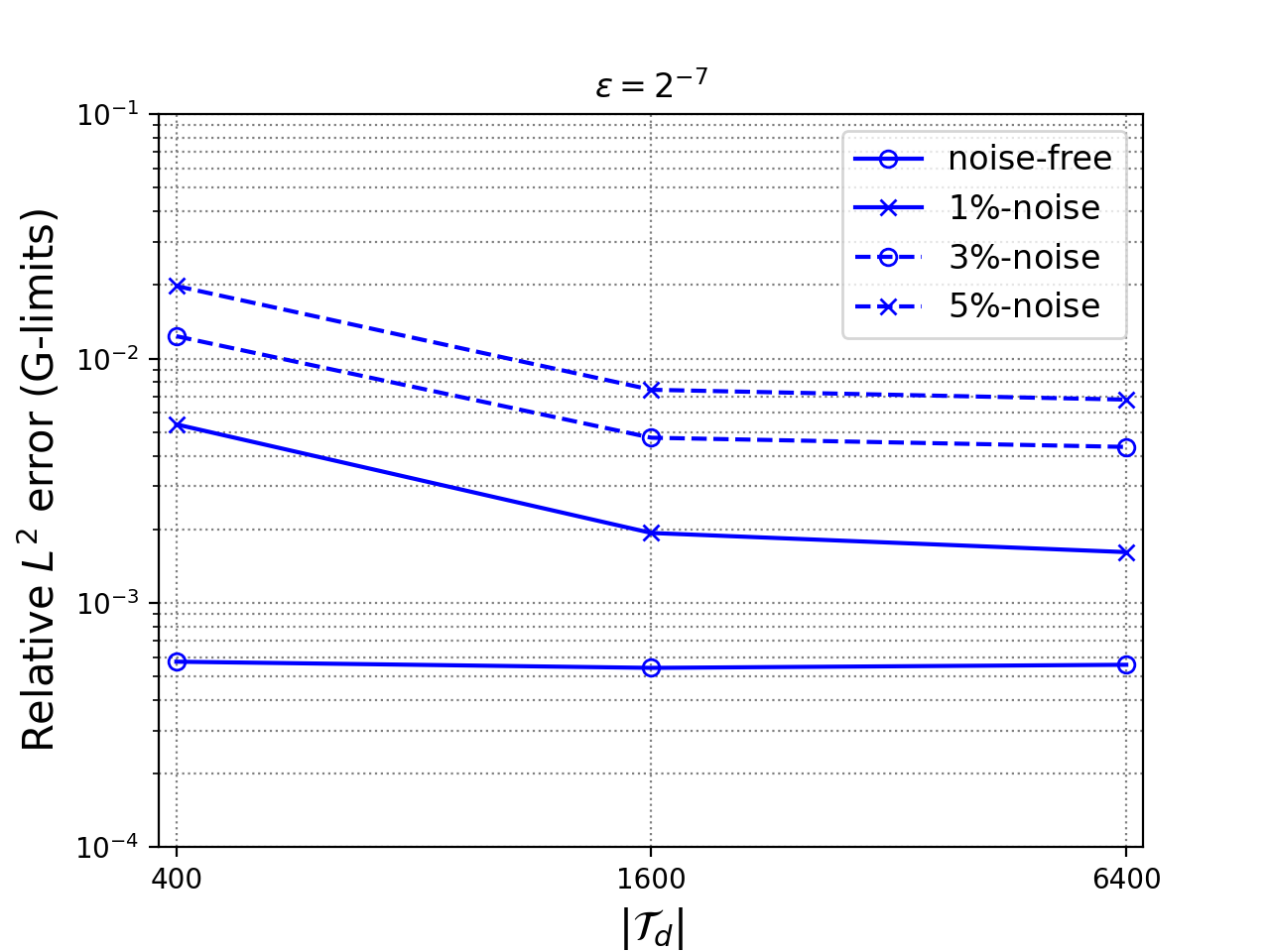}
  \caption{G-limits}
  \label{fig:2dnonper_errors_glimit_nd_sc}
\end{subfigure}
\hspace{.3in}
  \begin{subfigure}{0.40\textwidth}
  \includegraphics[width=\textwidth]{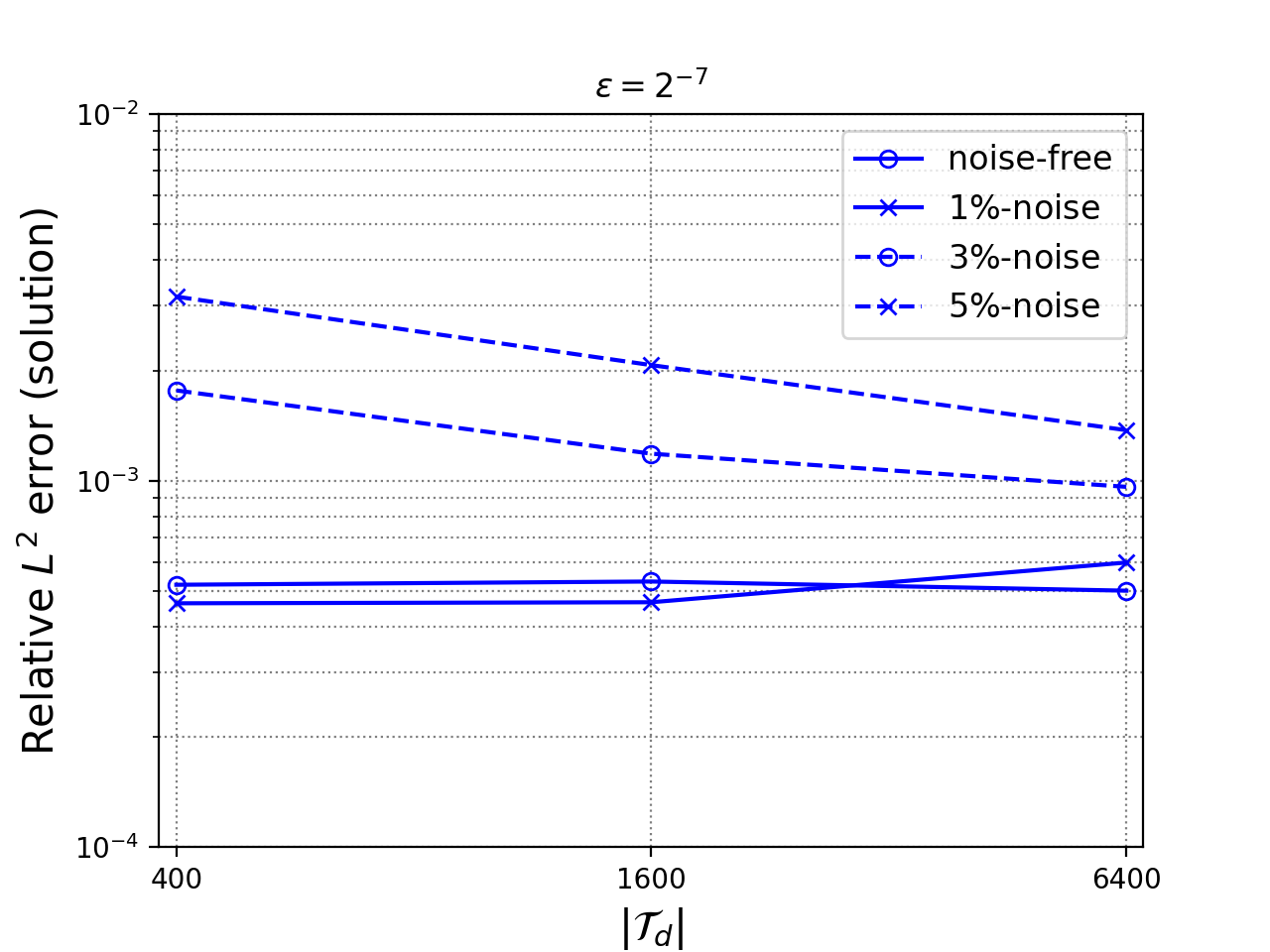}
  \caption{Homogenized solutions}
  \label{fig:2dnonper_errors_homsol_nd_sc}
 \end{subfigure}
 \caption{ Error results for problem (\ref{eq:original_nonper_2D}): the relative $L^2$ errors for the G-limits and the homogenized solutions with different number of multiscale data corrupted by different noise levels, when $\ep = 2^{-7}$ and the number of PDE residual points is $|{\mathcal T}_r| = |{\mathcal T}_d|$.}
\label{figure:glimitsoltdns3_sc}
\end{figure}
 \begin{figure}[!htb]
	\centering
   \begin{subfigure}{0.35\textwidth}
  \includegraphics[width=1\textwidth]{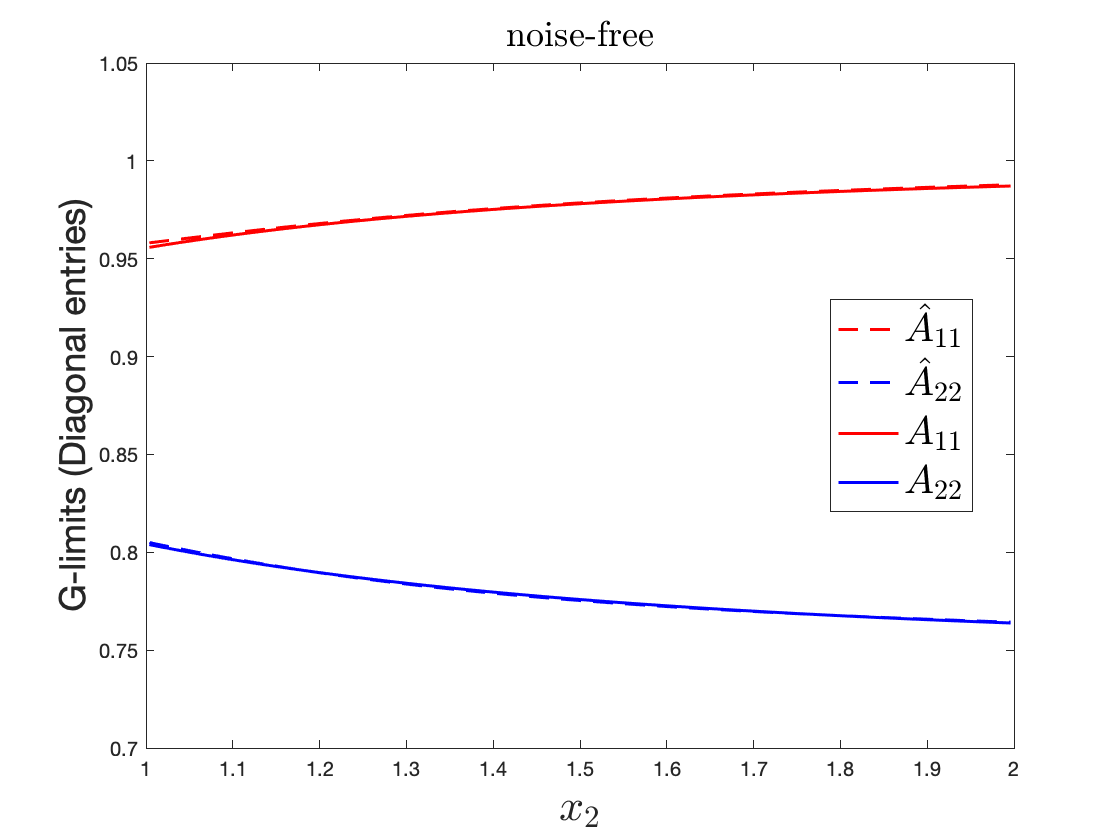}
  \caption{G-limits (noise-free)}
  \label{figure:2dnonper_plot_Glimits_nd40_ep7_mat}
  \end{subfigure}
  	   \hspace{-.25in}
\begin{subfigure}{0.35\textwidth}
  \includegraphics[width=1\textwidth]{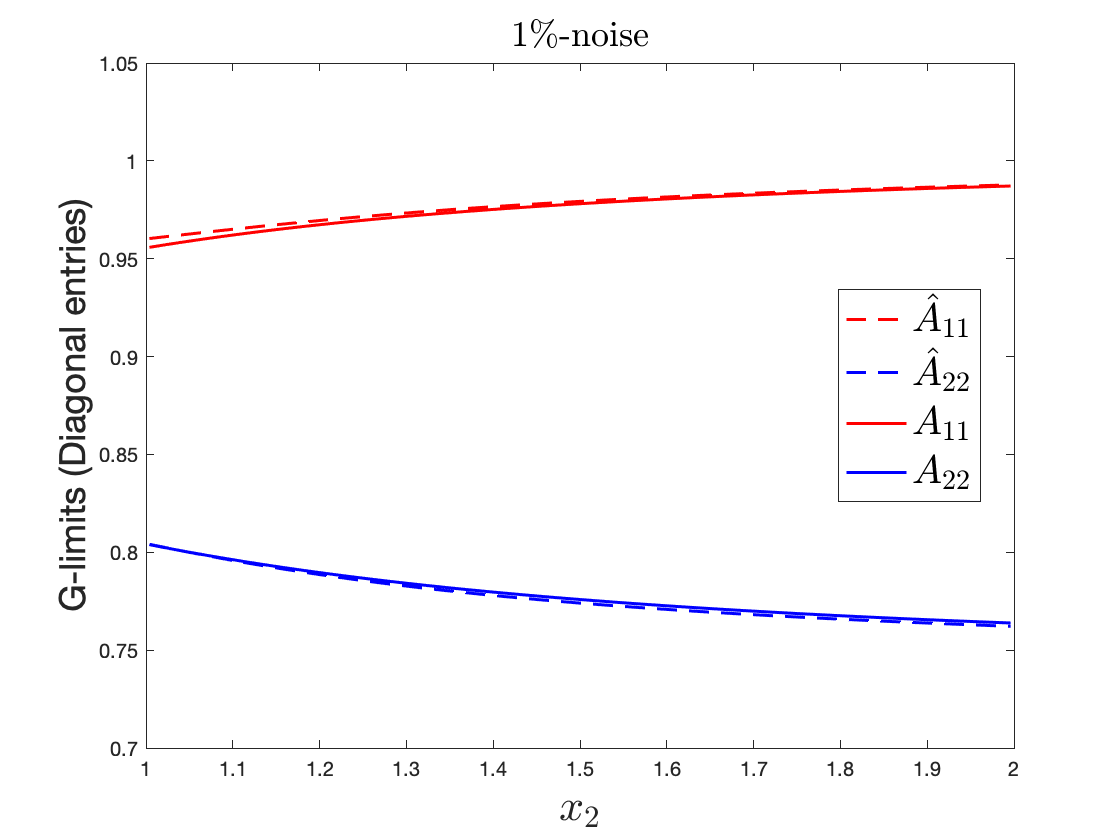}
  \caption{G-limits ($1\%$ noise)}
  \label{figure:2dnonper_plot_Glimits_nd40_ep7_mat_ns1_sc}
  \end{subfigure}
  	   \hspace{-.25in}
  \begin{subfigure}{0.35\textwidth}
  \includegraphics[width=1\textwidth]{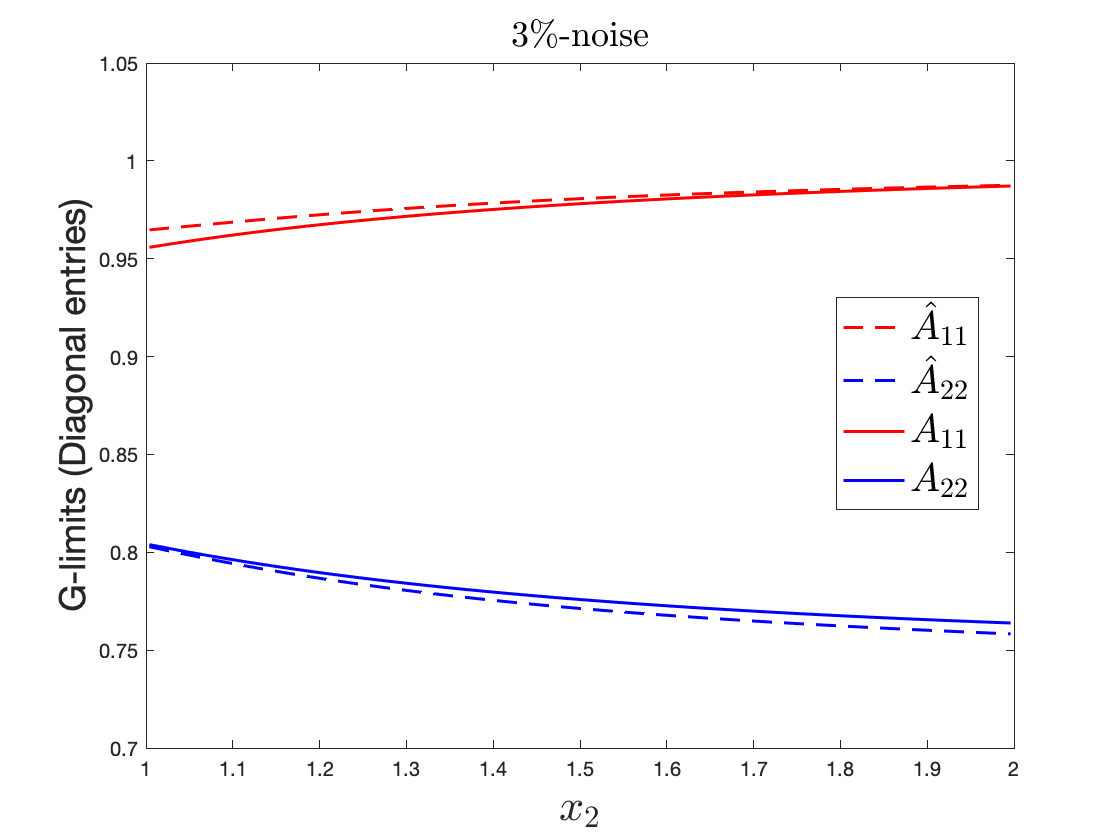}
  \caption{G-limits ($3\%$ noise)}
  \label{figure:2dnonper_plot_Glimits_nd40_ep7_mat_ns3_sc}
  \end{subfigure}
  	   \hspace{-.25in}
    \begin{subfigure}{0.35\textwidth}
  \includegraphics[width=1\textwidth]{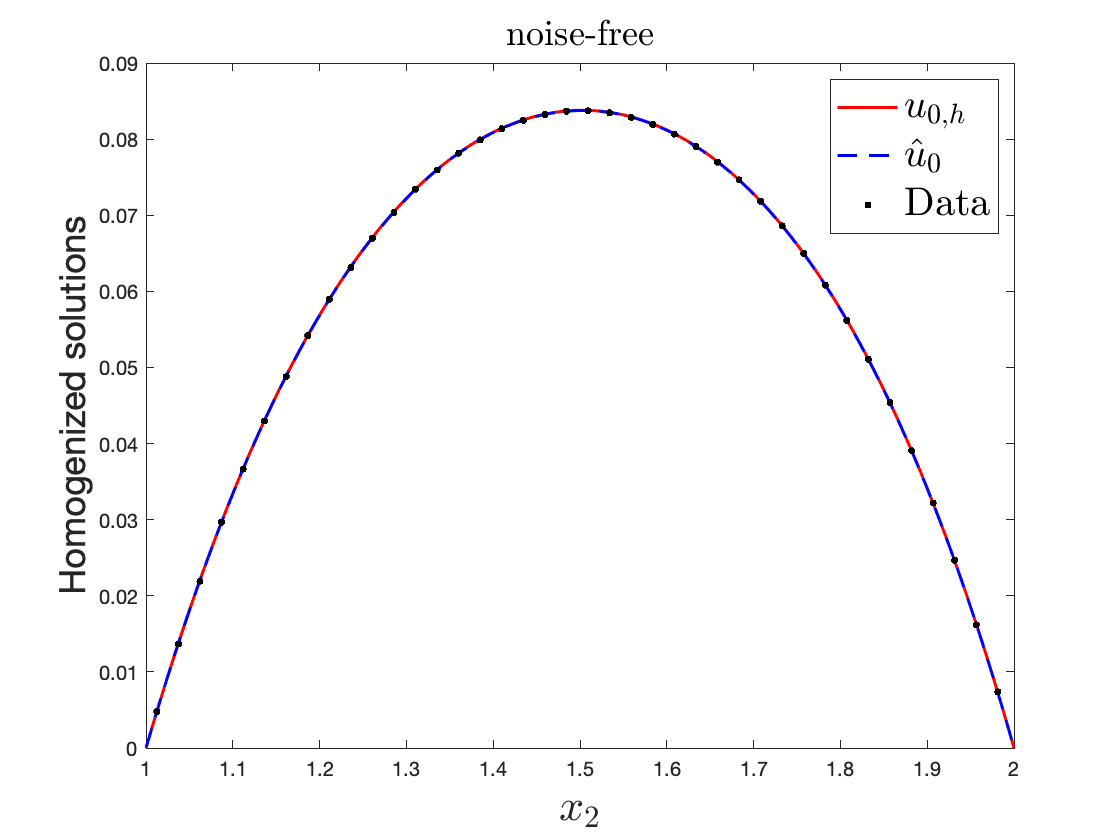}
 \caption{Solutions (noise-free)}
  \label{figure:2dnonper_plot_data_homsols_ums1_ep7}
  \end{subfigure}
  	   \hspace{-.25in}
  \begin{subfigure}{0.35\textwidth}
  \includegraphics[width=1\textwidth]{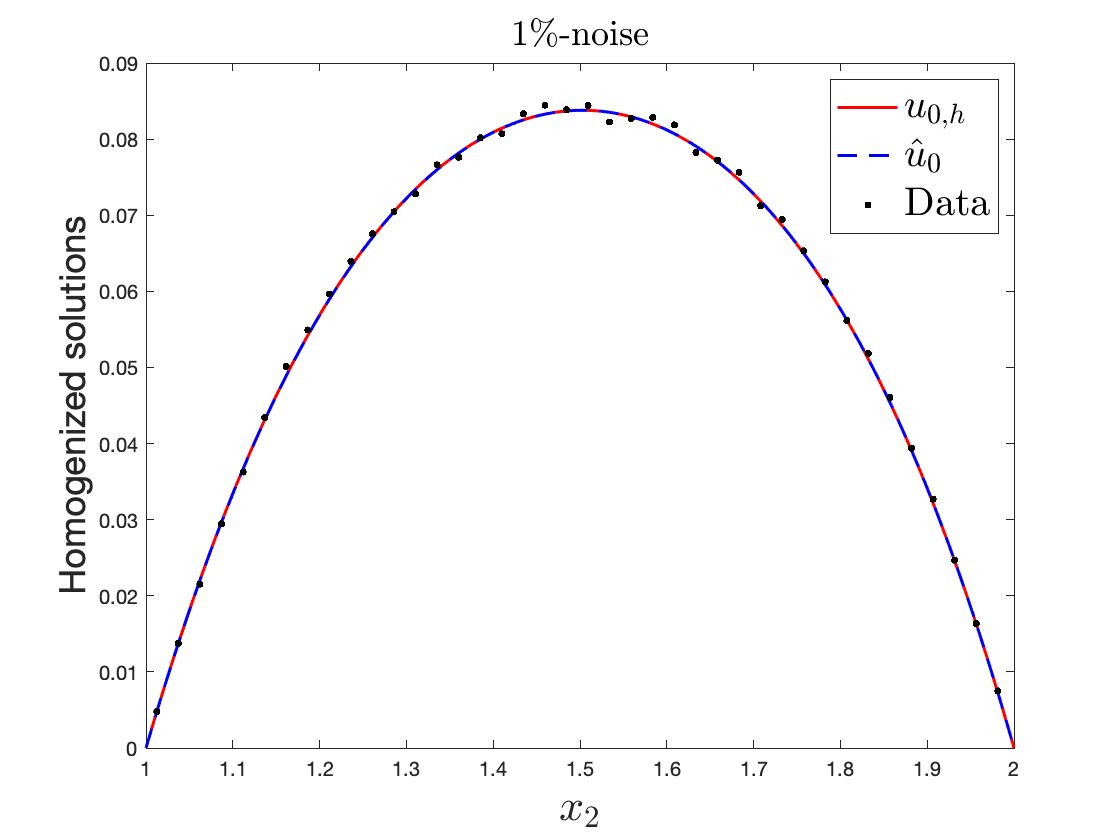}
 \caption{Solutions ($1\%$ noise)}
  \label{figure:2dnonper_plot_data_homsols_ums1_ns1_sc}
  \end{subfigure}
  	   \hspace{-.25in}
  \begin{subfigure}{0.35\textwidth}
  \includegraphics[width=1\textwidth]{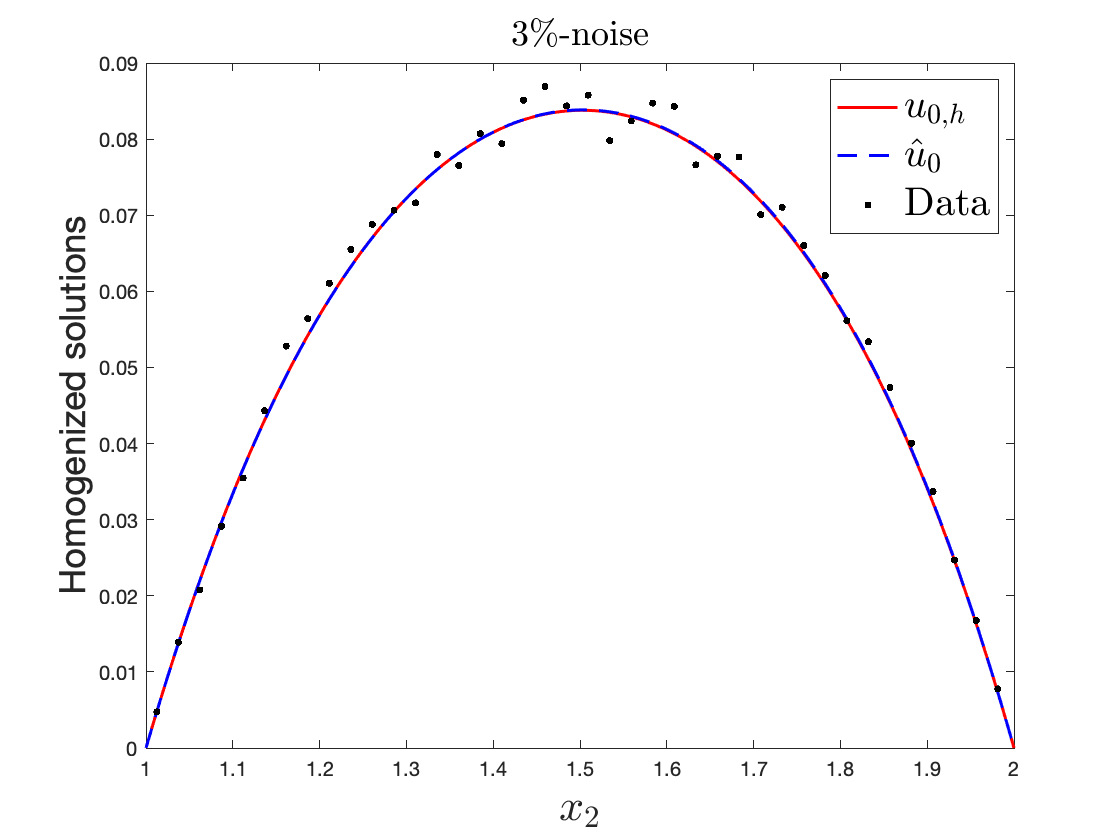}
 \caption{Solutions ($3\%$ noise)}
  \label{figure:2dnonper_plot_data_homsols_ums1_ns3}
  \end{subfigure}
    \vspace*{-4mm}
 \caption{Problem (\ref{eq:original_nonper_2D}): 
   comparison of  the reference solutions (the diagonal entries of  $A^*(x)$ and $u_{0,h}(x)$)  and the counterparts learned by PINNs with different  noise levels. (a. b. c.):  the  G-limit; (d. e. f.):  the homogenized solution at $x_1 = 1.25$, where $\epsilon = 2^{-7}$,  the number of  multiscale  data and PDE residual points are $|{\mathcal T}_d| = 1600$, $|{\mathcal T}_r| = 1600$.}
  \label{figure:2dnonper_plots_ns}
    \vspace*{-5mm}
 \end{figure}

  \begin{figure}[!htb]
	\centering
	\begin{subfigure}{0.40\textwidth}
  \includegraphics[width=\textwidth]{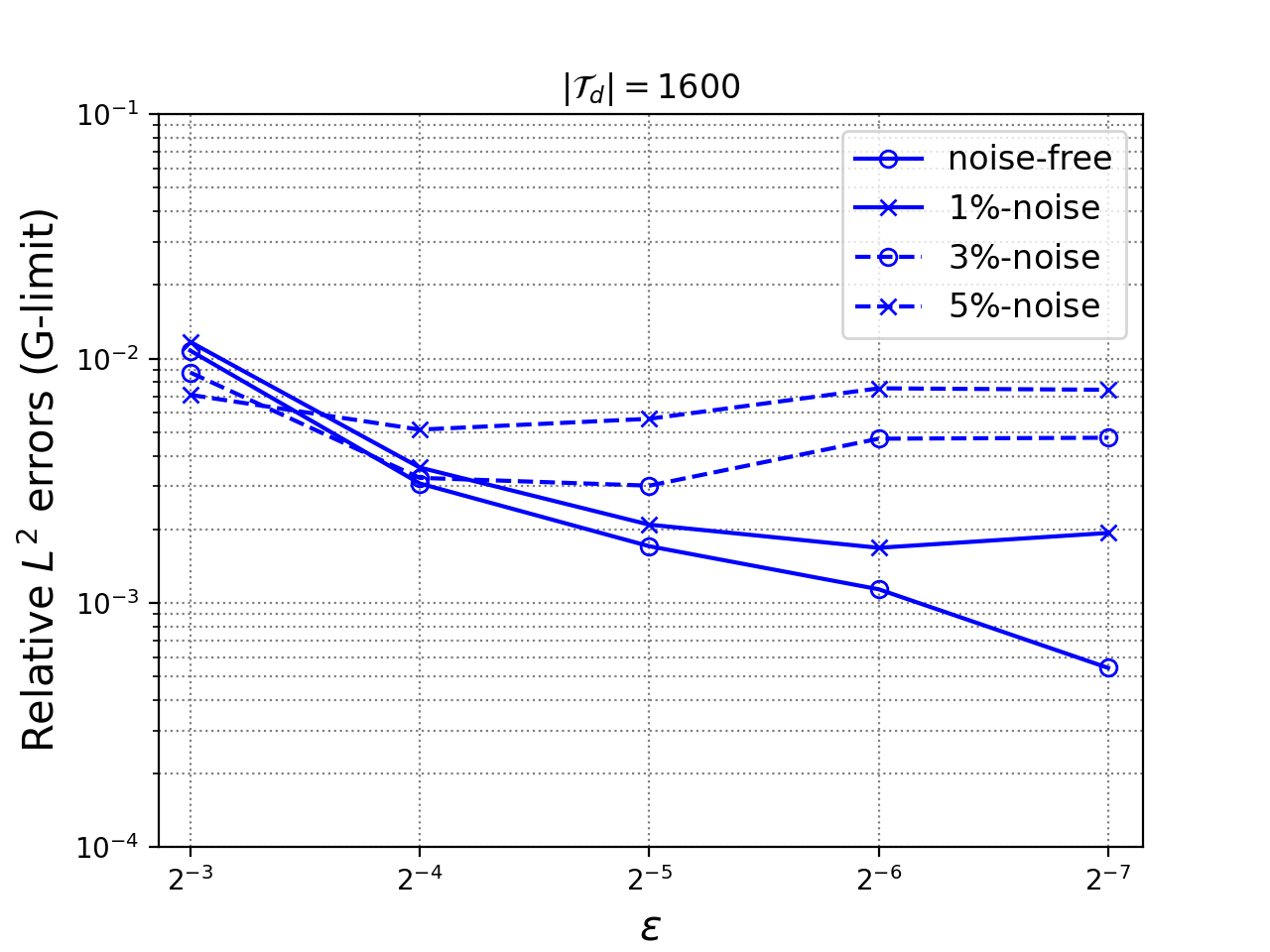}
  \caption{G-limits}
  \label{fig:2dnonper_errors_glimit_ep_sc}
\end{subfigure}
\hspace{.3in}
  \begin{subfigure}{0.40\textwidth}
  \includegraphics[width=\textwidth]{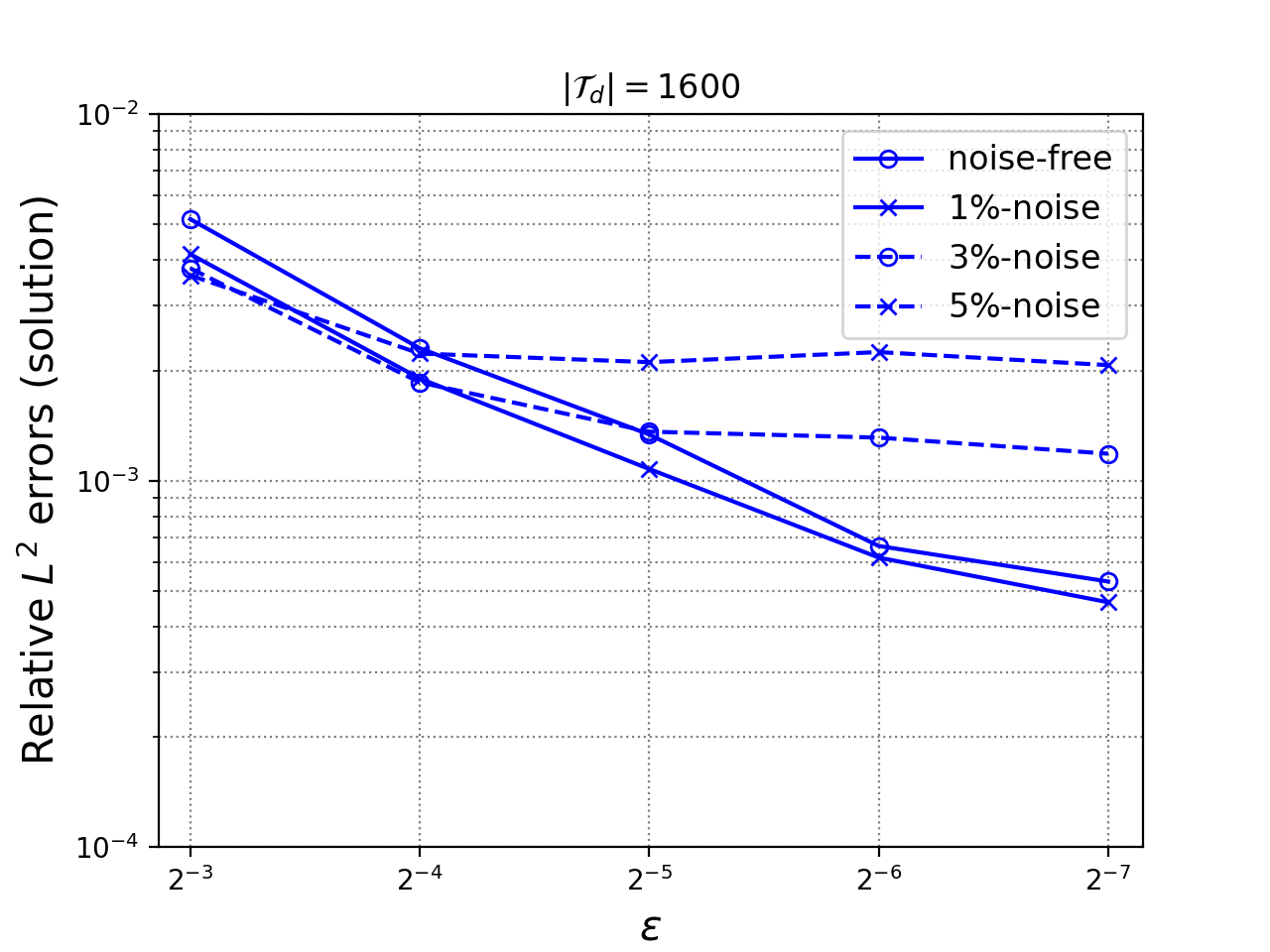}
  \caption{Homogenized solutions}
  \label{fig:2dnonper_errors_homsol_ep_sc}
 \end{subfigure}
   \vspace*{-4mm}
 \caption{ Error results for problem (\ref{eq:original_nonper_2D}): the relative $L^2$ errors for the G-limits  and the homogenized solutions  with different finescale parameter $\ep$ and noise levels in the  data when the number of multiscale data and PDE residual points are  $|{\mathcal T}_d| = 1600$ and $|{\mathcal T}_r| = 1600$.}
\label{figure:glimitsolepns3_Sc}
  \vspace*{-7mm}
\end{figure}

  \begin{figure}[!htb]
 	\centering
 \begin{subfigure}{0.35\textwidth}
  \includegraphics[width=1\textwidth]{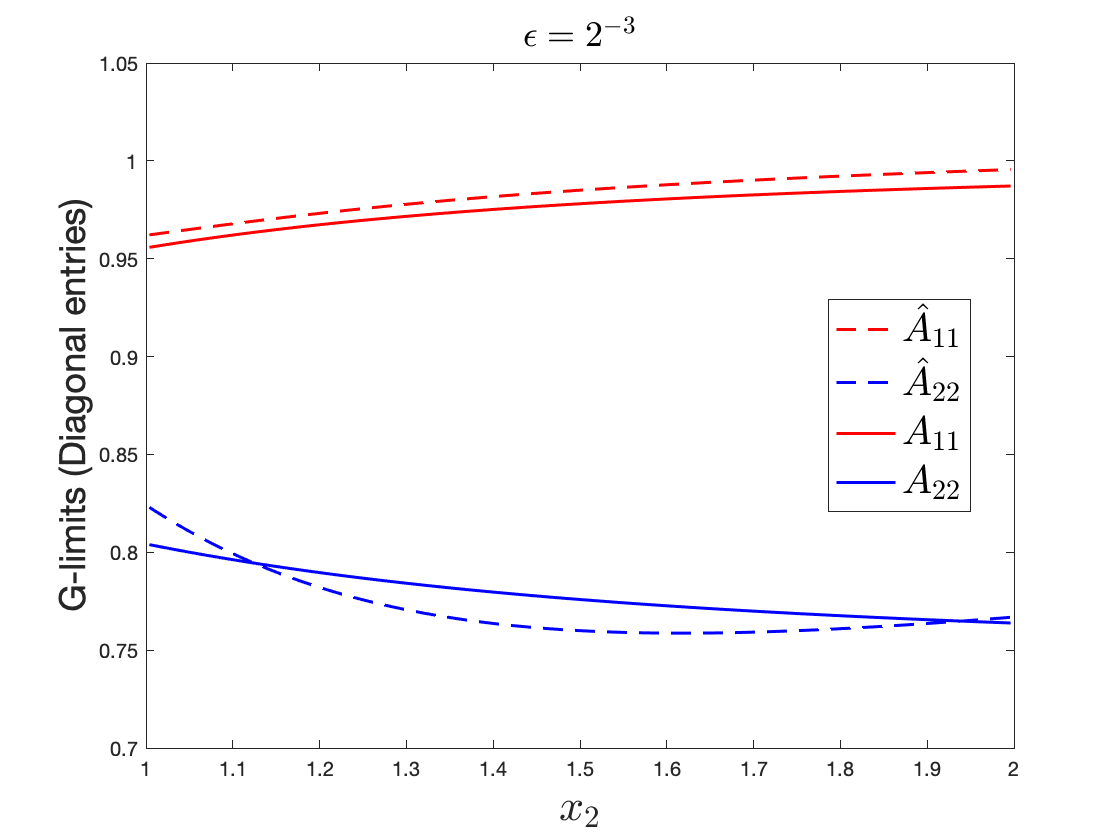}
  \caption{G-limits ($\ep = 2^{-3}$)}
  \label{figure:2dnonper_plot_Glimits_nd40_ep3_mat_sc}
  \end{subfigure}
  \hspace{-.25in}
   \begin{subfigure}{0.35\textwidth}
  \includegraphics[width=1\textwidth]{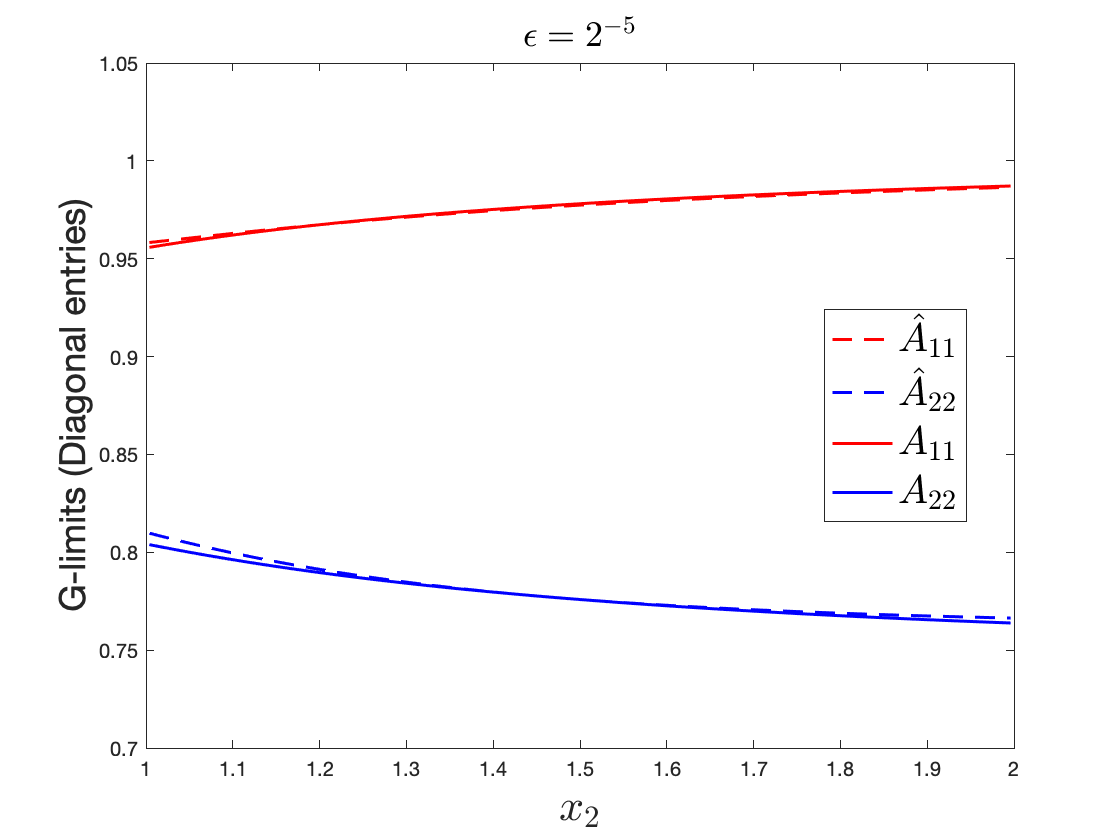}
  \caption{G-limits ($\ep = 2^{-5}$)}
  \label{figure:2dnonper_plot_Glimits_nd40_ep7_mat_ep5_sc}
  \end{subfigure}
   \hspace{-.25in}
   \begin{subfigure}{0.35\textwidth}
  \includegraphics[width=1\textwidth]{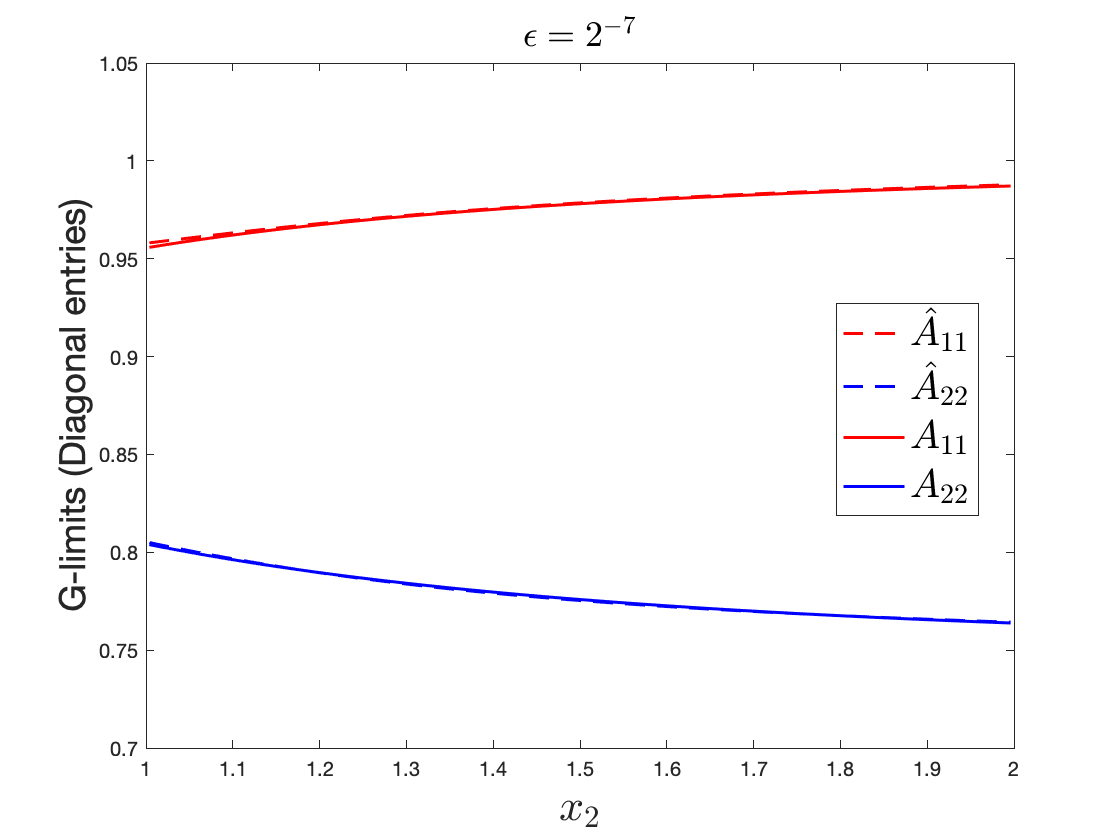}
  \caption{G-limits ($\ep = 2^{-7}$)}
  \label{figure:2dnonper_plot_Glimits_nd40_ep7_mat_sc}
  \end{subfigure}
  \begin{subfigure}{0.35\textwidth}
  \includegraphics[width=1\textwidth]{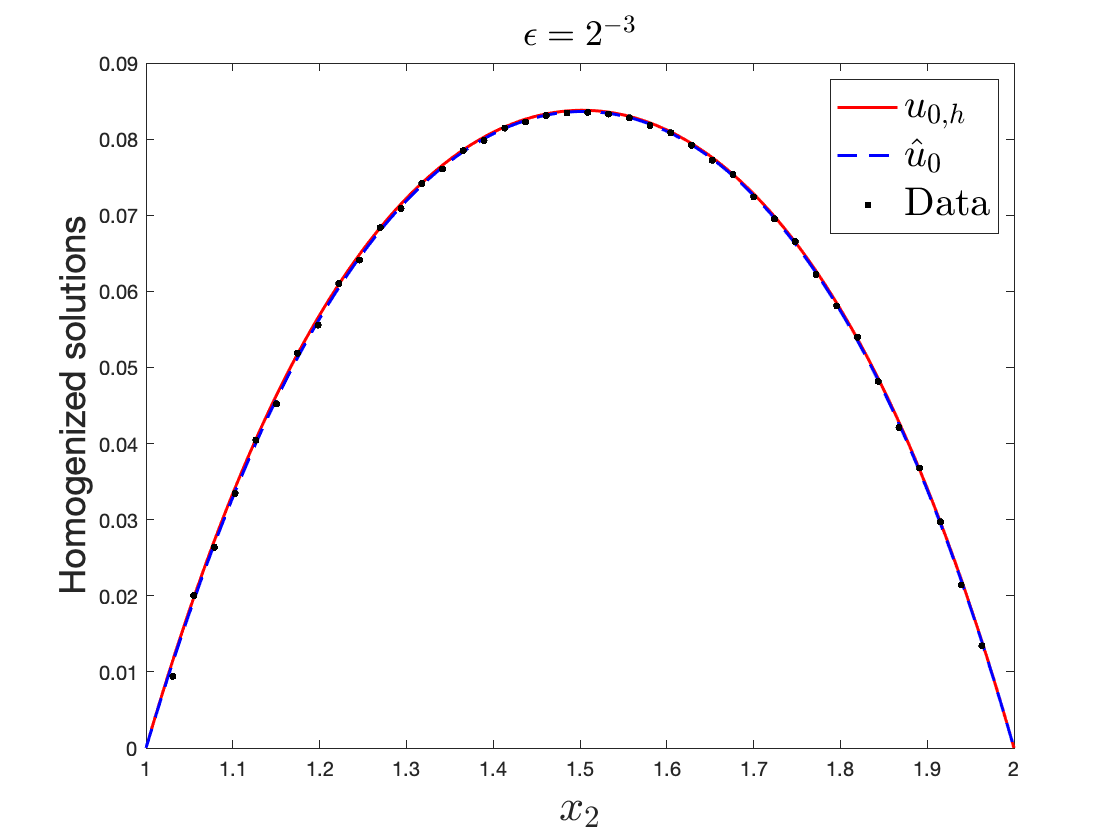}
 \caption{Solutions ($\ep = 2^{-3}$)}
  \label{figure:2dnonper_plot_data_homsols_ums1_ep3_sc}
  \end{subfigure}
     \hspace{-.25in}
  \begin{subfigure}{0.35\textwidth}
  \includegraphics[width=1\textwidth]{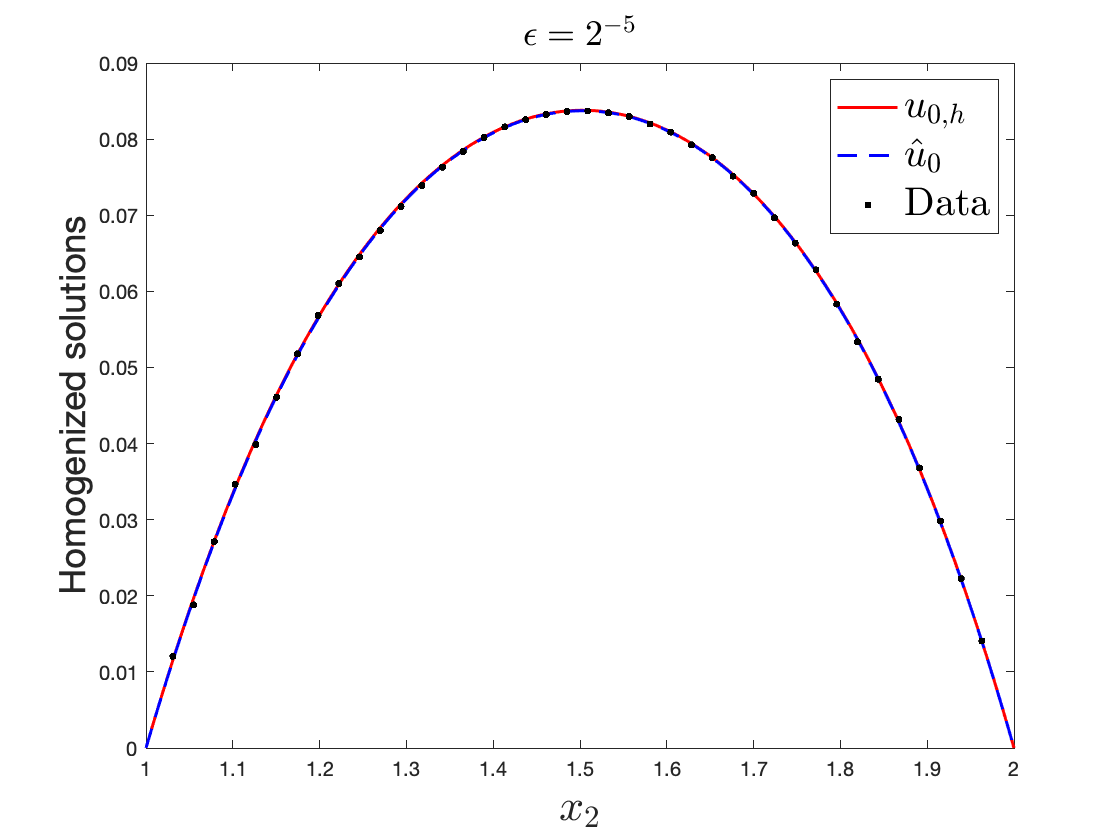}
 \caption{Solutions ($\ep = 2^{-5}$)}
  \label{figure:2dnonper_plot_data_homsols_ums1_ep5_sc}
  \end{subfigure}
       \hspace{-.25in}
  \begin{subfigure}{0.35\textwidth}
  \includegraphics[width=1\textwidth]{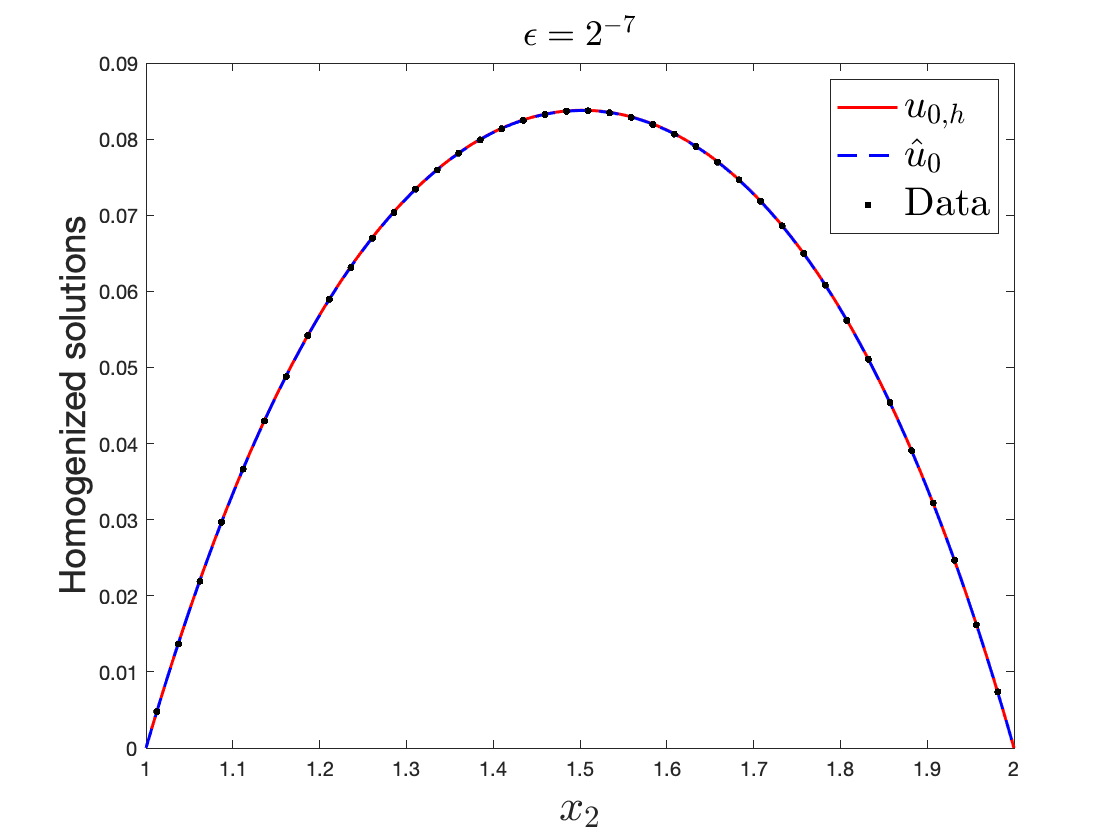}
 \caption{Solutions ($\ep = 2^{-7}$)}
  \label{figure:2dnonper_plot_data_homsols_ums1_ep7_sc}
  \end{subfigure}
  \caption{Problem (\ref{eq:original_nonper_2D}) with noise-free data:  comparison of  the reference coefficient and homogenized solutions (the diagonal entries of  $A^*(x)$ and $u_{0,h}(x)$)  and counterparts learned by PINNs with different values of $\ep$. (a. b. c.):  the  G-limit; (d. e. f.):  the homogenized solution at $x_1 = 1.25$, where the number of  multiscale  data and PDE residual points are $|{\mathcal T}_d| = 1600$, $|{\mathcal T}_r| = 1600$.}
  \label{figure:2dnonper_plot_ums1_ep}
  \end{figure}

  \begin{figure}[!htb]
	\centering
	\captionsetup{justification=centering}
		\begin{subfigure}{0.25\textwidth}
  \includegraphics[width=\textwidth]{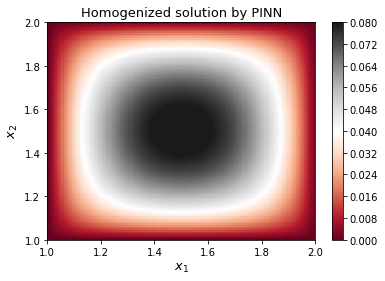}
  \caption{$\hat{u}_0$, noise-free data \\ \quad}
  \label{fig:2dnonper_plot_homsol_PINNs_12}
\end{subfigure}
\hspace{.3in}
  \begin{subfigure}{0.25\textwidth}
  \includegraphics[width=\textwidth]{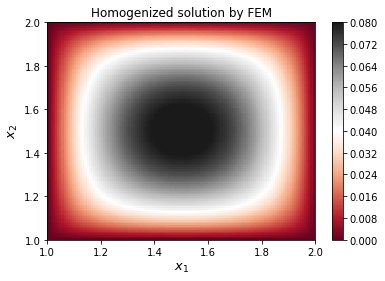}
  \caption{$u_{0,h}$ \\ \quad}
  \label{fig:2dnonper_plot_homsol_FEM_12}
 \end{subfigure}
 \hspace{.3in}
  \begin{subfigure}{0.25\textwidth}
  \includegraphics[width=\textwidth]{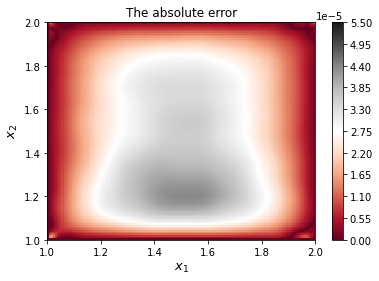}
  \caption{$\left|u_{0,h}(x)-\hat{u}_0(x)\right|$\\ (noise-free data)}
  \label{fig:2dnonper_plot_solerror_nd80_ep7}
 \end{subfigure}
	\begin{subfigure}{0.25\textwidth}
  \includegraphics[width=\textwidth]{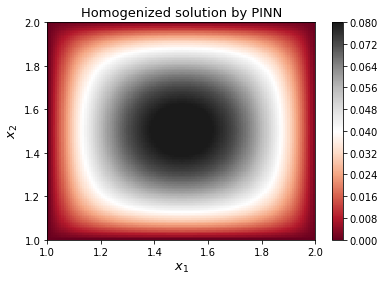}
  \caption{$\hat{u}_0$, $3\%$-noise data\\ \quad}
  \label{fig:2dnonper_plot_homsol_PINNs_12_ns3}
\end{subfigure}
\hspace{.3in}
  \begin{subfigure}{0.25\textwidth}
  \includegraphics[width=\textwidth]{Figure/2dnonper_plot_homsol_FEM_12_label.png}
  \caption{$u_{0,h}$\\ \quad}
  \label{fig:2dnonper_plot_homsol_FEM_12_ns3}
 \end{subfigure}
 \hspace{.3in}
  \begin{subfigure}{0.25\textwidth}
  \includegraphics[width=\textwidth]{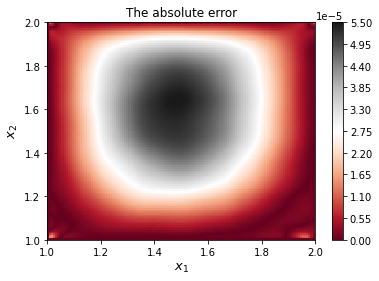}
  \caption{$\left|u_{0,h}(x)-\hat{u}_0(x)\right|$\\ ($3\%$-noise data)}
  \label{fig:2dnonper_plot_solerror_nd80_ep7_ns3}
 \end{subfigure}
 \captionsetup{justification=justified}
 \caption{The 2D Homogenized solutions of problem (\ref{eq:original_nonper_2D}) obtained with noise-free and $3\%$-noise data. (a. d.): Homogenized solution obtained by PINNs; (b. e.): the reference homogenized solution;
 (c. f.): the absolute error between the two solutions, 
 when $\ep=2^{-7}$, 
the number of multiscale solution data and PDE residual points used are $|{\mathcal T}_d|=1600$, $|{\mathcal T}_r|=1600$.
 }
\label{figure:2dnonper_sol_plot_ns3}
  \vspace*{-3mm}
\end{figure}

Figure \ref{figure:glimitsolepns3_Sc} shows the relative $L^2$ errors for the learned G-limits and the homogenized solutions for different finescale parameter $\epsilon$ and $1600$ multiscale solution data collected at fixed spatial locations for all $\epsilon$, i.e., $|{\mathcal T}_d|=1600$. 
For noise-free scenarios, the error decays as the finescale size $\ep$ decreases. To further examine this effect, we plot the corresponding the learned G-limit and homogenized solutions  in Figure \ref{figure:2dnonper_plot_ums1_ep}. In particular, the approximation quality of G-limits appears to be more sensitive to the size of   $\ep$ for the noise-free case.
When the noise dominates over the multiscale oscillations in the data, the results are no longer sensitive to the size of finescale. Nonetheless, we can still achieve the errors of less than $1\%$ for both G-limit and the homogenized solutions under a $5\%$ noise level.

To further highlight the performance of the proposed method, we also show the learned homogenized solution with the reference solutions in Figure \ref{figure:2dnonper_sol_plot_ns3}. As we can see, our solutions agree very well with the reference solutions. Overall, our results show that when mild noise and multiscale fluctuations are presented in the data, the PINNs can provide good estimations of the G-limit for the 2D non-periodic example.

\subsection{Homogenization of an ergodic random coefficient}
Finally, we consider the following two-scale elliptic equation with an ergodic coefficient,  inspired by the exmaple in \cite[Section 4.2]{brown2017hierarchical}:
\beq
\label{eq:1Drandom}
\bsp
-\frac{d}{d x}\left(A^\ep(x,\omega)\cdot\frac{d}{d x}u^\ep(x)\right) &= 1\ \ \textrm{in} \ \ \Omega = [0,1], \\
u^\ep(0) &= u^\ep(1) = 0,
\end{split}
\eeq
where $A^\ep(x,\omega) = A(x, T_{\xoe}(\omega)) = 3.1 + (x+1)\sin(2\pi(\omega_1+\xoe))+\sin(2\pi(\omega_2 + \sqrt{2}\xoe))$ for $\omega = (\omega_1,\omega_2)$   drawn  from a uniform distribution over $[0,1]^2$. 
Here the ergodic dynamical system $T: \mathbb{R} \times {\mathcal Z} \to {\mathcal Z}$ is given by
\beq
T(x)\omega = \omega + (1,\sqrt{2})x.
\eeq
\begin{figure}[!htb]
	\centering
	\begin{subfigure}{0.40\textwidth}
  \includegraphics[width=\textwidth]{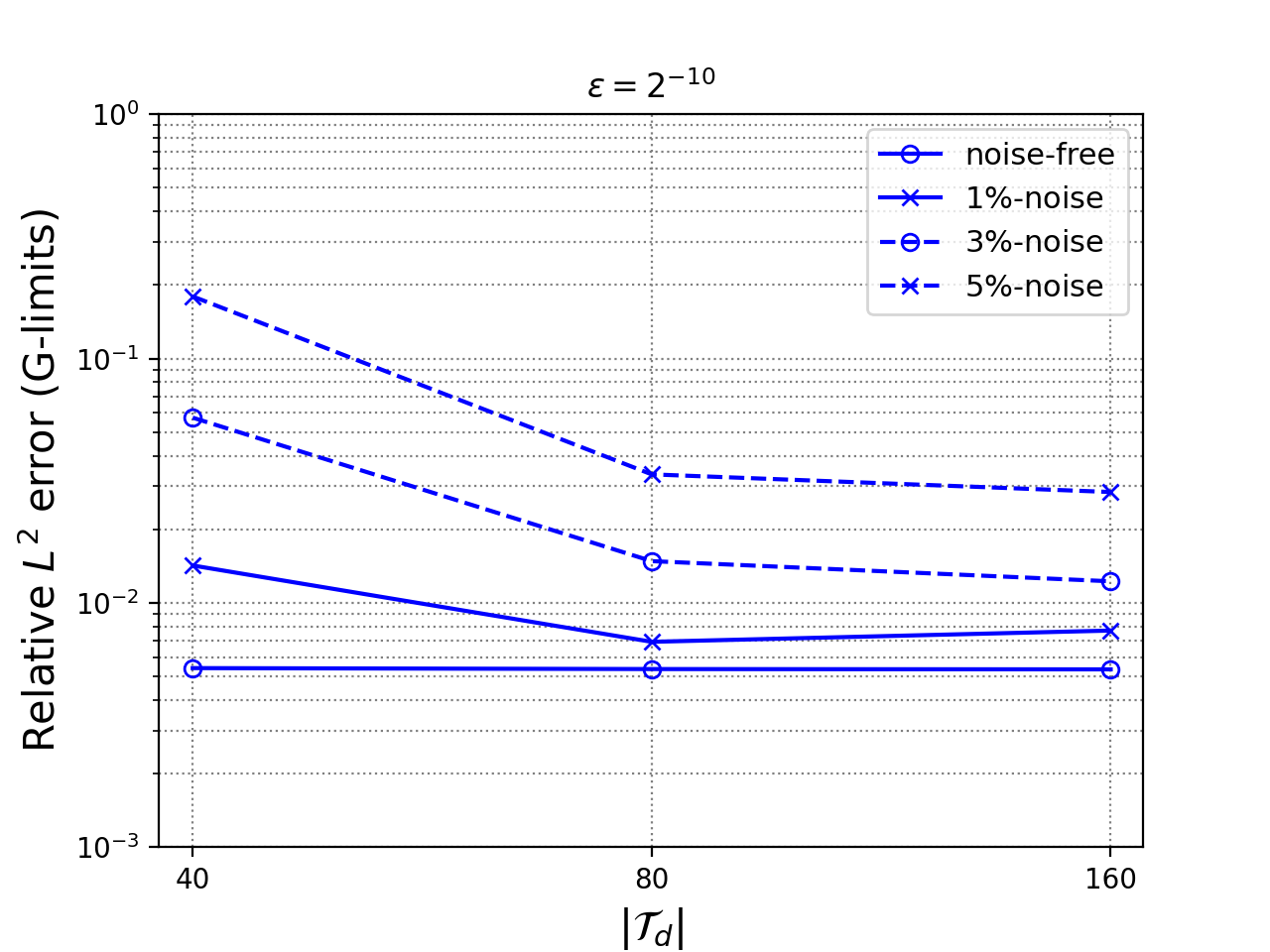}
  \caption{G-limits}
  \label{fig:1drandom_errors_glimit_nd}
\end{subfigure}
\hspace{.3in}
  \begin{subfigure}{0.40\textwidth}
  \includegraphics[width=\textwidth]{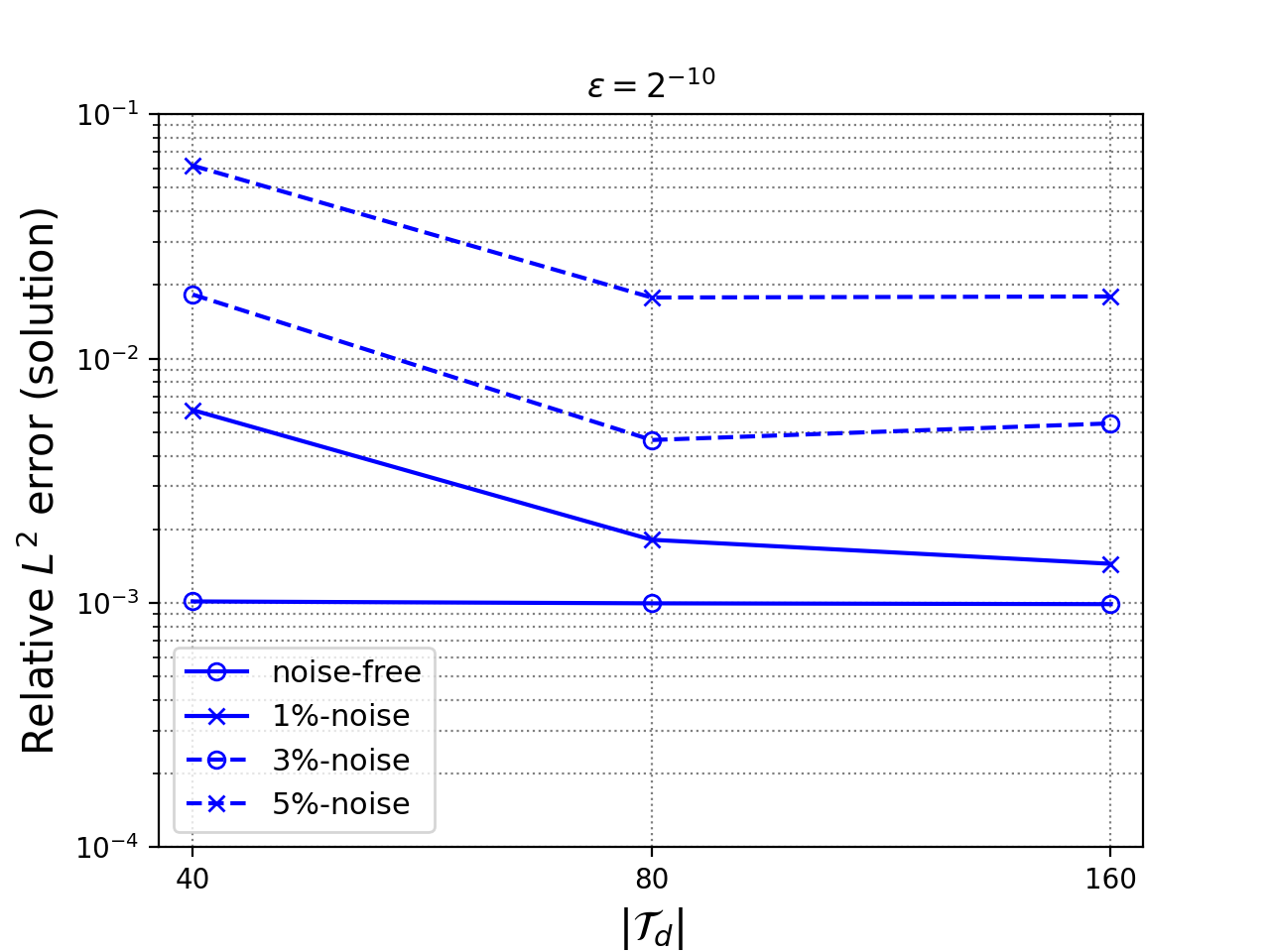}
  \caption{Homogenized solutions}
  \label{fig:1drandom_errors_homsol_nd}
 \end{subfigure}
   \vspace*{-4mm}
 \caption{Error results for problem (\ref{eq:1Drandom}): the relative $L^2$ errors for the G-limits  and the homogenized solutions with different number of multiscale data corrupted by different noise levels for $\ep = 2^{-10}$ and the number of PDE residual points is $|{\mathcal T}_r| = |{\mathcal T}_d| +20$.}
\label{figure:glimitsoltdns4}
  \vspace*{-3mm}
\end{figure}

\begin{figure}[!htb]
	\centering
	   \begin{subfigure}{0.35\textwidth}
  \includegraphics[width=\textwidth]{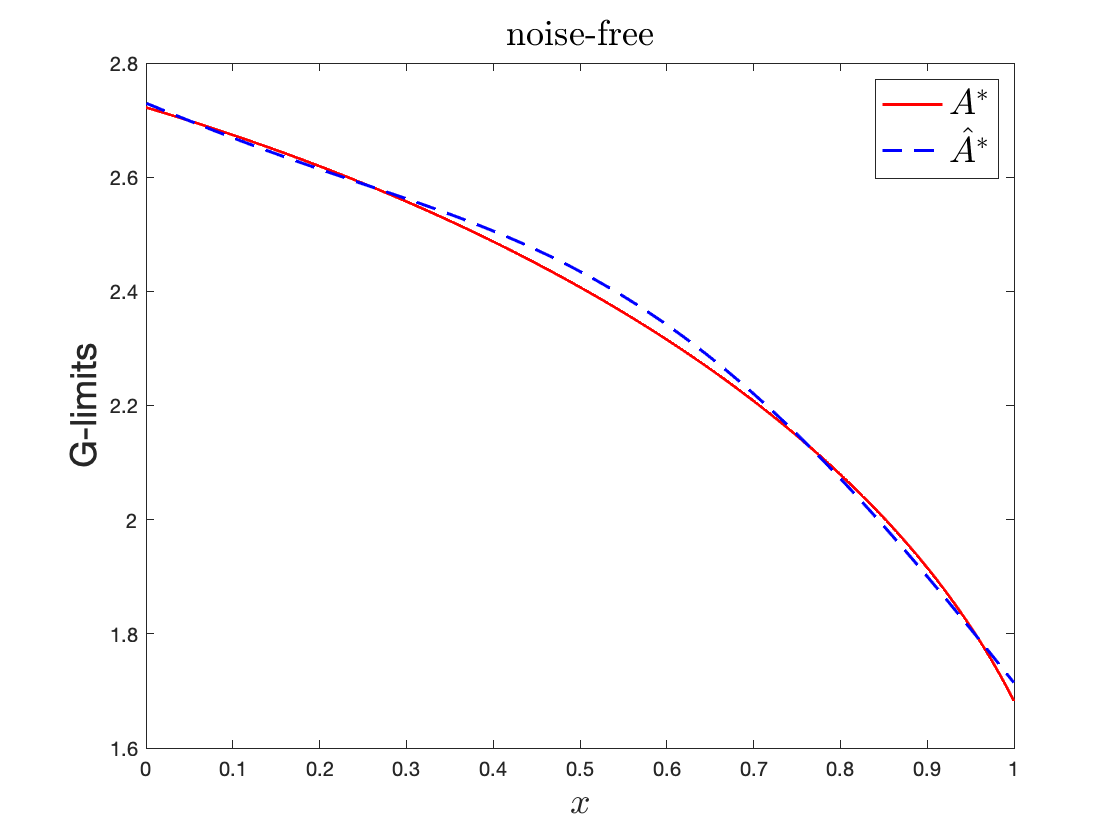}
  \caption{G-limits (noise-free)}
  \label{fig:1drandom_plot_kappa_ums1_ep10_160_1}
\end{subfigure}
\hspace{-.25in}
		\begin{subfigure}{0.35\textwidth}
  \includegraphics[width=\textwidth]{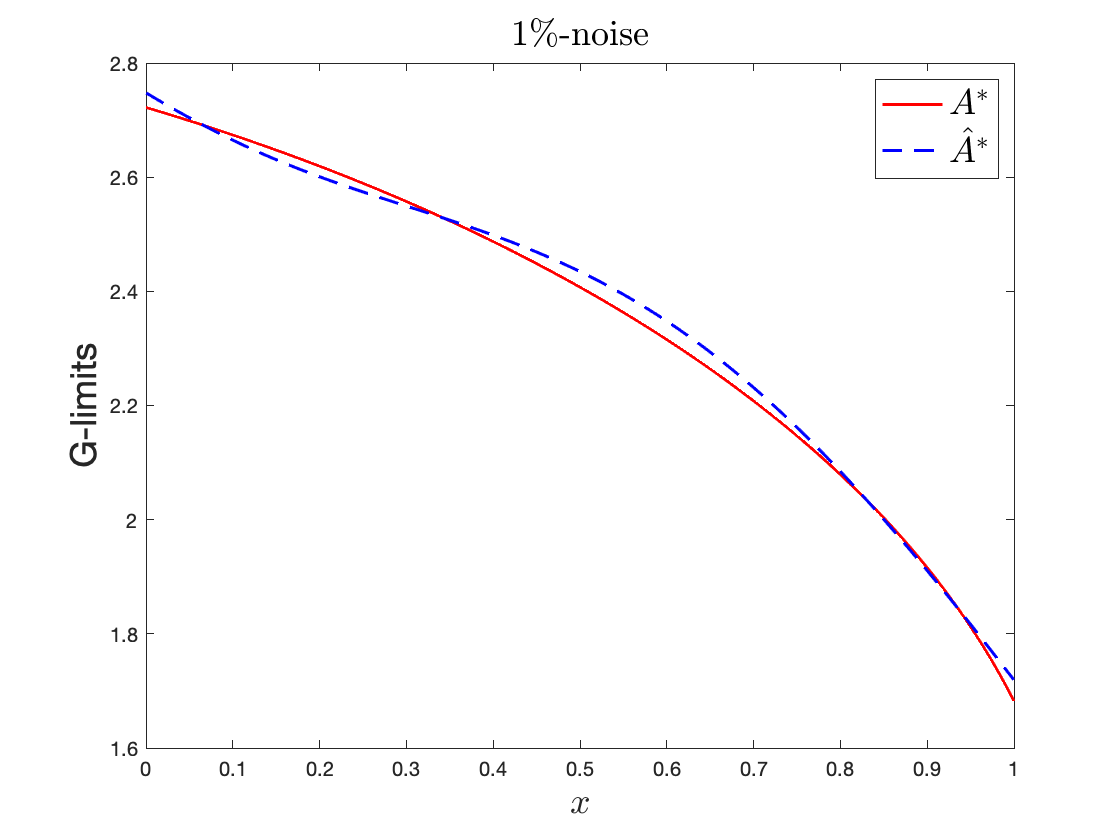}
  \caption{G-limits ($1\%$ noise)}
  \label{fig:1drandom_plot_kappa_ums1_ns1_160_sciann}
\end{subfigure}
\hspace{-.25in}
	\begin{subfigure}{0.35\textwidth}
  \includegraphics[width=\textwidth]{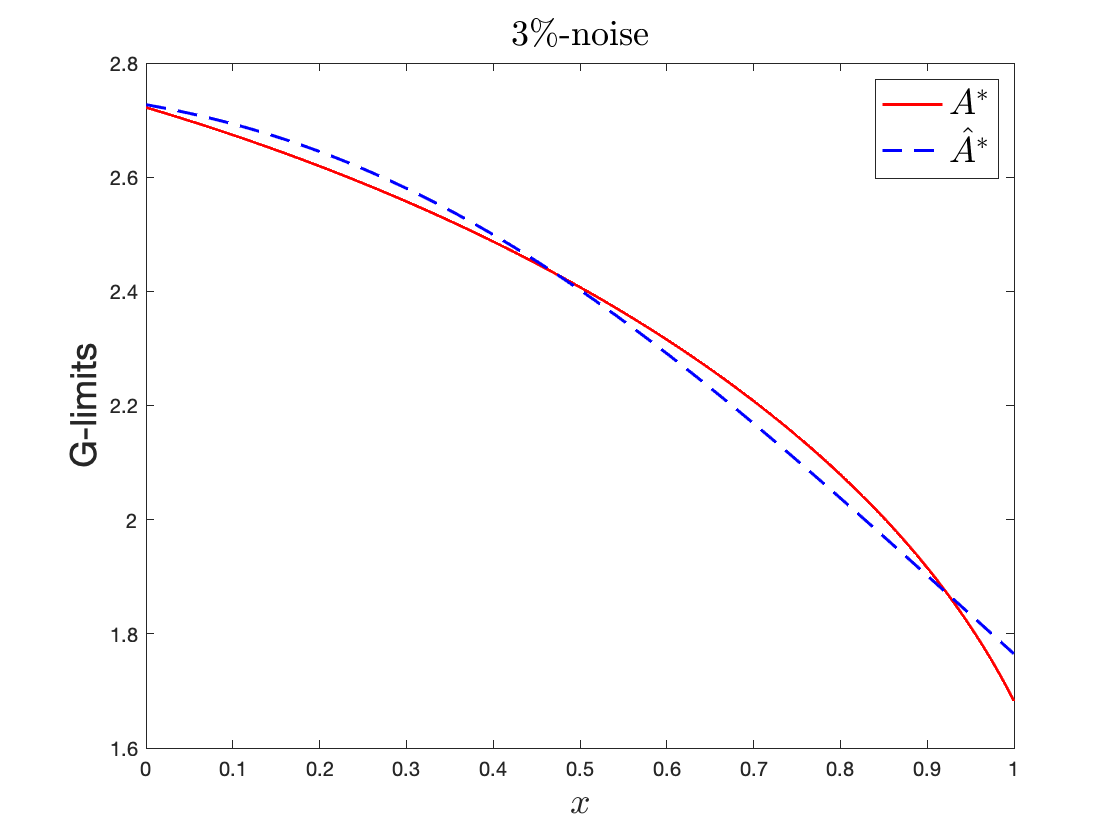}
  \caption{G-limits ($3\%$ noise)}
  \label{fig:1drandom_plot_kappa_ums1_ns3_160_sciann}
\end{subfigure}
\hspace{-.25in}
  \begin{subfigure}{0.35\textwidth}
  \includegraphics[width=\textwidth]{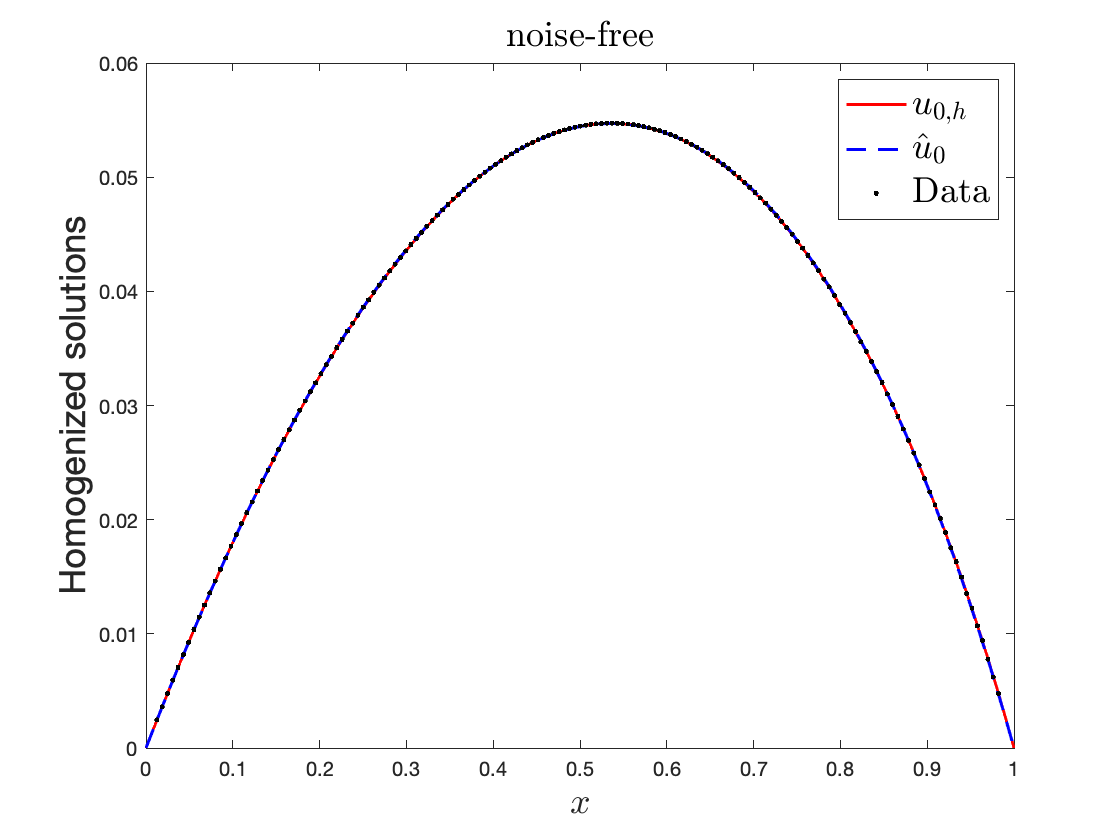}
  \caption{Solutions (noise-free)}
  \label{fig:1drandom_plot_data_homsols_ums1_ep10_160_1}
 \end{subfigure}
 \hspace{-.25in}
  \begin{subfigure}{0.35\textwidth}
  \includegraphics[width=\textwidth]{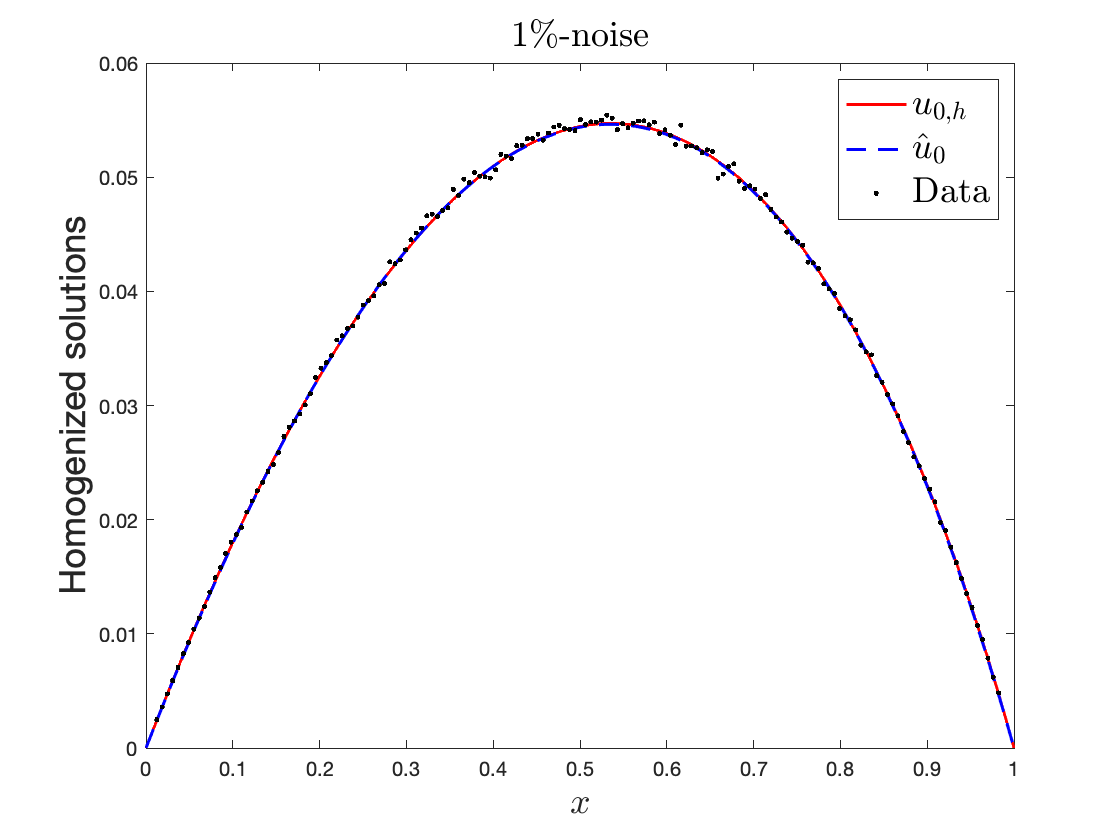}
  \caption{Solutions ($1\%$ noise)}
  \label{fig:1drandom_plot_data_homsols_ums1_ns1_160_sciann}
 \end{subfigure}
\hspace{-.25in}
  \begin{subfigure}{0.35\textwidth}
  \includegraphics[width=\textwidth]{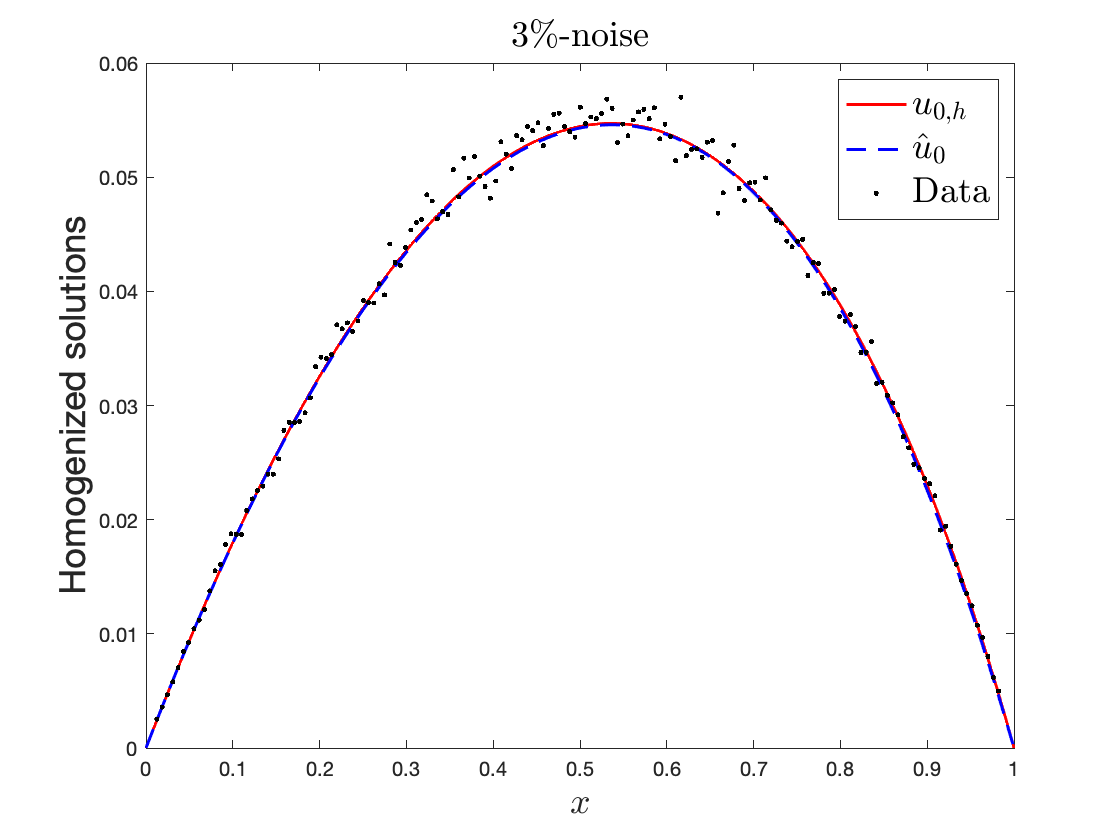}
  \caption{Solutions ($3\%$ noise)}
  \label{fig:1drandom_plot_data_homsols_ums1_ns3_160_sciann}
 \end{subfigure}
   \vspace*{-4mm}
 \caption{Problem (\ref{eq:1Drandom}):
 comparison of  the reference coefficient and  solutions ($A^*(x)$ and $u_{0,h}(x)$)  and the counterparts learned by PINNs  with different noise levels in the data. (a. b. c.):  the  G-limit; (d. e. f.):  the homogenized solution, where $\ep = 2^{-10}$ and the number of  multiscale  data and PDE residual points are $|{\mathcal T}_d| = 160$, $|{\mathcal T}_r| = 180$.}
\label{figure:1drandom_plots_ns}
  \vspace*{-3mm}
\end{figure}

\begin{figure}[!htb]
	\centering
	\begin{subfigure}{0.40\textwidth}
  \includegraphics[width=\textwidth]{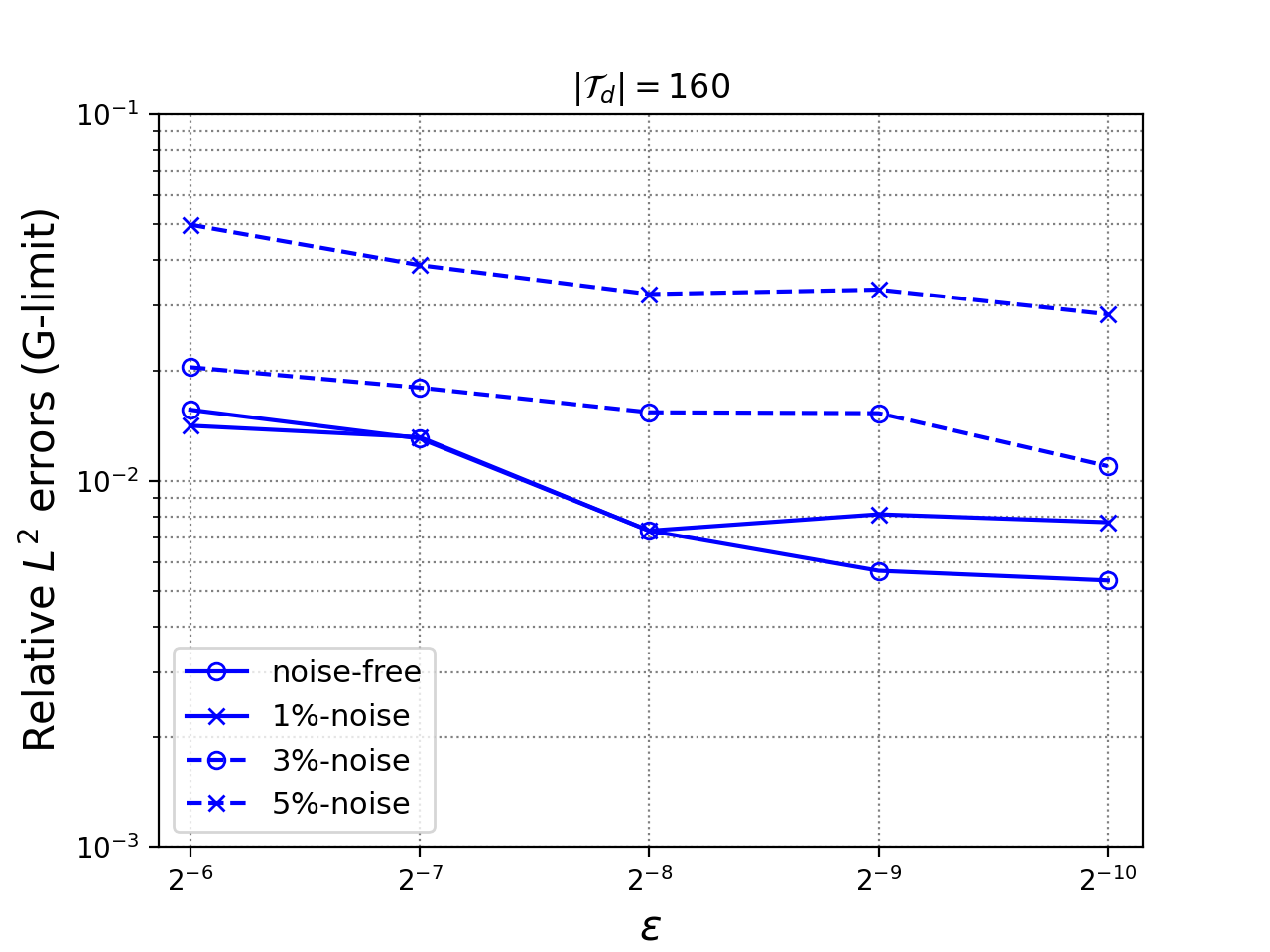}
  \caption{G-limits}
  \label{fig:1drandom_errors_glimit_ep}
\end{subfigure}
\hspace{.3in}
  \begin{subfigure}{0.40\textwidth}
  \includegraphics[width=\textwidth]{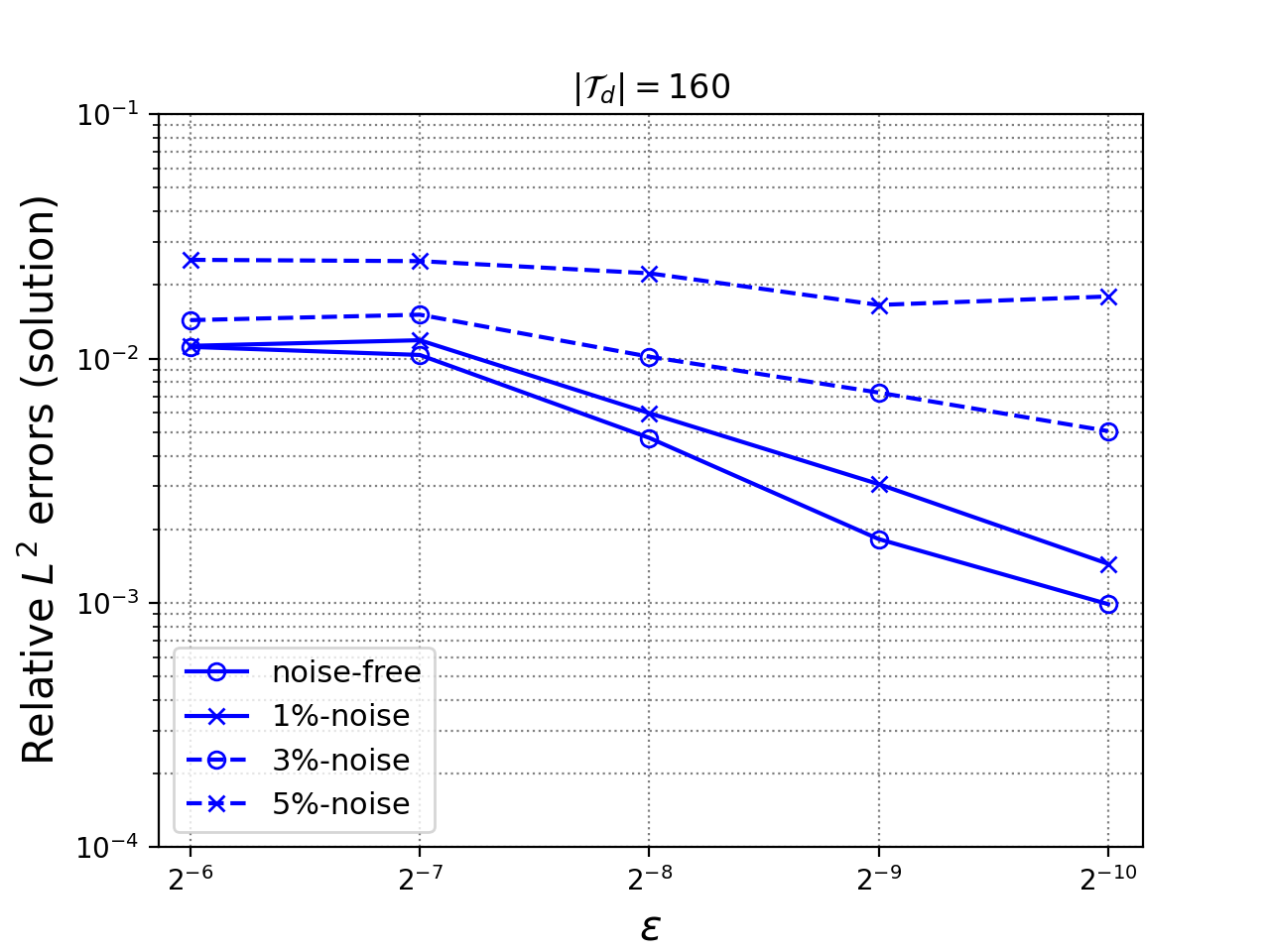}
  \caption{Homogenized solutions}
  \label{fig:1drandom_errors_homsol_ep}
 \end{subfigure}
   \vspace*{-4mm}
 \caption{Error results for problem (\ref{eq:1Drandom}): the relative $L^2$ errors for the G-limits and the homogenized solutions  with different finescale parameter $\ep$ and noise levels in the  data when the number of  multiscale  data and PDE residual points are  $|{\mathcal T}_d| = 160$ and $|{\mathcal T}_r| = 180$.}
\label{figure:glimitsolepns4}
  \vspace*{-5mm}
\end{figure}
Notably, the G-limit $A^*(x)$ of this ergodic homogenization problem is deterministic and independent of the realization of $\omega$ \cite{jikov2012homogenization}. 
In this example, it is known that the exact G-limit is given by $1/\mathbb{E}\left [1/A(x,\omega)\right]$, where $\mathbb{E}$ denotes the expectation with respect to the realizations of $\omega$ \cite{alexanderian2014primer}. 
Traditional approaches for ergodic homogenization usually first compute the local cell problems for many different realizations of the coefficient $A^\ep(x,\omega)$ to obtain the realization dependent approximations to the G-limits. Then  the G-limit can be approximated by taking its expectation. This procedure requires solving a lot of cell problems at many different points $x$ with thousands of realizations of $\omega$.
For PINNs, on the other hand, we just  need to collect the multiscale solution data based on a single realization of the coefficient for PINNs.

Specifically, the reference G-limit $A^*(x)$ is computed as the expectation by roughly $200,000$ 
Monte Carlo samples over $2000$ equidistant points in the spatial domain. 
Based on this G-limit, we compute the reference homogenized solution by FEM with mesh size $h=1/2000$.  For PINNs,  we  learned the G-limit based on only one realization of  $\omega = (0.5,0.5)$. The training data are equally spaced sampled from the multiscale solution for each finescale parameter value $\epsilon$ computed by FEM with a mesh size $h = 1/10^5$. The architecture parameters and other hyperparameters of PINNs are listed in {\it Table \ref{tb:hyperparameters}} in the appendix. The relative $L^2$ errors are computed using a mesh of size $h=1/2000$.

 Figure \ref{figure:glimitsoltdns4} shows the error convergence of the learned G-limits and homogenized solution for $\ep = 2^{-10}$. For the noise-free case, an error level ${\mathcal O}(10^{-3})$ for both G-limit and the homogenized solution can be achieved. For noisy data, while the errors for the G-limit and homogenized solution tend to stagnate after more than $80$ multiscale data points are used, we can still achieve errors of ${\mathcal O}(10^{-2})$  for both G-limit and the homogenized solution 
with $5\%$-noise corruption in the data. 

We also compare the learned G-limits and the homogenized solutions with their references for $\ep = 2^{-10}$ and $|{\mathcal T}_d| = 160$ data with different noise levels in Figure \ref{figure:1drandom_plots_ns}. While  the learned G-limit is close to the reference coefficient, the approximated homogenized solution almost overlaps with the data.
Again, 
we observed PINNs tend to learn the macroscopic behavior of the noisy data that is close to the reference homogenized solution.

The error results with finescale parameter $\ep$ are presented in Figure \ref{figure:glimitsolepns4}. For both noise-free and noisy scenarios, the errors tend to decay as the finescale size $\ep$ decreases, particularly for homogenized solutions. 
The effect of $\ep$  is less pronounced when the noise level is high because the noise dominates over the fine scale size of $\ep$.
Notably,
with $5\%$-noise corruption, we can still achieve  the relative errors less than $5\%$ for both G-limit and homogenized solution by incorporating the corresponding homogenized equation.

\section{Conclusion}

In this paper, we proposed a simple and flexible approach to estimate the G-limit and approximate the homogenized solution for multiscale elliptic equations from data, by adopting physics-informed neural networks (PINNs). Due to the lack of the homogenized solution data or measurements, we employ the multiscale solution data as the surrogate of the homogenized solution. 
Despite the rapid multiscale and noisy fluctuations presented in the data, we demonstrated that 
PINNs 
are capable to effectively extract the macroscopic (homogenized) behavior from data and provide good approximations to the G-limits and the homogenized solution.    
The applicability and performance  of the method have been demonstrated through a number of different benchmark problems. Finally, we remark that  except for the assumption of the existence and structure of the homogenized equation, our approach does not rely on the periodicity or the explicit formula of the underlying multiscale coefficient during the learning stage,  which can be 
applicable to  more general settings beyond periodic cases.

\section*{Acknowledgments}
XZ was supported by Simons Foundation.
\appendix

\section{Hyperparameters used in each numerical example}

\begin{table}[!htb]
\centering
\caption{ Hyperparameters used for each numerical example: For each example, the learning rate is decayed when the training loss plateaus.}
\begin{tabular}{|c|c|c|c|c|c|c|}
  \hline
\multicolumn{7}{|c|}{\bf{1D locally periodic coefficient (Section 4.1)} }\\
\hline
\multicolumn{2}{|c|}{NN depth}  & \multicolumn{2}{c|}{NN width}& \multirow{2}{*}{Initial learning rate} & \multirow{2}{*}{$\#$ of epochs} &  \multirow{2}{*}{Batch size} \\
\cline{1-4}$\hat{A}^*(x)$&$\hat{u}_{0}(x)$&$\hat{A}^*(x)$&$\hat{u}_{0}(x)$&&&\\
\hline
  3&3& 30&30 & 0.001 & $40000$ & $64$ \\
  \hline
  \multicolumn{7}{|c|}{\bf{1D heavily oscillatory coefficient (Section 4.2)}}\\
\hline
\multicolumn{2}{|c|}{NN depth}  & \multicolumn{2}{c|}{NN width}& \multirow{2}{*}{Initial learning rate} & \multirow{2}{*}{$\#$ of epochs} &  \multirow{2}{*}{Batch size} \\
\cline{1-4}$\hat{A}^*(x)$&$\hat{u}_{0}(x)$&$\hat{A}^*(x)$&$\hat{u}_{0}(x)$&&&\\
\hline
  3&3& 50&50 & 0.0001 & $80000$ & $64$ \\
  \hline
  \multicolumn{7}{|c|}{\bf{2D non-periodic coefficient (Section 4.3)}}\\
\hline
\multicolumn{2}{|c|}{NN depth}  & \multicolumn{2}{c|}{NN width}& \multirow{2}{*}{Initial learning rate} & \multirow{2}{*}{$\#$ of epochs} &  \multirow{2}{*}{Batch size} \\
\cline{1-4}$\hat{A}^*_{ii}(x)$&$\hat{u}_{0}(x)$&$\hat{A}^*_{ii}(x)$&$\hat{u}_{0}(x)$&&&\\
\hline
 2&4& 40&45 & 0.001 & $100000$ & $200$ \\
  \hline
  \multicolumn{7}{|c|}{\bf{1D ergodic random coefficient (Section 4.4)}}\\
\hline
\multicolumn{2}{|c|}{NN depth}  & \multicolumn{2}{c|}{NN width}& \multirow{2}{*}{Initial learning rate} & \multirow{2}{*}{$\#$ of epochs} &  \multirow{2}{*}{Batch size} \\
\cline{1-4}$\hat{A}^*(x)$&$\hat{u}_{0}(x)$&$\hat{A}^*(x)$&$\hat{u}_{0}(x)$&&&\\
\hline
  2&3& 10&30 & 0.001 & $60000$ & $64$ \\
\hline
\end{tabular} 
\label{tb:hyperparameters}
\end{table}

\clearpage

\bibliographystyle{siam}
\bibliography{random,asympPre,referencesHomPINN}

\begin{thebibliography}{10}

\bibitem{abdulle2020bayesian}
Assyr Abdulle and Andrea~Di Blasio.
\newblock A bayesian numerical homogenization method for elliptic multiscale
  inverse problems.
\newblock {\em SIAM/ASA Journal on Uncertainty Quantification}, 8(1):414--450,
  2020.

\bibitem{abdulle2020ensemble}
Assyr Abdulle, Giacomo Garegnani, and Andrea Zanoni.
\newblock Ensemble kalman filter for multiscale inverse problems.
\newblock {\em Multiscale Modeling \& Simulation}, 18(4):1565--1594, 2020.

\bibitem{alexanderian2014primer}
Alen Alexanderian.
\newblock A primer on homogenization of elliptic pdes with stationary and
  ergodic random coefficient functions.
\newblock {\em arXiv preprint arXiv:1408.5827}, 2014.

\bibitem{allaire1992homogenization}
Gr{\'e}goire Allaire.
\newblock Homogenization and two-scale convergence.
\newblock {\em SIAM Journal on Mathematical Analysis}, 23(6):1482--1518, 1992.

\bibitem{allaire2012shape}
Gr{\'e}goire Allaire.
\newblock {\em Shape optimization by the homogenization method}, volume 146.
\newblock Springer Science \& Business Media, 2012.

\bibitem{allaire1996multiscale}
Gr{\'e}goire Allaire and Marc Briane.
\newblock Multiscale convergence and reiterated homogenisation.
\newblock {\em Proceedings of the Royal Society of Edinburgh Section A:
  Mathematics}, 126(2):297--342, 1996.

\bibitem{arbabi2020linking}
Hassan Arbabi, Judith~E Bunder, Giovanni Samaey, Anthony~J Roberts, and
  Ioannis~G Kevrekidis.
\newblock Linking machine learning with multiscale numerics: data-driven
  discovery of homogenized equations.
\newblock {\em Jom}, 72(12):4444--4457, 2020.

\bibitem{brown2017hierarchical}
Donald~L Brown and Viet~Ha Hoang.
\newblock A hierarchical finite element monte carlo method for stochastic
  two-scale elliptic equations.
\newblock {\em Journal of Computational and Applied Mathematics}, 323:16--35,
  2017.

\bibitem{carneiro2018performance}
Tiago Carneiro, Raul Victor~Medeiros Da~N{\'o}brega, Thiago Nepomuceno, Gui-Bin
  Bian, Victor Hugo~C De~Albuquerque, and Pedro~Pedrosa Reboucas~Filho.
\newblock Performance analysis of google colaboratory as a tool for
  accelerating deep learning applications.
\newblock {\em IEEE Access}, 6:61677--61685, 2018.

\bibitem{chen2020physics}
Yuyao Chen, Lu~Lu, George~Em Karniadakis, and Luca Dal~Negro.
\newblock Physics-informed neural networks for inverse problems in nano-optics
  and metamaterials.
\newblock {\em Optics Express}, 28(8):11618--11633, 2020.

\bibitem{defranceschi1993introduction}
Anneliese Defranceschi.
\newblock An introduction to homogenization and g-convergence.
\newblock {\em School on Homogenization, Lecture notes of the courses held at
  ICTP, Trieste}, pages 63--122, 1993.

\bibitem{floden2009g}
Liselott Flod{\'e}n.
\newblock {\em G-convergence and homogenization of some sequences of monotone
  differential operators}.
\newblock PhD thesis, Mittuniversitetet, 2009.

\bibitem{frederick2014numerical}
Christina Frederick and Bjorn Engquist.
\newblock Numerical methods for multiscale inverse problems.
\newblock {\em arXiv preprint arXiv:1401.2431}, 2014.

\bibitem{goodfellow2016deep}
Ian Goodfellow, Yoshua Bengio, Aaron Courville, and Yoshua Bengio.
\newblock {\em Deep learning}, volume~1.
\newblock MIT press Cambridge, 2016.

\bibitem{gulliksson2016separating}
Marten Gulliksson, Anders Holmbom, Jens Persson, and Ye~Zhang.
\newblock A separating oscillation method of recovering the g-limit in standard
  and non-standard homogenization problems.
\newblock {\em Inverse Problems}, 32(2):025005, 2016.

\bibitem{haghighat2021sciann}
Ehsan Haghighat and Ruben Juanes.
\newblock Sciann: A keras/tensorflow wrapper for scientific computations and
  physics-informed deep learning using artificial neural networks.
\newblock {\em Computer Methods in Applied Mechanics and Engineering},
  373:113552, 2021.

\bibitem{jikov2012homogenization}
V.~Jikov, S.~Kozlov, and O.~Oleinik.
\newblock {\em Homogenization of differential operators and integral
  functionals}.
\newblock Springer Science \& Business Media, 2012.

\bibitem{lagaris1998artificial}
Isaac~E Lagaris, Aristidis Likas, and Dimitrios~I Fotiadis.
\newblock Artificial neural networks for solving ordinary and partial
  differential equations.
\newblock {\em IEEE transactions on neural networks}, 9(5):987--1000, 1998.

\bibitem{lu2019deepxde}
Lu~Lu, Xuhui Meng, Zhiping Mao, and George~E Karniadakis.
\newblock Deepxde: A deep learning library for solving differential equations.
\newblock {\em arXiv preprint arXiv:1907.04502}, 2019.

\bibitem{lu2021physics}
Lu~Lu, Raphael Pestourie, Wenjie Yao, Zhicheng Wang, Francesc Verdugo, and
  Steven~G Johnson.
\newblock Physics-informed neural networks with hard constraints for inverse
  design.
\newblock {\em arXiv preprint arXiv:2102.04626}, 2021.

\bibitem{nolen2009fine}
James Nolen and George Papanicolaou.
\newblock Fine scale uncertainty in parameter estimation for elliptic
  equations.
\newblock {\em Inverse Problems}, 25(11):115021, 2009.

\bibitem{papanicolau1978asymptotic}
G~Papanicolau, A~Bensoussan, and J-L Lions.
\newblock {\em Asymptotic analysis for periodic structures}.
\newblock Elsevier, 1978.

\bibitem{persson2012selected}
Jens Persson.
\newblock {\em Selected topics in homogenization}.
\newblock PhD thesis, Mittuniversitetet, 2012.

\bibitem{raissi2019physics}
Maziar Raissi, Paris Perdikaris, and George~E Karniadakis.
\newblock Physics-informed neural networks: A deep learning framework for
  solving forward and inverse problems involving nonlinear partial differential
  equations.
\newblock {\em Journal of Computational Physics}, 378:686--707, 2019.

\bibitem{shin2020convergence}
Yeonjong Shin, Jerome Darbon, and George~Em Karniadakis.
\newblock On the convergence and generalization of physics informed neural
  networks.
\newblock {\em arXiv e-prints}, pages arXiv--2004, 2020.

\bibitem{spagnolo1967sul}
Sergio Spagnolo.
\newblock Sul limite delle soluzioni di problemi di cauchy relativi
  all'equazione del calore.
\newblock {\em Annali della Scuola Normale Superiore di Pisa-Classe di
  Scienze}, 21(4):657--699, 1967.

\bibitem{spagnolo1976convergence}
Sergio Spagnolo.
\newblock Convergence in energy for elliptic operators.
\newblock In {\em Numerical Solution of Partial Differential Equations--III},
  pages 469--499. Elsevier, 1976.

\bibitem{wang2020and}
Sifan Wang, Xinling Yu, and Paris Perdikaris.
\newblock When and why pinns fail to train: A neural tangent kernel
  perspective.
\newblock {\em arXiv preprint arXiv:2007.14527}, 2020.

\bibitem{you2021data}
Huaiqian You, Yue Yu, Nathaniel Trask, Mamikon Gulian, and Marta D'Elia.
\newblock Data-driven learning of nonlocal physics from high-fidelity synthetic
  data.
\newblock {\em Computer Methods in Applied Mechanics and Engineering},
  374:113553, 2021.

\end{thebibliography}


\end{document}